\documentclass[11pt]{amsart}
\usepackage[margin=1in]{geometry}
\usepackage[T1]{fontenc}
\usepackage{lmodern}
\usepackage{amsmath,amssymb,amsthm,mathtools}
\usepackage{enumitem}
\usepackage{microtype}
\usepackage[hidelinks]{hyperref}
\usepackage{xcolor}

\newtheorem{theorem}{Theorem}[section]
\newtheorem{proposition}[theorem]{Proposition}
\newtheorem{lemma}[theorem]{Lemma}
\newtheorem{corollary}[theorem]{Corollary}
\newtheorem{claim}[theorem]{Claim}
\theoremstyle{definition}
\newtheorem{definition}[theorem]{Definition}
\newtheorem{remark}[theorem]{Remark}
\theoremstyle{plain}

\newcommand{\F}{\mathbb F}
\newcommand{\Aut}{\operatorname{Aut}}
\newcommand{\emb}{\operatorname{emb}}
\newcommand{\ind}{\operatorname{ind}}
\newcommand{\E}{\mathrm E}
\newcommand{\eps}{\varepsilon}

\newcommand{\alternativeheading}[2]{%
  \par\medskip
  \noindent{\bfseries Alternative \textnormal{(#1)}: #2.}\par
  \smallskip
}

\title{Explicit fractalizers: prime-order Paley graphs and other Cayley graphs}
\date{\today}
\author{Aldo Kiem}
\address{Zuse Institute Berlin, Technische Universität Berlin}
\email{kiem@campus.tu-berlin.de}
\thanks{Research is partially funded by the Deutsche Forschungsgemeinschaft
(DFG, German Research Foundation) under Germany’s Excellence Strategy—The
Berlin Mathematics Research Center MATH+
(EXC-2046/1, project ID: 390685689).}

\author{Fan Wei}
\address{Department of Mathematics, Duke University,
120 Science Drive, Durham, NC 27710, USA}
\email{fan.wei@duke.edu}
\thanks{Research supported by NSF grant DMS-2401414.}

\begin{document}

\begin{abstract}
The inducibility problem asks for the maximum number of induced
copies of a fixed graph among all graphs with a prescribed number
of vertices. Inducibility has been an active area of research in
extremal combinatorics, but determining all extremal graphs for
explicitly defined patterns remains challenging, particularly
when the description is required to hold at every host order.
A graph $H$ is called a \emph{fractalizer} if, for every positive
integer $n$, every $n$-vertex graph maximizing the number of
induced copies of $H$ is a balanced iterated blow-up of $H$,
obtained by recursively repeating the same pattern in parts
whose sizes differ by at most one. Previous probabilistic
results show that large random graphs and random abelian Cayley
graphs are fractalizers with probability tending to one,
establishing their abundance without directly providing
explicit families.

We prove that every sufficiently large prime-order Paley graph
is a fractalizer. Thus these classical arithmetic patterns
determine the exact recursive structure of every extremal host,
at every host order and without any algebraic assumptions on
the host. We also construct a second explicit infinite family
of nontrivial Cayley fractalizers, for which the fractalizer
property admits a simpler proof than in the Paley case.
Together, these results resolve the explicit-construction
question discussed at the 2025 American Institute of Mathematics
workshop ``Flag Algebras and Extremal Combinatorics.''
\end{abstract}

\maketitle

\setcounter{tocdepth}{1}
\tableofcontents

\section{Introduction}

Counting subgraphs in a host graph is one of the oldest and most studied problems in extremal combinatorics. While the classical theorems of Tur\'an and Erd\H{o}s-Stone govern the maximum edge density of graphs forbidding certain subgraphs, the question becomes notoriously intricate when counting induced subgraphs. For graphs, an induced copy must preserve both edges and nonedges, so adding edges to a host can destroy some copies while creating others. In 1975, Pippenger and Golumbic \cite{PG} introduced the inducibility problem: Given a $k$-vertex graph $H$ and an integer $n$, what is the maximum number of $k$-vertex subsets in an $n$-vertex graph that induce a copy of $H$? 

We study the stronger structural question of when every maximizer must be built by
repeating the same pattern recursively. Our main result establishes this phenomenon for sufficiently large prime-order Paley graphs:
a pattern defined by quadratic residues forces an exact recursive structure in extremal hosts, even though no algebraic structure is assumed of those hosts.

An embedding of a graph $H$ into a graph $G$ is an injective map $\vartheta:V(H)\to V(G)$ preserving both adjacency and nonadjacency.  We write $\emb(H,G)$ for the number of embeddings.  If $H$ has $k$ vertices, set
\begin{equation*}
 \emb(H,n):=\max_{|V(G)|=n}\emb(H,G),
 \qquad
 \emb(H):=\lim_{n\to\infty}\frac{\emb(H,n)}{n^k}.
\end{equation*}
Let $\ind(H,n)$ be the maximum number of $k$-subsets inducing a copy of $H$ in an $n$-vertex graph, and set
\begin{equation*}
 \ind(H):=\lim_{n\to\infty}\frac{\ind(H,n)}{\binom{n}{k}}.
\end{equation*}
Then $\emb(H)=|\Aut(H)|\ind(H)/k!$.  In particular, maximizing labeled embeddings and maximizing unlabeled induced copies are equivalent.

A fundamental source of lower bounds in inducibility is the {\it blow-up}
construction.
A blow-up of a graph $H$ is obtained by replacing every $x\in V(H)$ with a nonempty set $W_x$, placing all edges between $W_x$ and $W_y$ when $xy\in E(H)$ and no edges between them otherwise, while allowing arbitrary graphs inside the parts.  It is balanced if the part sizes differ by at most one.  

Pippenger and Golumbic~\cite{PG} observed that
this blow-up construction can be made self-similar. After forming the top-level
parts, one repeats the same pattern inside every part, and continues
recursively until a part contains fewer than $|V(H)|$ vertices. The
result is a \emph{balanced iterated blow-up} of $H$. 
More formally, a graph $\Gamma$ is a {balanced iterated blow-up} of $H$ if either $|V(\Gamma)|<|V(H)|$, or there is a partition $V(\Gamma)=\bigcup_{v\in V(H)}V_v$ such that the part sizes differ by at most one, the bipartite graph between $V_u$ and $V_v$ is complete or empty according as $uv\in E(H)$ or $uv\notin E(H)$, and each induced graph
$\Gamma[V_v]$ is itself a balanced iterated blow-up of $H$. Pippenger and Golumbic~\cite{PG} showed that, for every graph $H$ on $k \geq 2$ vertices, these graphs give the universal lower bound
\[
\operatorname{ind}(H)\geq \frac{k!}{k^k-k}.
\]

The lower bound raises a stronger structural question: when must every
extremal graph be a balanced iterated blow-up of the pattern itself?

Following~\cite{FSW}, we call $H$ a \emph{fractalizer} if, for every
positive integer $n$, every graph attaining $\emb(H,n)$ is a balanced
iterated blow-up of $H$.

\subsection{The scarcity of exact results}

Despite extensive work, inducibility is known exactly for only a
limited collection of classes of graphs; even the inducibility $\operatorname{ind}(P_4)$ of the four-vertex path
remains unknown. Many classical solved cases are complete multipartite
graphs or their complements; see, for example,
\cite{brown1994inducibility,schelp1998remark,hirst2014inducibility,
liu2023feasible}. Even in these structured cases, determining the optimal part sizes can require delicate analysis.

One general result concerns large blow-ups. Hatami, Hirst, and
Norine~\cite{hatami2014inducibility} proved that blow-ups of $H_0$ asymptotically optimize every sufficiently large balanced blow-up of $H_0$; Chen and Liu~\cite{CL26} recently proved that, for all sufficiently large host orders, every extremal host is a blow-up of $H_0$. 

The target extremal graphs also vary. Bodn\'ar and Pikhurko~\cite{bodnar2025some} computed the inducibility of the $4$-cycle with a pendant edge and showed that every almost extremal graph admits an almost balanced bipartition with $o(n^2)$ internal edges and a $5/6$-quasirandom cross-pair. In a complementary direction, Jain, Michelen, and Wei~\cite{jain2026binomial} proved that, for every graph $F$ on at least three vertices and every prescribed edge density $p\in(0,1)$, the binomial random graph $G(n,p)$ is not asymptotically optimal for the induced density of $F$. These results illustrate both the difficulty of determining the limiting inducibility and the variety of possible extremal structures. A fractalizer theorem asks for more: at every finite host order, every maximizer must be a balanced iterated blow-up of the pattern itself. Thus the conclusion is an exact recursive description at all orders, rather than merely a limiting density or a description valid for sufficiently large hosts.

\subsection{Fractalizers: from flag algebra and probabilistic existence to explicit structure}

The distinction between limiting density and exact finite-order
structure is particularly important in the fractalizer problem.
Previous work has approached this problem in two different ways: by analyzing individual small patterns using flag algebras, stability, and finite verification, and by proving that large random patterns satisfy the fractalizer property with high probability. The first route reveals the possible obstruction from exceptional finite hosts; the second establishes abundance without directly specifying an explicit family.

For the first route, Lidick\'y, Mattes, and
Pfender~\cite{lidicky2023c} determined the extremal graphs for $C_5$ for every $n$. Balanced iterated blow-ups are extremal, but at $n=8$ the M\"obius ladder and its complement give additional extremal constructions. Thus $C_5$ is only an ``almost fractalizer.''  Blumenthal and Phillips~\cite{blumenthal2021inducibility} found an
analogous obstruction for the net graph on six vertices. On eight vertices, the graph obtained from $K_4$ by attaching one pendant vertex to each vertex of $K_4$ has the same number of induced nets as a balanced iterated blow-up of the net. Hence the net is not a fractalizer. Nevertheless, they proved that for an infinite family of $n$, the balanced iterated blow-up
is the unique extremal graph on $n$ vertices. Their argument combines flag-algebra estimates with stability, detailed counting, and numerical optimization. These examples demonstrate how the all-$n$ fractalizer property can be
destroyed by exceptional configurations at small orders, even when the extremal graphs have the expected recursive structure in substantial ranges of larger orders.

A different picture emerges for large random graphs. Fox, Huang, and Lee~\cite{FHL} proved that asymptotically almost surely a random graph $H$ is a fractalizer: simultaneously for every positive integer $n$, every $n$-vertex extremal graph is a balanced iterated blow-up of $H$. Independently, Yuster~\cite{yuster2019exact} proved the same extremal description throughout the range $n\le2^{\sqrt{|V(H)|}}$. Fox, Sauermann, and Wei~\cite{FSW} subsequently proved an analogous all-$n$ theorem for random abelian Cayley graphs and for graphs obtained from them by deleting a small number of vertices.

These results establish that fractalizers are abundant, but their probabilistic verification does not directly yield an explicit family. For a graph given by a deterministic rule, estimates for degrees and subgraph counts provide only part of the required information. The reconstruction arguments also require uniform control of every induced embedding of every sufficiently large vertex-deleted subgraph. The challenge is therefore to establish deterministic structural properties strong enough to force the recursive extremal description at every host order.

\subsection{Main results}

The problem of finding a nontrivial explicit fractalizer was discussed
at the 2025 American Institute of Mathematics workshop ``Flag Algebras
and Extremal Combinatorics''; see~\cite[p.~4]{AIM25}. Here nontrivial
means neither complete nor empty. We resolve this problem with the
following two theorems.

Our first main theorem establishes the fractalizer property for a classical explicit family. For a prime $p\equiv1\pmod4$, the
\emph{Paley graph} $P_p$ has as its vertex set the field $\F_p$ of integers modulo $p$; distinct $x,y$ are adjacent precisely when $x-y$
is a nonzero square in $\F_p$.

\begin{theorem}\label{thm:main}
There exists $p_0$ such that, for every prime $p\ge p_0$ with
$p\equiv1\pmod4$, the Paley graph $P_p$ is a fractalizer. That is, for
every positive integer $n$, every $n$-vertex graph maximizing the number
of induced copies of $P_p$ is a balanced iterated blow-up of $P_p$.

Moreover, its embedding-inducibility and inducibility are
\begin{equation}\label{eq:inducibility-formulas}
 \emb(P_p)=\frac{p(p-1)/2}{p^p-p}
 \text{ and }
 \ind(P_p)=\frac{p!}{p^p-p}.
\end{equation}
\end{theorem}

The threshold in Theorem~\ref{thm:main} concerns the order of the pattern, not the host: once $p$ is sufficiently large, the conclusion holds for every $n$. The optimization is over all $n$-vertex graphs,
with no symmetry or algebraic assumptions. Nevertheless, every maximizer admits the recursive decomposition prescribed by $P_p$.
The theorem therefore determines exact finite extremal structure, not only a limiting density.

We also give a second answer to the explicit-construction problem. The family below is less canonical than the Paley graphs, but its fractalizer property admits a more direct proof.

\begin{theorem}\label{thm:main_part2}
For every sufficiently large prime $k\equiv1\pmod4$, the construction in Definition~\ref{def:Hconstruction} gives a Cayley graph $H_k$ on $\mathbb Z_k$ that is neither complete nor empty and is a fractalizer. In particular, for every positive integer $n$, every $n$-vertex graph maximizing the number of induced copies of $H_k$ is a balanced iterated blow-up of $H_k$.
\end{theorem}

\subsection{Proof ideas and main ingredients}

The two main theorems require rather different ideas.  The work of Fox, Sauermann, and Wei~\cite{FSW} gives a general route from certain structural properties of an abelian Cayley graph to the fractalizer conclusion.  In the full-graph case used here, these include primeness,
neighborhood separation, and a strong reconstruction property (Definition \ref{def:reasonable_part_two}). A graph is \emph{prime} if it has no proper vertex set of size at least two to which every outside vertex is adjacent either everywhere or nowhere.
Neighborhood separation requires every pair of vertices to be
distinguished by a positive proportion of the other vertices, each adjacent to exactly one of the pair. 
The principal requirement is a strong
\emph{deletion-robust reconstruction property}: every induced
embedding of a sufficiently large induced subgraph into the
original graph must extend to a translation or reflection
$x\mapsto\pm x+b$. In the full-graph specialization used here,
the required vertex deletion tolerance has order
\begin{equation}
    \Theta((\log k)^{1/5}k^{4/5} )\label{eq:threshold}.
\end{equation}
where $k$ is the number of vertices in the graph.

For the family in Theorem~\ref{thm:main_part2}, the main difficulty is to realize these structural properties in a completely explicit graph.
Fox, Sauermann, and Wei verify these properties probabilistically for random abelian Cayley graphs. Here we need to develop an explicit non-random  construction. Thus the difficulty is that we need to develop an explicit construction whose cyclic coordinates remain recoverable after vertex deletion. Two requirements must be met simultaneously. Selected cyclic differences must be recognizable from the graph itself, and the relations they encode must force a partial embedding to follow a
single global symmetry. We combine simultaneous quadratic-residue conditions with a multiscale additive design to meet both requirements. The additive design prescribes translate overlaps, which character-sum estimates convert into codegree values, where the codegree of a pair is its number of common neighbors. Distinct prescribed values are separated by a positive multiple of $k$. A bounded modification of the connection set resolves the remaining equal-codegree ambiguities through adjacency, while changing every
codegree by only a bounded amount. Since the permitted number of deletions is $o(k)$, the resulting recognition is uniform over all allowed deleted sets.

Recognizing the marked differences is not yet enough to reconstruct the embedding. The unit difference determines a translation or reflection on each surviving interval, but different intervals may initially carry different affine maps. The larger bridge differences and their diagonal companions force these maps to agree wherever the corresponding configurations survive. A small-diameter auxiliary
Cayley graph and a diameter--boundary argument control the
configurations destroyed by deletion and show that one affine rule holds on almost all vertices. Neighborhood separation then extends that rule to the entire domain. This proves
Proposition~\ref{prop:main}, giving reconstruction at the deletion scale required by~\cite{FSW}; the accompanying primeness and neighborhood-separation estimates complete the proof of Theorem~\ref{thm:main_part2}.

The Paley graph presents an additional difficulty. Its algebraic structure is already given, so there is no freedom to build additional rigidity into the graph. The required input is not merely a classification of its automorphisms, but an extension theorem for partial embeddings: every induced embedding defined on an arbitrary sufficiently large subset must be the restriction of a global automorphism. The classical description of the full automorphism group does not by itself supply this extension property. 

A key observation to achieve some partial rigidity is by encoding an induced embedding $f:P_p[X]\to P_p$ as the point set $U_f=\{(x,f(x)):x\in X\}\subseteq\F_p^2$. Preservation of adjacency and nonadjacency forces all its secant slopes to be nonzero squares. This connects the partial-embedding problem to finite-geometric completion and direction theorems~\cite{Szonyi,Szi,FST,Redei,Somlai}. Using this connection, we prove that every such $f$ has the form
$f(x)=ax+b$ on $X$, with $a$ a nonzero square, whenever
$|X|\ge p-\sqrt p/3$; see Lemma~\ref{lem:partial-affine-rigidity}. The conclusion holds uniformly over the choice of the missing vertices and the partial embedding. Taking $X=\F_p$ also recovers the classical automorphism-group description following from Carlitz's theorem~\cite{Carlitz}.

Two obstacles prevent a direct application of~\cite{FSW}. First, the allowed symmetries must be enlarged: Paley graphs admit square multipliers other than $\pm1$ when $p>5$, so the translations-and-reflections condition fails even on the full graph.
Second, our partial-extension theorem allows at most $\sqrt p/3$ vertex deletions, substantially fewer than the deletion tolerance \eqref{eq:threshold} required by~\cite{FSW}. We do not establish extension to the larger affine group at that stronger scale. We therefore prove a new fractalizer criterion, Proposition~\ref{prop:full-fsw}, accommodating both the enlarged automorphism group and the weaker quantitative rigidity condition. It applies without a Cayley assumption and permits $2\le|\Aut(H)|\le k^2$, where $k=|V(H)|$.

The main new extremal ingredient is a sharper estimate for embeddings that send many pattern vertices into a small exceptional subset of the host.
A \emph{signature} is a set whose adjacency patterns distinguish all remaining pattern vertices. Once its image is fixed, those vertices have pairwise disjoint candidate sets.  Their intersections with the exceptional region therefore have
total size at most the size of that region. The aggregate estimate in~\cite{FSW} bounds each choice of exceptional pattern vertices by the same worst-case estimate. We instead retain the individual sizes of the intersections of the candidate sets with the exceptional region and sum their weighted contributions first. Maclaurin's inequality controls this sum without the additional loss from estimating each choice separately. At the parameters required by square-root-scale rigidity, our estimate gives exponential decay whereas the earlier aggregate bound does not; see Lemma~\ref{lem:weighted-hitting}.

The structural argument retains the signature and cleaning
architecture of~\cite{FHL,FSW}. The new estimate replaces the two aggregate-counting steps whose earlier bounds are ineffective at the new rigidity scale. The fixed-symmetry estimates are formulated for arbitrary automorphisms and summed over a family of size at most $k^2$, and the remaining quantitative conditions are verified at the new parameter scales. These ingredients show that almost every embedding approximately follows an automorphism.
Extremality then forces the prospective classes to partition the host and every interpart adjacency to agree exactly with the pattern.

The full-pattern setting simplifies the final step. In an exact
blow-up of a prime graph $H$, every embedding of $H$ either lies entirely inside one part or uses all top-level parts once, according to an automorphism. For parts $W_1,\dots,W_k$, this gives
\[
\emb(H,G) = |\Aut(H)|\prod_{i=1}^k |W_i|
+
\sum_{i=1}^k\emb(H,G[W_i]).
\]
A transfer argument controlling both the product and the internal terms forces the parts to be balanced. Induction inside the parts then gives the iterated blow-up structure and completes the proof of Proposition~\ref{prop:full-fsw}. Applying the criterion with the finite-geometric rigidity, neighborhood separation, and primeness of the Paley graph proves Theorem~\ref{thm:main}.

\subsection{Organization}

Part~I proves Theorem~\ref{thm:main_part2}, developing the robust reconstruction argument and the explicit Cayley graph construction.
Part~II states the new full-graph criterion, verifies its hypotheses for Paley graphs, and proves it using the new embedding estimate, local stability, and balancing. The appendices contain the parameter verification and the counting and stability statements inherited from~\cite{FSW}.

\part{Explicit construction: proof of Theorem~\ref{thm:main_part2}}

\section{Structural reduction and robust reconstruction}
\label{sec:reconstruction}

\subsection{The Fox-Sauermann-Wei criterion}

We begin by stating the structural result of Fox, Sauermann,
and Wei that connects the properties of a fixed abelian Cayley graph to
the inducibility problem. This is the external extremal input to our
proof of Theorem \ref{thm:main_part2}. It reduces the extremal problem to three structural
properties. Verifying them for random Cayley graphs is probabilistic \cite{FSW};
the central task here is to realize them deterministically. In
particular, the near-spanning reconstruction property requires a new
mechanism that occupies the main part of the proof.

For an abelian group $G$ and a set $S\subseteq G$ satisfying
$S=-S$ and $0\notin S$, the Cayley graph
$\operatorname{Cay}(G,S)$ has vertex set $G$, with distinct vertices
$x,y$ adjacent precisely when
\[
y-x\in S.
\]
For distinct vertices $x,y$ of a graph $H$, write
$a_H(x,y)=1$ if $xy\in E(H)$ and $a_H(x,y)=0$ otherwise.
An \emph{induced embedding} $f:H[X]\to H$ is an injective map satisfying
\[
a_H(x,y)=a_H(f(x),f(y))
\]
for all distinct $x,y\in X$.

\begin{definition}[Fox--Sauermann--Wei~\cite{FSW}]
\label{def:reasonable_part_two}
Let $H$ be a Cayley graph of an abelian group $G$ of order $k$.
For $0<q<1/2$ and $0<\delta<1$, we call $H$
\emph{$(q,\delta)$-reasonable} if the following hold.
\begin{enumerate}
    \item The graph $H$ is prime: there is no set
    $U\subseteq V(H)$ with $2\le |U|<|V(H)|$ such that every vertex
    outside $U$ is either adjacent to every vertex of $U$ or adjacent
    to no vertex of $U$.

    \item For every two distinct vertices $v,w\in V(H)$, at least $qk$ vertices in $V(H)\setminus\{v,w\}$ are adjacent to
exactly one of $v$ and $w$.

    \item If $X\subseteq V(H)$ satisfies $|X|\ge(1-\delta)k$ and
    $f:H[X]\to H$ is an induced embedding, then there are
    $\sigma\in\{\pm1\}$ and $b\in G$ such that
    \[
    f(x)=\sigma x+b
    \qquad\text{for every }x\in X.
    \]
\end{enumerate}
\end{definition}
Here $x\mapsto x+b$ is a translation and $x\mapsto -x+b$ is a
reflection. Since $H$ is an undirected Cayley graph, all these maps are
automorphisms. In particular, condition~(3), applied with $X=V(H)$,
implies that $H$ has no automorphisms other than its translations and
reflections.

The following specialization of the theorem of Fox, Sauermann, and
Wei is obtained by taking the target graph in their general result to be
the entire ambient Cayley graph. It converts the three properties above
into the exact all-$n$ extremal conclusion.
Throughout the paper, all logarithms are base $e$.

\begin{theorem}\cite[Theorem~2.7]{FSW}\label{thm:FSW}
Let $k \ge 10^{200}$ and let $ {H}$ be a Cayley graph of an abelian group of size $ {k}$. Let $H$ be $(q, \delta)$-reasonable and $q \ge 10^4 (\ln k)^{6/5} k^{-1/5}$ and $\delta \ge 10^3 (\ln k)^{1/5} k^{-1/5}$. Then for all $n$, every $n$-vertex graph $\Gamma$
satisfying $\operatorname{emb}(H,\Gamma)=\operatorname{emb}(H,n)$
is a balanced iterated blow-up of $H$.
\end{theorem}

We apply Theorem~\ref{thm:FSW} with
\[
\delta=10^3(\log k)^{1/5}k^{-1/5}.
\]
For the graph constructed in Section~\ref{sec:H_construction},
neighborhood separation, condition~(2), will follow from
Lemma~\ref{lem:q}, while primeness, condition~(1), will follow from
Lemma~\ref{lem:prime}. The most challenging task is condition~(3), which is the
content of the following proposition.

\begin{proposition}[Robust reconstruction]\label{prop:main}
Let $k\equiv1\pmod4$ be a sufficiently large prime, and let $H$ be the Cayley
graph constructed in Definition~\ref{def:Hconstruction}. If
$X\subseteq\mathbb Z_k$ satisfies $|X|\ge(1-\delta)k$
and $f:H[X]\to H$ is an induced embedding, then there are
$\sigma\in\{\pm1\}$ and $b\in\mathbb Z_k$ such that
\[
f(x)=\sigma x+b
\qquad\text{for every }x\in X.
\]
Thus $f$ is the restriction to $X$ of a translation or reflection of
$\mathbb Z_k$.
\end{proposition}

\subsection{Robust reconstruction from marked differences}

The proof of Proposition~\ref{prop:main} has a graph-specific part and a
graph-independent part. The graph-specific part, proved in
Sections~\ref{sec:H_construction}--\ref{sec:final}, shows that the
adjacency and codegree of a pair identify a finite collection of marked
cyclic differences. It follows that every near-spanning induced
embedding preserves those differences, up to sign.

In this section we prove the graph-independent local-to-global step:
once the marked differences are preserved, the embedding must agree
with a single translation or reflection on almost all vertices. A final
neighborhood-separation argument then extends the same formula to every
vertex in its domain.

For $d\in\mathbb Z_k$, we call $\{\pm d\}$ the \emph{unoriented
difference class} of $d$. A map $g:X\to\mathbb Z_k$ is said to
\emph{preserve the difference $d$} if, whenever $x,x+d\in X$,
\[
g(x+d)-g(x)\in\{\pm d\}.
\]

For $z\in\mathbb Z_k$, define its cyclic norm by
$\|z\|_k
=
\min_{m\in\mathbb Z}|\widetilde z-mk|$,
where $\widetilde z$ is any integer representative of $z$. For
$x,y\in\mathbb Z_k$, define their \emph{cyclic distance} by
$d_{\rm cyc}(x,y)=\|x-y\|_k.$
Equivalently, $d_{\rm cyc}$ is the graph distance in the standard cycle
$\operatorname{Cay}(\mathbb Z_k,\{\pm1\})$.
It should not be confused with graph distance in the constructed graph
$H$ or in the auxiliary Cayley graph $\Theta$.

Because $k$ is odd, specifying a nonzero cyclic distance represented
by $d\in\mathbb Z_k$ is equivalent to specifying the difference class
$\{\pm d\}$. Accordingly, for
a symmetric set $\Delta=-\Delta\subseteq\mathbb Z_k$, we say that a map
$g:X\to\mathbb Z_k$ \emph{preserves the cyclic distances represented
by $\Delta$} if
\[
g(x+d)-g(x)\in\{\pm d\}
\]
whenever $d\in\Delta$ and $x,x+d\in X$.

\begin{lemma}\label{lem:g-rigid}
Let $T_1,\dots,T_6$ be distinct integers satisfying $3<T_i<k/4.$
Define
$\Upsilon=\{\pm1\}\cup\bigcup_{i=1}^6\{\pm T_i\},$
and 
$\Delta
 =
\Upsilon\cup\bigcup_{i=1}^6
\{\pm(T_i+1),\pm(T_i-1)\}.
$
Let
$\Theta=\operatorname{Cay}(\mathbb Z_k,\Upsilon)$
have diameter $D$, and assume that
$D\delta<1/210.$

Let $X\subseteq\mathbb Z_k$ satisfy $|X|\ge(1-\delta)k,$
and let $g:X\to\mathbb Z_k$ be injective. Suppose that $g$ preserves the cyclic distances represented by $\Delta$, i.e., whenever
$x,x+d\in X$ and $d\in\Delta$,
\[
g(x+d)-g(x)\in\{\pm d\}.
\]
Then there exist a giant set
$C_{\rm giant}\subseteq X$ with 
$|C_{\rm giant}|\ge(1-210D\delta)k
$ on which $g$ acts as an isometry, i.e., there exists $\sigma\in\{\pm1\}$ and  $b\in\mathbb Z_k$ such that 
$g(x)=\sigma x+b$
for every $x\in C_{\rm giant}.$
\end{lemma}

\begin{proof}

Consider the connected components of the subgraph of the cycle
$\operatorname{Cay}(\mathbb Z_k,\{\pm1\})$
induced by $X$. If $X\ne\mathbb Z_k$, each component is a path (maximal interval) and can
be written uniquely in the form $I=\{x_0,x_0+1,\dots,x_0+\ell-1\}$,
where $x_0-1\notin X$. If $X=\mathbb Z_k$, there is a single component, namely the whole
cycle; in this case choose any $x_0\in\mathbb Z_k$.

We claim that on every component containing at least two vertices,
there exist $\sigma_I\in\{\pm1\}$ and $b_I\in\mathbb Z_k$ such that
\[
g(x)=\sigma_Ix+b_I
\qquad\text{for every }x\in I.
\]
In other words, vertices in the same component share the same  affine map.  
Indeed, as $1 \in \Delta$, preservation of the marked difference $1$ gives
$g(x_0+1)-g(x_0)=:\sigma_I\in\{\pm1\}.$
For some $j \geq 1$, 
suppose that
$g(x_0+j)-g(x_0+j-1)=\sigma_I.$
Preservation of the difference $1$ gives
$g(x_0+j+1)-g(x_0+j)\in\{\pm1\}.$
The value $-\sigma_I$ is impossible, since it would give
$g(x_0+j+1)=g(x_0+j-1)$,
contrary to the injectivity of $g$. Hence every successive difference
equals $\sigma_I$, and therefore by induction,
\[
g(x_0+j)=g(x_0)+j\sigma_I.
\]
Thus $g(x)=\sigma_Ix+b_I$ throughout $I$, where
$b_I=g(x_0)-\sigma_Ix_0$.

If $X=\mathbb Z_k$, this affine formula holds on all of $X$, and the lemma follows. Henceforth assume that $X\ne\mathbb Z_k$.

Let 
$I_1,\dots,I_m$ be the components.  
Discard every component consisting of one element only, and let $X_0$ be
the union of the remaining components. Let
$M_0=\mathbb Z_k\setminus X_0$ be the missing set. 
There are at most $|\mathbb Z_k\setminus X|$ singleton intervals as this number bounds $m$. Therefore,
\begin{equation}\label{eq:M0bound}
|M_0|\le |\mathbb Z_k\setminus X| + |X \setminus X_0| \leq 2\delta k.
\end{equation}

We just showed that on each interval $I_i$ the function $g$ acts strictly as an isometry $g(x) = \sigma_i x + b_i$ where $\sigma_i \in \{\pm 1\}$. 
Now, we try to synchronize these local isometries into a global isometry, at least on a large component $C_{\textrm{giant}}$.
To this end, we define a “strong link" with respect to a bridge $T_i$ as a {\it $T_i$-square} of 4-vertices in $X$: \[\{x, x+1, x+T_i, x+T_i+1\}.\] 

\begin{claim}\label{claim:merge}
    If $I_A$ and $I_B$ are connected by a strong link, i.e., there is a $T_i$-square intersecting both $I_A$ and $I_B$, then $\sigma_A=\sigma_B$ and $b_A=b_B$.
\end{claim}
\begin{proof}[Proof of Claim:]
Let $I_A$ and $I_B$ be connected by a strong link. We have already shown the claim when $A = B$, so we assume $A\ne B$.
Then necessarily we have without loss of generality that $x, x+1$ are in $I_A$ and $x+T_i, x+T_i+1$ in $I_B$, as $I_A, I_B$ are the components and thus maximal intervals. 

Recall that $g(x) = \sigma_A x + b_A$, $g(x+1) = \sigma_A (x+1) + b_A$, and 
$g(x + T_i) = \sigma_B (x+T_i) + b_B$, $g(x+T_i + 1) = \sigma_B(x+T_i+1) + b_B$.
From the definition of the distance preservation property of $g$, we have 
$\{\pm T_i\} \ni g(x+T_i) - g(x) = (\sigma_B (x+T_i) + b_B) - (\sigma_A x + b_A)$ and so
\begin{equation}
    (\sigma_B - \sigma_A)x + \sigma_B T_i + b_B - b_A \in \{\pm T_i\}. \label{eq:difference}
\end{equation}
Similarly, 
$\{\pm T_i\} \ni g(x+T_i+1) - g(x+1) = (\sigma_B (x+T_i+1) + b_B) - (\sigma_A (x+1) + b_A)$ and therefore
\[(\sigma_B - \sigma_A)x + (\sigma_B - \sigma_A) + \sigma_B T_i + b_B - b_A \in \{\pm T_i\}.\] 
Subtracting these two inequalities gives us
\[ \sigma_B - \sigma_A \in \{-2T_i, 0, 2T_i\}.\]
Because $\sigma_A,\sigma_B\in\{\pm1\}$ and $3<T_i<k/4$, the congruence 
$\sigma_B-\sigma_A\in\{-2T_i,0,2T_i\}$
forces $\sigma_B=\sigma_A$.
 This proves the first point in the claim.

Now we will prove $b_A = b_B$.
Plugging in $\sigma_A = \sigma_B$ into (\ref{eq:difference}), 
we have 
\begin{equation}\sigma_B T_i + b_B - b_A \in \{\pm T_i\}. \label{eq:b_shift}
\end{equation}
We already proved that $\sigma_A=\sigma_B=:\sigma$ and $\sigma \in \{\pm 1\}$. 
Substituting this into (\ref{eq:b_shift}) gives $
b_B-b_A\in\{0,-2\sigma T_i\}.$
If $b_B - b_A = 0$ then we are done.

Suppose instead that
\[
b_B-b_A=-2\sigma T_i.
\]
Since a diagonal in the $T_i$-square has difference $T_i+1\in\Delta$, its image difference
must belong to $\{\pm(T_i+1)\}$. However,
\begin{align*}
g(x+T_i+1)-g(x)
=\sigma(x+T_i+1)+b_B-(\sigma x+b_A)
=\sigma(T_i+1)+b_B-b_A
=\sigma(1-T_i),
\end{align*}
and $\sigma(1-T_i)\notin\{\pm(T_i+1)\}.$
This contradiction shows that
$b_B-b_A=0$.
\end{proof}

We showed that any two intervals $I_A,I_B$ connected by a strong link forces $g$ to be the exact same affine map on both $I_A$ and $I_B$. Now we show that there is a large component of intervals $I_A$ that is connected by strong links.

Call an edge of the auxiliary graph $\Theta = \operatorname{Cay}(\mathbb{Z}_k, \Upsilon)$ ``bad" if either

\begin{enumerate}
    \item it is an edge of difference $1$ incident with $M_0$; or
    \item it is an edge $\{u,u+T_i\}$ for which at least one of
    $\{u, u+1, u+T_i, u+T_i+1\}$
    belongs to the missing set $M_0$.
\end{enumerate}

It remains to bound the number of bad edges. Every vertex $z\in M_0$ is
incident with at most two unit edges. For a fixed $i$, the vertex $z$ can occur in
at most four $T_i$-squares, corresponding to
$u=z,u=z-1,u=z-T_i,
u=z-T_i-1.$
Since there are six choices of indices $i$, 
\begin{equation}\label{eq:Ebad}
|E_{\rm bad}|
 \le (2+4\cdot6)|M_0|
 =26|M_0|
 \le52\delta k = o(k).
\end{equation}

Group the intervals according to their affine parameters
$(\sigma_A,b_A)$. For each pair $(\sigma,b)$, let $C_{\sigma,b} \subset X_0$ be the
union of the vertices in all intervals carrying the affine map
$x\mapsto\sigma x+b$, and we call  $C_{\sigma,b} \subset X_0$ an {\it affine class}.  We claim that every edge of $\Theta$ leaving an affine class is bad. Indeed,
if a unit edge has both endpoints in $X_0$, then its endpoints belong to
the same maximal interval and hence to the same affine class. Similarly,
if an edge $\{u,u+T_i\}$ is not bad, then all four vertices
$\{u, u+1, u+T_i, u+T_i+1\}$
belong to $X_0$. They form a strong link, so by Claim \ref{claim:merge} the intervals containing
$u$ and $u+T_i$ have the same affine parameters. Therefore the edge
cannot leave its affine class. We will combine this observation with
the later isoperimetric estimate for $\Theta$ (Claim \ref{claim:iso}).

We call an affine class $C_{\sigma,b}$ {\it minor} if
$|C_{\sigma,b}|\le k/2$. 
We will now show that there is a giant component. We prove this by utilizing a standard diameter-boundary estimate, whose short proof is
included for completeness. For $C\subseteq V(\Theta)$, define
\[
\partial_\Theta C
 =
\bigl\{\{x,y\}\in E(\Theta):|\{x,y\}\cap C|=1\bigr\}.
\]
Thus $\partial_\Theta C$ is the set of edges of $\Theta$ with exactly one
endpoint in $C$.

\begin{claim}\label{claim:iso}
For every $C\subseteq V(\Theta)$ with $|C|\le k/2$,
    \begin{equation}
|\partial_\Theta C|
\ge
\frac{|C|}{2D}.
 \label{eq:edgeoundary}
\end{equation}
\end{claim}

\begin{proof}[Proof of Claim]

For $z\in\mathbb Z_k$, we can write $
z=s_1+\cdots+s_t$ for some $t\le D$ and  $s_j\in\Upsilon,$ as $D$ is the diameter of the auxilliary graph $\Theta = \operatorname{Cay}(\mathbb{Z}_k, \Upsilon)$.  Write $r_j=s_1+\cdots+s_j$, with $r_0=0$. Then
$
|C\setminus(C+z)|
\le
\sum_{j=1}^t
|(C+r_{j-1})\setminus(C+r_j)|
=
\sum_{j=1}^t|C\setminus(C+s_j)|
\le D|\partial_\Theta  C|.
$
To see why the last inequality holds, notice for every $s\in\Upsilon$, each $x\in C\setminus(C+s)$ gives the boundary
edge $\{x-s,x\}$, and these edges are distinct as $x$ varies. Hence 
$|C\setminus (C+s)|\le|\partial_\Theta C|$. 

Then, we sum up this inequality over all $z \in \mathbb{Z}_k$. The left hand side is $|C|(k - |C|)$ because it is summing over $c \in C$, the number of elements $z$ not in $c - C$ which is $k - |C|$. Thus \[|C|(k - |C|) \leq k \cdot D \cdot |\partial_\Theta  C|.\] Since $|C| \leq k/2$, it follows that $|\partial_\Theta  C| \geq {|C|}/({2D})$ which proves (\ref{eq:edgeoundary}). 
\end{proof}

Applying Claim~\ref{claim:iso} to each minor affine class and using
\eqref{eq:Ebad}, and the observation that every edge leaving an affine class is bad, we obtain
\begin{equation}
\sum_{C_{\sigma,b}\ {\rm minor}}|C_{\sigma,b}|
\le
2D\sum_{C_{\sigma,b}\ {\rm minor}}|\partial_\Theta C_{\sigma,b}|
\le
4D|E_{\rm bad}|
\le
208D\delta k.
\label{eq:minorsum}
\end{equation}
Here an extra factor $2$ in the second inequality appears because a bad edge may
belong to the boundaries of two different affine classes.

There can be at most one affine class of size greater than $k/2$. If no
such class existed, then \eqref{eq:minorsum} would give
$|X_0|\le208D\delta k.$
But
$
|X_0|\ge k-|M_0|\ge(1-2\delta)k,
$
contradicting $208D\delta < (1-2\delta)$ by our assumption on $D, \delta$.
Therefore there is a unique giant affine class $C_{\rm giant}$. Thus by 
(\ref{eq:M0bound}) and (\ref{eq:minorsum}), we have
$|C_{\rm giant}| \ge k-2\delta k-208D\delta k
\ge (1-210D\delta)k$, and 
a single affine map
$g(x)=\sigma x+b$ holds on $C_{\rm giant}$.
\end{proof}

We next show that under assumptions on the graph $H$, the
affine formula valid on $C_{\rm giant}$ extends to every vertex of $X$.

\begin{lemma}\label{lem:outlier}
Let $H=\operatorname{Cay}(\mathbb Z_k, S)$ be an abelian Cayley graph on $\mathbb Z_k$. Let $X\subseteq \mathbb Z_k$ with $|X| \geq (1-\delta)k$ and let $f:H[X]\to H$ be an induced embedding. Let $C_{\textrm{giant}} \subseteq X$ be a subset such that restricted to $C_{\textrm{giant}}$, the embedding is $f(x) = \sigma x + b$ for some $\sigma \in \{ \pm 1\}$ and some $b \in \mathbb{Z}_k$. 

If the common neighborhood in $H$ of any two distinct $u,v\in V(H)$ is strictly smaller than $|S| - \frac{1}{2}(k - |C_{\textrm{giant}}|)$, then $f = \sigma x + b$ on all of $X$.  
\end{lemma}
\begin{proof}
It suffices to prove that the affine formula valid on
$C_{\rm giant}$ also holds on $X\setminus C_{\rm giant}$.

Fix $z\in X\setminus C_{\rm giant}$ and write $w=f(z)$. Since
$f$ is an induced embedding, every $x\in C_{\rm giant}$ has the
same adjacency to $z$ as $f(x)$ has to $w$. Hence
\[
x-z\in S
\iff
(\sigma x+b)-w\in S
\iff
(x-z)-\tau\in S,
\]
where  $\tau=\sigma w-\sigma b-z.$
Here we used $S=-S$ and $\sigma\in\{\pm1\}$.

If $\tau=0$, then $w=\sigma z+b$, as required. Suppose therefore
that $\tau\ne0$. Put
$Y=C_{\rm giant}-z.$
For every $y\in Y$, the preceding equivalence gives
$y\in S
\iff
y-\tau\in S
\iff
y\in S+\tau$.
Thus
$Y\cap\bigl(S\triangle(S+\tau)\bigr)=\varnothing,$
and therefore
\begin{equation}
 |S\triangle(S+\tau)|
\le
k-|Y|
=
k-|C_{\rm giant}|. \label{eq:S}   
\end{equation}

On the other hand, since $\tau\ne0$, the vertices $0$ and $\tau$
are distinct. Their common neighborhood has size
$|S\cap(S+\tau)|$.
By the hypothesis of the lemma,
$|S\cap(S+\tau)|
<
|S|-\frac12\bigl(k-|C_{\rm giant}|\bigr)$.
Consequently,
$|S\triangle(S+\tau)|
=2|S|-2|S\cap(S+\tau)|>
k-|C_{\rm giant}|$,
contradicting (\ref{eq:S}). Hence $\tau=0$, and so
$f(z)=\sigma z+b$. Since $z$ was arbitrary, the same formula holds
on all of $X$.

\end{proof}

\section{Construction of the Cayley graph $H$}
\label{sec:H_construction}

We construct $H$ by working backwards from the results of
Section~\ref{sec:reconstruction}. The heavy lifting in that section
reduces the robust reconstruction requirement in
Definition~\ref{def:reasonable_part_two}\textup{(3)} to a concrete target. We
must construct a cyclic Cayley graph for which every near-spanning
induced embedding preserves a suitable finite set $\Delta$ of cyclic
differences. Once this preservation is known, the small-diameter and
consistency arguments of Section~\ref{sec:reconstruction} force the
embedding to agree with a single translation or reflection. A final
neighborhood-separation argument then recovers any remaining exceptional
vertices.

Section~\ref{sec:reconstruction} does not yet imply how to make
an induced embedding preserve the differences in $\Delta$.
The first
main idea of the construction is to encode these differences by
graph-theoretic statistics. For a cyclic Cayley graph
$H=\operatorname{Cay}(\mathbb Z_k,S)$, let
\[
c(d)=|N_H(0)\cap N_H(d)|
\]
denote the codegree associated with two vertices in $H$ of cyclic difference $d$. We design $H$
so that every marked difference class $\{\pm d\}$ can be identified
from the pair
$\bigl(\mathbf 1_S(d),c(d)\bigr)$,
consisting of whether two vertices at difference $d$ are adjacent and
how many common neighbors they have. Thus the labels of the marked
differences are recorded intrinsically in the graph and they are also separated from unmarked differences. Therefore, for any induced embedding  $f:H[X]\to H$ where $X$ is large, any two vertices $x,y$ in $X$ of cyclic distance $d \in \Delta$, whether $x,y$ are adjacent and their common neighbors in $X$ is close to  $\bigl(\mathbf 1_S(d),c(d)\bigr)$, and the same applies to their images. Therefore if $c(d)$ for $d \in \Delta$ are sufficiently separated, they have to be mapped to two vertices of cyclic distance $d \in \Delta$ as well.

Given this idea, the next challenge is how to construct such
well-separated levels $(1_S(d), c(d))$ simultaneously for all marked differences $d\in \Delta$, while
preventing any of the other $k-O(1)$ differences from acquiring the
same signature.

 Our idea is to use a Paley-type connection set defined by
several simultaneous quadratic-residue conditions. The resulting
character-sum estimates relate $c(d)$ to the overlap $|A_0\cap(A_0-d)|$ of a finite auxiliary set $A_0$ with its translate. Thus the required
graph construction is reduced to a finite additive problem: build
$A_0$ so that the marked differences have controlled large overlaps,
whereas every unmarked difference has only a small overlap.

We now formalize this construction in two stages. First, we define a
Cayley graph from an arbitrary additive template. We then give an
explicit template and prove all of its properties that will be used in
the remainder of part I.

\begin{definition}[The Cayley graph associated with an additive template]
\label{def:additivetemplate}
Let $k\equiv1\pmod4$ be prime, and let
$G_1,\dots,G_6\subseteq G_0\subseteq\mathbb Z_k$
and 
$T_1,\dots,T_6\in\mathbb Z_k$.
For $1\le i\le6$, define
\[
B_0
 =
(G_0+\{0,1\})\cup(-G_0-\{0,1\}),
\qquad
B_i=G_i+\{0,1\},
\]
and
\[
A_0
 =
B_0\cup
\bigcup_{i=1}^{6}
\bigl((T_i+B_i)\cup(-T_i-B_i)\bigr).
\]

For every finite set $A\subseteq\mathbb Z_k$, put
\[
Q(A)
 =
\{x\in\mathbb Z_k:
  \chi(x+a)=1\text{ for every }a\in A\},
\]
where $\chi$ denotes the Legendre symbol modulo $k$. Define
\[
D^+
 =
\bigcup_{i=1}^{6}\{\pm(T_i+1)\},
\qquad
D^-
 =
\bigcup_{i=1}^{6}\{\pm(T_i-1)\},
\]
and
\[
\widetilde S=Q(A_0)\setminus\{0\},
\qquad
S=(\widetilde S\cup D^+)\setminus D^-.
\]
Assume that $0\notin D^+\setminus D^-$, so that $0\notin S$.
The Cayley graph associated with this additive template is
$H=\operatorname{Cay}(\mathbb Z_k,S).$
\end{definition}
We shall use the two symmetric sets of differences
\[
\Upsilon
 =
\{\pm1,\pm T_1,\dots,\pm T_6\}
\]
and
\[
\Delta
 =
\Upsilon
\cup
\bigcup_{i=1}^{6}
\{\pm(T_i+1),\pm(T_i-1)\}.
\]

We now construct a concrete $H$ by specifying the different sets above. Throughout this section,
$k\equiv1\pmod4$ is prime and
$\chi:\mathbb Z_k\to\{-1,0,1\}$ denotes the Legendre symbol. Since
$\chi(-1)=1$, the set of nonzero quadratic residues is invariant under
negation. A set $A\subseteq\mathbb Z$ is {\it Sidon} if every equality
$a+b=c+d$ with $a,b,c,d\in A$ implies
$\{a,b\}=\{c,d\}$.
Given two sets $A, B$, define $A + B = \{a+b: a \in A, b \in B\}$, and $A - B = \{a-b: a \in A, b \in B\}$. The \emph{greedy Sidon
sequence} $(a_j)_{j\ge1}$ is defined recursively by $a_1=1$ and, for
$j\ge2$, by taking $a_j$ to be the smallest integer larger than
$a_{j-1}$ for which $\{a_1,\dots,a_j\}$
is a Sidon set.

\begin{definition}[Construction of $H$]\label{def:Hconstruction}
Fix six distinct odd primes $11\le M_1<\cdots<M_6$ and an integer
$L>\sum_{i=1}^{6}M_i$. Let $a_1,\dots,a_L$ be the first $L$ terms of
the greedy Sidon sequence, put $g_j=La_j$ and
$G_0=\{g_1,\dots,g_L\}$. Writing $m_0=0$ and
$m_i=\sum_{h=1}^{i}M_h$, define
$G_i=\{g_{m_{i-1}+1},\dots,g_{m_i}\}$.

For every sufficiently large prime $k\equiv1\pmod4$, let
$T_i=\lfloor k^{i/7}\rfloor$ for $1\le i\le6$. 
Let $H=\operatorname{Cay}(\mathbb Z_k,S)$ be the Cayley graph with the additive
template in Definition~\ref{def:additivetemplate}. 
    
\end{definition}

\begin{lemma}[Explicit realization of the additive template]
\label{lem:distinctcod}
Let $H=\operatorname{Cay}(\mathbb Z_k,S)$ be the Cayley graph constructed as in Definition \ref{def:Hconstruction}. 
For every sufficiently large prime $k\equiv1\pmod4$, it  has the following
properties.

\begin{enumerate}
    \item The auxiliary graph
    $\Theta=\operatorname{Cay}(\mathbb Z_k,\Upsilon)$ has diameter at
    most $14k^{1/7}$.

    \item For $A_d=A_0\cap(A_0+d)$, one has
    \begin{align*}
    |A_0|&=4L+4\sum_{i=1}^{6}M_i,\\
    |A_{\pm1}|&=2L+2\sum_{i=1}^{6}M_i,\\
    |A_{\pm T_i}|&=4M_i &&(1\le i\le6),\\
    |A_{\pm(T_i+1)}|
      =|A_{\pm(T_i-1)}|&=2M_i &&(1\le i\le6).
    \end{align*}
    Moreover, $|A_d|\le16$ whenever
    $d\notin\Delta\cup\{0\}$. The displayed marked values are pairwise
    distinct except for
    $|A_{\pm(T_i+1)}|=|A_{\pm(T_i-1)}|$ for the same $i$, and every
    marked value is greater than $16$.

    \item If $O_{\max}=2L+2\sum_{i=1}^{6}M_i$, then
    $\max_{d\ne0}|A_d|=O_{\max}=|A_0|/2$.

    \item The connection set satisfies $D^+\subseteq S$,
    $D^-\cap S=\varnothing$, $S=-S$, $0\notin S$, and
    $|S\mathbin{\triangle}Q(A_0)|\le25$. In particular, $H$ is
    undirected and loopless.
\end{enumerate}
\end{lemma}
\begin{proof}
Let $\mathcal T=\{0,\pm T_1,\dots,\pm T_6\}$. We first establish the
separation properties needed for the overlap calculation. For all
sufficiently large $k$, we have $T_1>10g_L$,
$4T_6+4g_L+4<k$, and
\[
\min_{\substack{u,v\in\mathcal T-\mathcal T\\u\ne v}}
\|u-v\|_k>4g_L+4.
\tag{11}\label{eq:scale-separation}
\]
Indeed, $T_1\to\infty$ as $k$ goes to infinity and
$T_6=O(k^{6/7})=o(k)$. If
$u,v\in\mathcal T-\mathcal T$ are distinct, then $u-v$ is a nonzero
integer linear combination $\sum_{i=1}^{6}c_iT_i$ with uniformly
bounded coefficients. If $j$ is the largest index for which
$c_j\ne0$, then
$u-v=c_jk^{j/7}+O(k^{(j-1)/7})$. Hence $|u-v|\to\infty$, while
$|u-v|\le4T_6=o(k)$, and consequently $\|u-v\|_k\to\infty$. Since
only finitely many coefficient vectors can occur,
\eqref{eq:scale-separation} holds uniformly.

We first prove the diameter bound. Put $Q=k^{1/7}$ and $T_0=1$. Given
$x\in\{0,\dots,k-1\}$, perform successive Euclidean divisions by
$T_6,T_5,\dots,T_1$. This gives
$x=a_0+\sum_{i=1}^{6}a_iT_i$, where $0\le a_0<T_1$,
$0\le a_i<T_{i+1}/T_i$ for $1\le i\le5$, and
$0\le a_6<k/T_6$. For sufficiently large $k$, we have $a_0<Q$ and
$a_i<2Q$ for $1\le i\le6$. Therefore
$a_0+\sum_{i=1}^{6}a_i<13Q<14k^{1/7}$. Since the displayed
representation gives a walk from $0$ to $x$ using differences in
$\Upsilon$, the diameter of $\Theta$ is at most $14k^{1/7}$.

We next calculate the overlaps. Every element of one of the thirteen
blocks $B_0,T_i+B_i,-T_i-B_i$ lies within cyclic distance at most
$g_L+1$ of its corresponding macro-location in $\mathcal T$. Hence
\eqref{eq:scale-separation} implies that these blocks are pairwise
disjoint. Their constituent two-point intervals are also disjoint:
differences between distinct elements of $G_0$ are nonzero multiples
of $L>4$, and the positive and negative parts of $B_0$ cannot meet
for sufficiently large $k$. In particular, every $v\in A_0$ has a
unique representation $v=t_v+s_v+\epsilon_v$, where the admissible
triples are
\begin{align*}
&t_v=0,\quad s_v=g\in G_0,\quad
  \epsilon_v\in\{0,1\},\\
&t_v=0,\quad s_v=-g,\ g\in G_0,\quad
  \epsilon_v\in\{0,-1\},\\
&t_v=T_i,\quad s_v=g\in G_i,\quad
  \epsilon_v\in\{0,1\},\\
&t_v=-T_i,\quad s_v=-g,\ g\in G_i,\quad
  \epsilon_v\in\{0,-1\}.
\end{align*}
The definition also gives $A_0=-A_0$, and hence
$|A_d|=|A_{-d}|$.

We shall use the following consequence of
\eqref{eq:scale-separation}. Suppose two ordered pairs
$(x,y),(x',y')\in A_0^2$ determine the same difference:
$y-x=y'-x'$. If
$t_y-t_x\ne t_{y'}-t_{x'}$, then
\eqref{eq:scale-separation} shows that the cyclic distance between
these two macro-differences is greater than $4g_L+4$. Equality of the
two full differences, however, shows that this distance is at most
$4g_L+4$, a contradiction. Thus
$t_y-t_x=t_{y'}-t_{x'}$.

It follows that
\[
(s_y-s_x)+(\epsilon_y-\epsilon_x)
 =
(s_{y'}-s_{x'})+(\epsilon_{y'}-\epsilon_{x'}).
\]
The difference between the two $s$-terms is divisible by $L$, while
the difference between the two $\epsilon$-terms has absolute value at
most $4$. Since $L>4$, the $s$-differences and the
$\epsilon$-differences must agree separately.

The set $A_0$ consists of
$2L+2\sum_{i=1}^{6}M_i$ disjoint two-point intervals. Therefore
$|A_0|=4L+4\sum_iM_i$. Each such interval contributes one ordered pair
at difference $1$, and the preceding uniqueness argument rules out
any other contribution. Hence
$|A_{\pm1}|=2L+2\sum_iM_i$.

Fix $i$. At macro-difference $T_i$ and generator difference zero,
there are precisely two ordered macro-block pairs: from the positive
central block to the positive $T_i$-block, and from the negative
$-T_i$-block to the negative central block. For every $g\in G_i$,
these two macro-block pairs together give four ordered pairs at
difference $T_i$, two at difference $T_i+1$, and two at difference
$T_i-1$. Consequently,
$|A_{\pm T_i}|=4M_i$ and
$|A_{\pm(T_i+1)}|=|A_{\pm(T_i-1)}|=2M_i$.

Now suppose $d\notin\Delta\cup\{0\}$. In any representation
$d=y-x$, we must have $s_y-s_x\ne0$, since $s_y-s_x=0$ would force
$d\in\Delta\cup\{0\}$. Because $G_0$ is Sidon, a fixed nonzero signed
difference has at most four ordered representations in $\pm G_0$: at
most one in each same-sign pattern and at most two in the applicable
opposite-sign pattern.

For each fixed ordered pair $(s_x,s_y)$, there are at most two
compatible ordered macro-block pairs with the prescribed
macro-difference. Indeed, each signed generator occurs in the central
block and in at most one satellite block. Finally, once the
macro-block pair has been fixed, a prescribed value of
$\epsilon_y-\epsilon_x$ has at most two representations. Therefore
$|A_d|\le4\cdot2\cdot2=16$.

It remains to compare the marked overlap values. Since $M_1\ge11$,
every $2M_i$ and $4M_i$ is greater than $16$. The values $2M_i$ are
pairwise distinct, as are the values $4M_i$. An equality
$2M_i=4M_j$ is impossible because all the $M_i$ are odd. Finally,
\[
2L+2\sum_{i=1}^{6}M_i
 >
4\sum_{i=1}^{6}M_i
 >
4M_j
\]
for every $j$. This proves all the assertions about the overlaps and,
in particular, gives
$\max_{d\ne0}|A_d|=O_{\max}=|A_0|/2$.

We finally verify the assertions about $S$. The separation conditions
imply $D^+\cap D^-=\varnothing$ and
$0\notin D^+\cup D^-$. Since $A_0=-A_0$ and $\chi(-1)=1$, we have
$Q(A_0)=-Q(A_0)$. The sets $D^+$ and $D^-$ are also invariant under
negation, so $S=-S$. Its definition gives
$D^+\subseteq S$, $D^-\cap S=\varnothing$, and $0\notin S$.
Therefore $H$ is undirected and loopless.

Passing from $Q(A_0)$ to $S$ can delete $0$, add at most the twelve
elements of $D^+$, and delete at most the twelve elements of $D^-$.
Thus $|S\mathbin{\triangle}Q(A_0)|\le25$.
\end{proof}

\subsection{From additive overlaps to codegrees}

Lemma~\ref{lem:distinctcod} Items 2 and 3 gives an exact description of the translate
overlaps of $A_0$. We now use character-sum estimates to convert these
overlaps into degree and codegree estimates for $H$. The resulting
codegree levels, together with the bounded correction distinguishing
$T_i+1$ from $T_i-1$, will make the differences in $\Delta$
recognizable from the graph.
\begin{claim}\label{claim:expreNA}
For every finite $A\subseteq\mathbb Z_k$,
\[
\left|
 |Q(A)|-\frac{k}{2^{|A|}}
\right|
\le
\frac{|A|}{2}\sqrt{k}+|A|.
\]
\end{claim}

\begin{proof}
We express $|Q(A)|$ analytically over the entire field $\mathbb{Z}_k$ as
\[|Q(A)| = \sum_{x \in \mathbb{Z}_k} \prod_{a \in A} \frac{1 + \chi(x+a)}{2} - E_A.\]
where $E_A =
\sum_{x\in-A}
\prod_{a\in A}\frac{1+\chi(x+a)}2.$
Indeed, outside $-A$ the product is the indicator that $\chi(x+a)=1$ for every $a\in A$, whereas the points in $-A$ may
contribute nonzero fractional terms to the product. Clearly
$|E_A|\le|A|$.

We can expand the expression as 
\begin{align*}
    |Q(A)| &= \frac{1}{2^{|A|}} \sum_{x \in \mathbb{Z}_k} \left( 1 + \sum_{\emptyset \neq B \subseteq A} \prod_{b \in B} \chi(x+b) \right) - E_A
    \\
    &= \frac{k}{2^{|A|}} + \frac{1}{2^{|A|}} \sum_{\emptyset \neq B \subseteq A} \sum_{x \in \mathbb{Z}_k} \chi \left( \prod_{b \in B} (x+b) \right) - E_A.
\end{align*}
For any non-empty $B \subseteq A$, the polynomial $f(x) = \prod_{b \in B}(x+b)$ has distinct roots, so it is not a perfect square in $\mathbb{Z}_k[x]$. 
By Weil's bound, \[\left| \sum_{x} \chi(f(x)) \right| \leq (|B|-1)\sqrt{k}.\] Thus the total deviation is bounded above by
\begin{align*}
\left||Q(A)|-\frac{k}{2^{|A|}}\right|
&\le
\frac{1}{2^{|A|}}
\sum_{j=1}^{|A|}\binom{|A|}{j}(j-1)\sqrt{k}
+|E_A|\\
&\le
\frac{|A|2^{|A|-1}-2^{|A|}+1}{2^{|A|}}\sqrt{k}
+|A| \le
\frac{|A|}{2}\sqrt{k}+|A|.
\end{align*}
as desired.
\end{proof}

 \begin{lemma}\label{lem:codegree}
For $d\in\mathbb Z_k$, let 
$A_d=A_0\cap(A_0+d)$ and let $c(d)$ be the codegree of two vertices at difference $d$ in $H$,
and put $c(0)=|S|$. Then for every  $d$,
\[
c(d)
 =
\frac{k}{2^{\,2|A_0|-|A_d|}}+\epsilon_d,
\]
where
\[
|\epsilon_d|
\le
|A_0|\sqrt{k}+2|A_0|+50.
\]
\end{lemma}

\begin{proof}
First consider the unmodified set $S_0=Q(A_0).$
Then $|S_0\cap(S_0+d)| =
|Q(A_0\cup(A_0-d))|.$ By inclusion-exclusion,
\[
|A_0\cup(A_0-d)|=2|A_0|-|A_0\cap(A_0-d)| =2|A_0| - |A_d|.
\]
where the second equality follows by translating $A_0\cap(A_0-d)$ by $d$. 
The claim now follows for $S_0$ from Claim~\ref{claim:expreNA}, with error at most $|A_0|\sqrt{k}+2|A_0|.$

By Lemma \ref{lem:distinctcod}, the actual generator set $S$ satisfies
$|S\triangle S_0|\le 25$
as we may delete $0$, add the twelve elements of $D^+$, and delete the
twelve elements of $D^-$. 
For every $d$,
\[
\bigl(S\cap(S+d)\bigr)
\triangle
\bigl(S_0\cap(S_0+d)\bigr)
\subseteq (S\triangle S_0) \cup((S\triangle S_0)+d).
\]
Thus the  modification from $S_0$ to $S$ changes every codegree by at most
$2|S\triangle S_0|\le50.$
This proves the lemma.
\end{proof}

\section{Verification of the Structural Criterion}\label{sec:final}

Now we can prove the local distance preserving lemma. 

\begin{lemma}\label{lem:localdist}
Let $X\subseteq V(H)$ satisfy
$|X|\ge(1-\delta)k,$
and let  $f:H[X]\to H$
be an induced embedding. If $x,x+d\in X$ and $d\in\Delta$, then
\[
f(x+d)-f(x)\in\{\pm d\}.
\]
\end{lemma}

\begin{proof}
Let $c_X(x,x+d)$ be the number of common neighbors of $x$ and $x+d$
belonging to $X$. Since at most $\delta k$ vertices are missing,
\[
c(d)-\delta k
\le c_X(x,x+d)
\le c(d).
\]
Let
$d'=f(x+d)-f(x).$
Because $f$ is an induced embedding, the common neighbors of the two
image vertices lying in $f(X)$ are exactly the images of the common
neighbors in $X$. Therefore
\[
c(d')-\delta k
\le c_X(x,x+d)
\le c(d').
\]
Consequently,
\begin{equation}\label{eq:differencec}
|c(d')-c(d)|\le\delta k.
\end{equation}

Lemma~\ref{lem:codegree} gives
\[
\left|
\frac{k}{2^{\,2|A_0|-|A_{d'}|}}
-
\frac{k}{2^{\,2|A_0|-|A_d|}}
\right|
\le
\delta k+O(\sqrt{k}).
\]
Since $\delta=o(1)$ and $|A_0|$ is fixed, distinct values of $|A_d|$
give main terms separated by a fixed positive multiple of $k$. Hence
$|A_{d'}|=|A_d|$
for sufficiently large $k$.

Injectivity gives $d'\ne0$. If $d\in\Upsilon$, then
Lemma~\ref{lem:distinctcod} implies
$d'\in\{\pm d\}.$

If $d\in\{\pm(T_i+1),\pm(T_i-1)\},$
the same lemma first gives
\[
d'\in\{\pm(T_i+1),\pm(T_i-1)\}.
\]
By construction,
$\{\pm(T_i+1)\}\subseteq S$ and $\{\pm(T_i-1)\}\cap S=\emptyset.$
Since $f$ preserves both adjacency and nonadjacency, it cannot interchange
these two orbits. Therefore again
$d'\in\{\pm d\}.$
\end{proof}

We now synthesize the preceding lemmas to  establish the local-to-global rigidity of the explicitly constructed Cayley graph $H$.

\begin{proof}[Proof of Proposition \ref{prop:main}]
Let
$f:H[X]\to H$
be an induced embedding, where $|X|\ge(1-\delta)k$. By
Lemma~\ref{lem:localdist}, $f$ preserves every distance in $\Delta$ up
to sign.

By Lemma \ref{lem:distinctcod}, $D\le14k^{1/7}.$
Since $\delta=10^3(\log k)^{1/5}k^{-1/5},$
we have $D\delta
 = O\left((\log k)^{1/5}k^{-2/35}\right) = o(1).$
Thus $210D\delta<1$ for sufficiently large $k$, and
Lemma~\ref{lem:g-rigid} gives a set $C_{\rm giant}$ satisfying $|C_{\rm giant}|\ge(1-210D\delta)k$
on which
$f(x)=\sigma x+b$
for fixed $\sigma\in\{\pm1\}$ and $b\in\mathbb Z_k$.

    To extend this affine isometry from almost-everywhere rigidity to absolute global rigidity, we will utilize the codegree threshold required by Lemma \ref{lem:outlier}.

    Let
\[
O_{\max}=2L+2\sum_{i=1}^6M_i=\frac{|A_0|}{2}
\]
and
\[
\alpha
 =
\frac{1}{2^{|A_0|}}
\left(
1-\frac{1}{2^{|A_0|-O_{\max}}}
\right)>0.
\]
Lemma~\ref{lem:codegree} gives, uniformly for $d\ne0$,
\[
|S|-c(d)\ge\alpha k-O(\sqrt{k}) \geq \alpha k / 2.
\]
for sufficiently large $k$. Since $D\delta=o(1)$, we may also assume  $105D\delta<\frac{\alpha}{2}.$
It follows that
\[
\frac12(k-|C_{\rm giant}|)
\le105D\delta k
<
\frac{\alpha k}{2}
\le|S|-c(d).
\]
Therefore Lemma~\ref{lem:outlier} applies and gives $f(x)=\sigma x+b$ for every $x\in X$. 

Define $\phi:\mathbb Z_k\to\mathbb Z_k$ as 
$\phi(x)=\sigma x+b.$
Since $S=-S$, $\phi$ is an automorphism of $H$, and the earlier
identity says precisely that $
f=\phi|_X.$
This proves Proposition~\ref{prop:main}.
\end{proof}

We now prove two additional structural properties of the Cayley graph $H = \operatorname{Cay}(\mathbb{Z}_k, S)$.

\begin{lemma}[Neighborhood Separation]\label{lem:q}
    There exists a strictly positive constant $q > 0$ such that for every sufficiently large prime \(k\equiv1\pmod4\), and for any two distinct vertices $v, w \in V(H)$, there are at least $q k$ vertices in $V(H) \setminus \{v, w\}$ that are adjacent to exactly one of the vertices $v$ and $w$.
\end{lemma}
\begin{proof}
    The set of vertices adjacent to exactly one of $v$ or $w$ is the symmetric difference of their neighborhoods, denoted $N(v) \triangle N(w)$. Since $H$ is a Cayley graph, $N(v) = v + S$ and $N(w) = w + S$. Factoring out the translation, the size of this symmetric difference is the same as $|S \triangle (S - c)|$, where $c = v - w \neq 0$.
    
    We compute the symmetric difference of the neighborhood of two vertices. 
    \[|N(v) \triangle N(w)| = |S \triangle (S - c)| = 2|S| - 2|S \cap (S - c)| = 2c(0) - 2c(c).\]
    
    By Lemma \ref{lem:distinctcod}, the maximum possible overlap for any non-zero distance is strictly bounded by $O_{\max} = 
2L+2\sum_{i=1}^6M_i=
\frac{|A_0|}{2}$. Thus, the codegree $c(c)$ is bounded from above by substituting $O_{\max}$ into the Weil bound formula from Lemma \ref{lem:codegree}:
    \[|N(v) \triangle N(w)| \ge 2\left( \frac{k}{2^{|A_0|}} \right) - 2\left( \frac{k}{2^{2|A_0| - O_{\max}}} \right) - O(\sqrt{k})\]
    \[= \frac{2k}{2^{|A_0|}} \left( 1 - \frac{1}{2^{|A_0| - O_{\max}}} \right) - O(\sqrt{k}).\]
    
    Recall from the construction of $A_0$ that $|A_0| = 4L + 4\sum_{i=1}^6 M_i$, and $O_{\max} = 2L + 2\sum_{i=1}^6 M_i$. This strictly implies that $|A_0| = 2 O_{\max}$, and therefore $|A_0| > O_{\max}$. Because $|A_0|$ and $O_{\max}$ are fixed structural parameters independent of $k$, we can define the absolute constant:
    \[q = \frac{1}{2^{|A_0|}} \left( 1 - \frac{1}{2^{|A_0| - O_{\max}}} \right).\]
    Since $|A_0|>O_{\max}$, we have $0<q<1/2$.
 Substituting this back yields:
    \[|N(v) \triangle N(w)| \ge 2q k - O(\sqrt{k}).\]
    
    The lemma  asks for the number of such vertices excluding $v$ and $w$ themselves. This requires subtracting at most 2 from our bound. For sufficiently large $k$, the dominant linear term $2qk$ is much larger than the sublinear error term $O(\sqrt{k}) + 2$. Therefore, we  bound the size from below by $q k$:
    \[|N(v) \triangle N(w) \setminus \{v, w\}| \ge 2q k - O(\sqrt{k}) - 2 \ge q k.\]
    This confirms that every pair of vertices in $H$ has highly distinct neighborhoods. 
\end{proof}

We next show that $H$ is a prime graph. 
\begin{lemma}\label{lem:prime}
    The graph $H$ is a prime graph. That is, there is no subset $U \subseteq V(H)$ with $2 \leq |U| < |V(H)|$ such that every vertex outside $U$ is either complete to $U$ or anti-complete to $U$.
\end{lemma}
\begin{proof}
    Suppose for the sake of contradiction that $H$ contains a non-trivial such $U$ where $2 \leq |U| < k$. 
    
    Because $|U| \ge 2$, we can select two distinct vertices $v, w \in U$. By the definition of $U$, any vertex $x \notin U$ has identical adjacency to all vertices in $U$. Therefore, $x$ is either adjacent to both $v$ and $w$, or adjacent to neither. 
    
    This implies that no vertex outside of $U$ can be adjacent to exactly one of $\{v, w\}$. Consequently, the symmetric difference of their neighborhoods must  satisfy
    \[N(v) \triangle N(w) \subseteq U.\]
    This imposes a strict lower bound on the size of the $U$ as 
    \[|U| \ge |N(v) \triangle N(w)|.\]
    
    From our proof of Lemma \ref{lem:q}, we know analytically that for any distinct $v, w$, the size of the symmetric difference is $|N(v) \triangle N(w)| = 2|S| - 2c(v-w)$. Substituting this into our bound yields:
    \begin{equation}
        |U| \ge 2|S| - 2c(v-w). \label{eq:module_size}
    \end{equation}
    
    Now, because $|U| < k$, there are vertices outside of $U$. By definition, there are only two possibilities for a vertex outside of $U$: it is either complete to $U$ (adjacent to every vertex in $U$) or anti-complete to $U$ (no edges to $U$).
    
Case 1: There exists an $x$ that is complete to $U$.
    If $x$ is connected to every vertex in $U$, then its degree must be at least the size of the module: $\text{deg}(x) \ge |U|$. 
    Since $H$ is a regular graph, $\text{deg}(x) = |S|$. Substituting this into inequality (\ref{eq:module_size}) yields:
    \[|S| \ge 2|S| - 2c(v-w) \implies |S| \leq 2c(v-w).\]
    However, by Lemma \ref{lem:codegree} and Lemma \ref{lem:distinctcod}, the degree is $|S| = \frac{k}{2^{|A_0|}} \pm O(\sqrt{k})$, while the maximum non-zero codegree is bounded by $c(v-w) \leq \frac{k}{2^{2|A_0| - O_{\max}}} + O(\sqrt{k})$.  Consequently,
\[
|S|-2c(v-w)
\geq 
\left(
\frac1{2^{|A_0|}}
-
\frac2{2^{\,2|A_0|-O_{\max}}}
\right)k
+
O(\sqrt{k}).
\]
The coefficient of $k$ is positive because
$O_{\max}=|A_0|/2$ and $|A_0|/2>1$. Hence
$|S|>2c(v-w)$ for all sufficiently large $k$. 
Thus, no vertex $x \notin U$ can be complete to $U$.
    
Case 2: Every $x$ is anti-complete to $U$.
    If every vertex $x \notin U$ has zero edges connecting to $U$, then there are  no edges between $U$ and $V(H) \setminus U$. Since both $U$ and its complement are nonempty, this implies that $H$ is disconnected.
    
    However, $H = \operatorname{Cay}(\mathbb{Z}_k, S)$ is a Cayley graph over a group of prime order. Because $k$ is prime, $\mathbb{Z}_k$ contains no non-trivial subgroups. This means that if the generator set $S$ is non-empty, choosing any single element $s \in S$ generates the entire group (i.e., the edges corresponding to just $s$ form a spanning cycle of length $k$). 
    
Lemma \ref{lem:codegree} shows that the degree of the graph is $|S| = \frac{k}{2^{|A_0|}} \pm O(\sqrt{k})$. For sufficiently large $k$, $|S| > 0$, guaranteeing the generator set is non-empty. Therefore, $H$ is always a connected graph. This leads to a  contradiction.
\end{proof}

To prove Theorem~\ref{thm:main_part2}, choose $k$ sufficiently large.
Proposition~\ref{prop:main}, Lemma~\ref{lem:q}, and
Lemma~\ref{lem:prime} show that $H$ is $(q,\delta)$-reasonable. The
constant $q>0$ is independent of $k$, and hence
$q\ge10^4(\log k)^{6/5}k^{-1/5}$
for sufficiently large $k$. Moreover, $
\delta=10^3(\log k)^{1/5}k^{-1/5}.$
Thus all the hypotheses of Theorem~\ref{thm:FSW} are satisfied, and
$H$ is a fractalizer. Lastly,
$|S|=\frac{k}{2^{|A_0|}}+O(\sqrt{k})$,
so for sufficiently large $k$ we have
$0<|S|<k-1$. Hence $H$ is neither empty nor complete. This proves Theorem~\ref{thm:main_part2}.

\part{Prime-order Paley graphs: proof of Theorem~\ref{thm:main}}
\section{A full-graph fractalizer criterion}

This second part is devoted to proving theorem \ref{thm:main}. 
For Paley graphs, the automorphism group $\operatorname{Aut}(P_p)$ contains all affine linear maps $\mathbb F_p\to \mathbb F_p$ of the form $ax+b$ with $a$ a quadratic residue and $b$ arbitrary. Since $a$ is not restricted to $\pm 1$, the original $(q,\delta)$-reasonableness condition of Definition \ref{def:reasonable_part_two} does not directly apply. Secondly, Paley graphs are deletion tolerant at a different scale: The Paley rigidity theorem proven below allows only $\sqrt{p}/3$ deleted vertices. This is less than the scale used in part I. The new fractalizer condition in Proposition~\ref{prop:full-fsw} accommodates both aspects.

We begin by introducing the notion of a $(q,\delta;\Phi)$-reasonable graph.

\begin{definition}\label{def:reasonable}
Let $H$ be a graph on $k$ vertices, let
$\Phi\subseteq\Aut(H)$ be a specified family of automorphisms, and let
$0<q<1/2$ and $0<\delta<1$.  We say that $H$ is
\emph{$(q,\delta;\Phi)$-reasonable} if the following conditions hold.

\begin{enumerate}[label=\textnormal{(\alph*)},leftmargin=2.5em]
\item $H$ is prime.
\item For every two distinct vertices $x,y\in V(H)$, at least $qk$ vertices of $V(H)\setminus\{x,y\}$ are adjacent to exactly one of $x,y$.
\item If $X\subseteq V(H)$ has $|X|\ge (1-\delta)k$ and $f:X\to V(H)$ is injective with
$a_H(f(x),f(y))=a_H(x,y)$ for all distinct $x,y\in X$, then $f=\phi|_X$ for some $\phi\in\Phi$.
\end{enumerate}
\end{definition}
For comparison, in Definition \ref{def:reasonable_part_two} the corresponding
rigidity condition requires every sufficiently large partial embedding
to agree with a translation or reflection of the Cayley graph.
In the new formulation, we use the specified family
$\Phi\subseteq\Aut(H)$ in place of those translations and reflections. Both the condition (a) of being prime and condition (b) regarding neighborhood separation are unchanged.

Condition~\textnormal{(c)} automatically identifies this specified
family with the full automorphism group.  Indeed, if
$\psi\in\Aut(H)$, then applying condition~\textnormal{(c)} with
$X=V(H)$ and $f=\psi$ gives $\psi\in\Phi$.  Since
$\Phi\subseteq\Aut(H)$ by assumption, it follows that
$\Phi=\Aut(H)$.  We retain $\Phi$ in the notation because, in
applications, the partial-rigidity argument naturally identifies an
explicit family of symmetries before thereby proving that this family
is the full automorphism group.

Proposition~\ref{prop:full-fsw} below is not a formal specialization
of Theorem \ref{thm:FSW}.  It works for a weaker rigidity scale
$\delta=\Theta(k^{-1/2})$ and allows an arbitrary automorphism group
of order at most $k^2$. The new estimate of
Section~\ref{sec:weighted-hitting} makes the first change possible,
while the full-graph symmetry argument makes the second possible.

\begin{proposition}[Fractalizer criterion]\label{prop:full-fsw}
For any $q_0\in(0,1/2)$ and $c_0>0$, there exists $k_0=k_0(q_0,c_0)$ with the following property.  Let $H$ be a graph on $k\ge k_0$ vertices, let $\Phi\subseteq \Aut(H)$ satisfy
$2\le |\Phi|\le k^2$, and suppose that $H$ is $(q_0,\delta;\Phi)$-reasonable for some $\delta\ge c_0k^{-1/2}$.  Then $H$ is a fractalizer.
\end{proposition}

All asymptotic thresholds in this proposition concern the order
$k=|V(H)|$ of the pattern graph, not the order of the host graph.
Thus, once $H$ satisfies the hypotheses, the conclusion holds
simultaneously for every $n$ and every $n$-vertex graph maximizing
$\emb(H,\cdot)$; no lower bound on $n$ is imposed.

For integers \(k\ge1\) and \(n\ge0\), let
\begin{equation}\label{eq:Ek-def}
 \E_k(n):=\max\left\{\prod_{i=1}^k n_i:\ n_i\in\mathbb Z_{\ge 0},\ \sum_{i=1}^k n_i=n\right\}.
\end{equation}
The balanced choice, in which the $n_i$ differ by at most one, attains
$\E_k(n)$.  When $n\ge k$, these are precisely the maximizing
$k$-tuples. We also set $\E_0(n):=1$.

The next lemma is the bridge from primeness to the recursive
structure of extremizers.  It says that if an induced copy of $H$ in
a blow-up $G$ meets more than one top-level part, then it meets each
top-level part in at most one vertex.  The list of parts it visits
therefore defines an induced self-embedding of $H$, and partial
rigidity identifies this map with a global automorphism.
\begin{lemma}\label{lem:embedding-classification}
Let $H$ be a $(q,\delta;\Phi)$-reasonable graph on $k\ge2$
vertices, and let $G$ be a blow-up of $H$ with parts $(W_j)_{j\in V(H)}$.  Every embedding $\vartheta:H\hookrightarrow G$ is of exactly one of the following two types:
\begin{enumerate}[label=\textnormal{(\roman*)},leftmargin=2.5em]
\item $\vartheta(V(H))\subseteq W_j$ for some $j\in V(H)$;
\item there exists $\phi\in\Phi$ such that $\vartheta(x)\in W_{\phi(x)}$ for every $x\in V(H)$.
\end{enumerate}
\end{lemma}

\begin{proof}
For each $j\in V(H)$, let
$X_j:=\{x\in V(H):\vartheta(x)\in W_j\}$.
Thus $X_j$ is the set of vertices of $H$ whose images lie in the part $W_j$.
Suppose first that $X_j=V(H)$ for some $j$.  Then
$\vartheta(V(H))\subseteq W_j$, and we are in case~\textnormal{(i)}.

We may therefore assume that the image of $\vartheta$ is not contained in
a single part.  We claim that $|X_j|\le 1$ for every $j\in V(H)$.
Suppose otherwise that, for some $j$, we have
$2\le |X_j|\le k-1$.  Let $z\in V(H)\setminus X_j$.  Then
$\vartheta(z)$ lies in a part $W_\ell$ with $\ell\neq j$.  By the
definition of a blow-up, the adjacency between $W_j$ and $W_\ell$ is
constant: either every vertex of $W_\ell$ is adjacent to every vertex of
$W_j$, or there are no edges between the two parts.  It follows that
$\vartheta(z)$ has the same adjacency to every vertex in
$\vartheta(X_j)$.  Since $\vartheta$ preserves both adjacency and
nonadjacency, the vertex $z$ has the same adjacency to every vertex of
$X_j$.
This holds for every $z\in V(H)\setminus X_j$.  Thus $X_j$ is a proper
subset of $V(H)$, containing at least two vertices, such that every
vertex outside $X_j$ is adjacent either to all vertices of $X_j$ or to
none of them.  This contradicts the assumption that $H$ is prime.
Therefore $|X_j|\le 1$ for every $j$.

Consequently, distinct vertices of $H$ are mapped by $\vartheta$ into
distinct parts of the blow-up.  We may therefore define a map
$f:V(H)\to V(H)$ by declaring $f(x)$ to be the unique index such that
$\vartheta(x)\in W_{f(x)}$.  The preceding argument shows that $f$ is injective.  Since its domain
and codomain both have $k$ elements, $f$ is bijective.
Let $x,y\in V(H)$ be distinct.  Since $f(x)\neq f(y)$, the vertices
$\vartheta(x)$ and $\vartheta(y)$ lie in two distinct parts,
$W_{f(x)}$ and $W_{f(y)}$.  Hence
\[
a_H(f(x),f(y))
=
a_G(\vartheta(x),\vartheta(y))
=
a_H(x,y).
\]
The first equality follows from the definition of a blow-up, and the
second follows because $\vartheta$ is an induced embedding.

We now apply Definition~\ref{def:reasonable}\textnormal{(c)} with
$X=V(H)$.  It follows that there exists $\phi\in\Phi$ such that
$f(x)=\phi(x)$ for every $x\in V(H)$.  Therefore
we have case~\textnormal{(ii)}.

Finally, the two cases are mutually exclusive.  Indeed, in
case~\textnormal{(ii)}, the map $\phi$ is a permutation of $V(H)$, so
the images of the $k$ vertices of $H$ lie in $k$ distinct parts, whereas
in case~\textnormal{(i)} they all lie in one part.
\end{proof}

\section{The Paley graph satisfies the rigidity hypotheses}

Let $p\equiv 1\pmod 4$ be prime, and let $\chi$ denote the quadratic character of $\F_p$, extended by $\chi(0)=0$.  Thus $x$ and $y$ are adjacent in $P_p$ if and only if $\chi(x-y)=1$.  Define
\begin{equation}\label{eq:paley-affine-group}
 \Phi_p:=\{x\mapsto ax+b:\ a\in(\F_p^\times)^2,\ b\in\F_p\}.
\end{equation}
Every element of $\Phi_p$ is an automorphism of $P_p$, and
\begin{equation}\label{eq:paley-group-size}
 |\Phi_p|=\frac{p(p-1)}2<p^2.
\end{equation}
The converse inclusion $\Aut(P_p)\subseteq\Phi_p$ will follow from the partial affine rigidity lemma below.
Also note a multiplier $a$ that is a nonsquare is not an automorphism: it interchanges edges and nonedges.

We verify Definition~\ref{def:reasonable} one condition at a time.

\subsection{Pair separation}
The pair separation below is the pseudorandomness input in the abstract criterion.  The  neighborhood separation condition allows a logarithmic-size signature to be selected; see Lemma~\ref{lem:small-signature}.  For the Paley graph, the number of vertices distinguishing a given pair can be computed exactly by a quadratic-character identity.
\begin{lemma}\label{lem:paley-separation}
For any distinct $u,v\in\F_p$, exactly $(p-1)/2$ vertices
$z\in\F_p\setminus\{u,v\}$ are adjacent to exactly one of $u,v$.
\end{lemma}

\begin{proof}
For $z\notin\{u,v\}$, the indicator that $z$ is adjacent to exactly
one of $u,v$ is
\[
\frac{1-\chi((z-u)(z-v))}{2}.
\]

We parametrize $z$ as $z=u+(v-u)t$, then $(z-u)(z-v)=(v-u)^2t(t-1)$. 
Since $(v-u)^2$ is a nonzero square, it follows that
\[
\sum_{z\in\F_p}\chi((z-u)(z-v))
=
\sum_{t\in\F_p}\chi(t(t-1)).
\]

Let $N$ be the number of pairs $(t,y)\in\F_p^2$ satisfying
$y^2=t(t-1)$.  On the one hand, we have 
\[
N=\sum_{t\in\F_p}\bigl(1+\chi(t(t-1))\bigr)
=p+\sum_{t\in\F_p}\chi(t(t-1)).
\]
On the other hand, the equation $y^2=t(t-1)$ can be written as $(2y)^2 = (2t-1)^2 - 1$; by setting
$r=2t-1$ and $s=2y$, it becomes
$r^2-s^2=1,$ or equivalently, 
$(r-s)(r+s)=1.$
For every $\lambda\in\F_p^\times$, there is a unique solution that 
$r-s=\lambda$ and $r+s=\lambda^{-1}$. 
Hence $N=p-1$.  It follows that
\[
\sum_{t\in\F_p}\chi(t(t-1))=-1.
\]

Therefore $\#\{z\in\F_p\setminus\{u,v\}:z\text{ distinguishes }u,v\}$ is exactly 
\begin{align*}
\sum_{z\ne u,v} \frac{1-\chi((z-u)(z-v))}{2}=\sum_{z} \frac{1-\chi((z-u)(z-v))}{2} - 1= \frac{p-2}{2} - \frac{\sum_{t \in \mathbb{F}_p} \chi(t(t-1))}{2} =  
\frac{p-1}{2}.
\end{align*}
\end{proof}

In particular, Definition~\ref{def:reasonable}(b) holds with any fixed $q<1/2$ once $p$ is sufficiently large.

\subsection{Primeness}
The same separation phenomenon also rules out a nontrivial class of
vertices that looks identical from the outside.  This is precisely
the primeness condition needed in
Lemma~\ref{lem:embedding-classification}.
\begin{lemma}\label{lem:paley-prime}
The graph $P_p$ is prime.
\end{lemma}

\begin{proof}
Suppose, to the contrary, that there is a set
$M\subsetneq\F_p$ with $2\le |M|\le p-1$ such that every vertex
outside $M$ is either complete or anticomplete to $M$.  Choose
distinct $u,v\in M$.

No vertex outside $M$ can distinguish $u$ and $v$.  Hence all
$(p-1)/2$ vertices that distinguish them belong to $M$.  Since
neither $u$ nor $v$ is counted among these vertices,
\[
|M|-2\ge\frac{p-1}{2},
\]
and therefore $|M|\ge(p+3)/2$.

Every vertex of $P_p$ has degree $(p-1)/2$.  Thus a vertex outside
$M$ cannot be complete to $M$, since $|M|>(p-1)/2$.  It must
therefore be anticomplete to $M$.  Hence there are no edges between
$M$ and its nonempty complement, so $P_p$ is disconnected. Since $p\equiv1\pmod4$, the element $-1$ is a quadratic residue.
Thus $x-(x+1)=-1$ is a square for every $x\in\F_p$, so
$x$ is adjacent to $x+1$.  These edges form a spanning cycle, and
hence $P_p$ is connected, a contradiction.

\end{proof}

\subsection{Partial affine rigidity}

The partial-rigidity condition is the Paley-specific part of the
argument and one of the core arguments in this paper.  We encode a partial induced embedding $f$ by the point set
\[
\{(x,f(x)):x\in X\}
\]
in the affine plane.  Preservation of Paley adjacency and nonadjacency implies that every secant slope is a nonzero square, so the point set determines at most half of all directions.  The completion theorem (Theorem \ref{thm:szonyi-sziklai-completion}) below fills in the missing first coordinates without introducing new directions, and the R\'edei--Megyesi theorem then forces the completed point set to be a line.

For a point set $U\subseteq AG(2,p)$ in the affine plane, write $D(U)\subseteq\F_p\cup\{\infty\}$ for the set of directions determined by pairs of distinct points of $U$.

\begin{lemma}[Partial affine rigidity]\label{lem:partial-affine-rigidity}
Let $X\subseteq\F_p$ satisfy $|X|\ge p-\sqrt p/3$.  Suppose that $f:X\to\F_p$ is injective and
$\chi(f(x)-f(y))=\chi(x-y)$ for all distinct $x,y\in X$.  Then there exist $a,b\in\F_p$, with $a$ a nonzero square, such that $f(x)=ax+b$ for every $x\in X$.
\end{lemma}

We use the following completion theorem of Sz\H{o}nyi, in the slightly
generalized form of Sziklai, as recorded in
\cite[Theorem~18]{FST}.

\begin{theorem}[Sz\H{o}nyi \cite{Szonyi}, Sziklai \cite{Szi}]
\label{thm:szonyi-sziklai-completion}
Let $Q$ be a prime power, and let
$U\subseteq\operatorname{AG}(2,Q)$ have $Q-\varepsilon$ points.
Write $D=D(U)$.  Suppose that there is a real number $\alpha$ with
$1/2<\alpha<1$ such that
\[
\varepsilon<\alpha\sqrt Q, \text{ and }
|D|<(Q+1)(1-\alpha).
\]
Then there exists a set
$U'\subseteq\operatorname{AG}(2,Q)$ such that
\begin{enumerate}
    \item $U\subseteq U'$;
    \item $|U'|=Q$; and 
    \item $D(U')=D(U)$.
\end{enumerate}
\end{theorem}
We will also use the following classical  theorem of R\'edei and
Megyesi; see \cite{Redei} and, for a modern proof,
\cite[Theorem~1.1]{Somlai}.

\begin{theorem}[R\'edei--Megyesi,  see \cite{Somlai}]
\label{thm:redei-direction}
Let $p$ be a prime, and let
$U\subseteq\operatorname{AG}(2,p)=\mathbb F_p^2$ be a set of $p$
points.  If $U$ is not contained in an affine line, then
\[
|D(U)|\ge \frac{p+3}{2}.
\]
\end{theorem}

\begin{proof}[Proof of Lemma \ref{lem:partial-affine-rigidity}]
Consider the affine point set $U_f:=\{(x,f(x)):x\in X\}$.  For distinct $x,y\in X$, the direction determined by $(x,f(x))$ and $(y,f(y))$ is
$\bigl(f(x)-f(y)\bigr)/(x-y)$, and
\begin{equation}\label{eq:square-slopes}
 \chi\left(\frac{f(x)-f(y)}{x-y}\right)=1.
\end{equation}
Thus every finite direction determined by $U_f$ is a nonzero square.  Moreover, $\infty\notin D(U_f)$ because $U_f$ is the graph of a function.  Therefore, the set $D(U_f)$ has the following properties simultaneously: 
\begin{equation}\label{eq:direction-upper-bound}
 D(U_f)\subseteq(\F_p^\times)^2, \ 
 |D(U_f)|\le\frac{p-1}{2}, \ 
\text{ and } \infty\notin D(U_f).
\end{equation}
where the second property is a consequence of the first property. 

Set $\alpha:=(p+2)/(2(p+1))$.  Then $1/2<\alpha<1$ and
\[(p+1)(1-\alpha)=p/2>(p-1)/2\ge |D(U_f)|.\]  Writing $\eps=p-|X|$, we also have $\eps\le\sqrt p/3<\alpha\sqrt p$.  The assumptions in the Sz\H{o}nyi--Sziklai theorem (Theorem \ref{thm:szonyi-sziklai-completion}) therefore are satisfied. Thus it guarantees a set $U'\supseteq U_f$ such that the following two properties hold simultaneously:
\begin{equation}\label{eq:completion-conclusion}
 |U'|=p
\text{ and }
 D(U')=D(U_f).
\end{equation}

Since $\infty\notin D(U')$, no two points of $U'$ have the same first coordinate.  The first-coordinate projection from $U'$ to $\F_p$ is therefore injective, and hence bijective because both sets have $p$ elements.  The extension $U'$ uniquely determines a mapping $f':\F_p\to\F_p$ by 
$U'=\{(x,f'(x)):x\in\F_p\}$.  The containment $U_f\subseteq U'$ gives  us $f'|_X=f$.

If $U'$ were not contained in an affine line,
Theorem~\ref{thm:redei-direction} would give
$|D(U')|\ge(p+3)/2$. However by\eqref{eq:completion-conclusion} and 
\eqref{eq:direction-upper-bound}, \[|D(U')| = |D(U_f)| \leq \frac{p-1}{2}.\]  
A contradiction is reached. Hence $U'$ is contained in an
affine line.

Since $\infty\notin D(U')$, this line is nonvertical and has an
equation $y=ax+b$ for some $a,b\in\F_p$.  Every affine line in
$\operatorname{AG}(2,p)$ has exactly $p$ points, and $|U'|=p$, so
\[
U'=\{(x,ax+b):x\in\F_p\}.
\]
The containment $U_f\subseteq U'$ gives
$f(x)=ax+b$ for every $x\in X$.  Finally, the unique direction
determined by $U'$ is $a$, and
$a\in D(U')=D(U_f)\subseteq(\F_p^\times)^2.$
Thus $a$ is a nonzero square.
\end{proof}

\begin{corollary}
\label{cor:paley-automorphisms}
For every sufficiently large prime $p\equiv1\pmod4$,
\begin{equation}\label{eq:full-aut-group}
\Aut(P_p)=\Phi_p.
\end{equation}
Moreover, Definition~\ref{def:reasonable}\textnormal{(c)} holds for
$P_p$ with $\delta=1/(3\sqrt p)$.
\end{corollary}

\begin{proof}
Taking $X=\F_p$ in
Lemma~\ref{lem:partial-affine-rigidity} gives
\eqref{eq:full-aut-group}.  The same lemma applies whenever
$|X|\ge p-\sqrt p/3$, which is precisely the required extension
property with $\delta=1/(3\sqrt p)$.
\end{proof}

\begin{proof}[Proof of Theorem~\ref{thm:main}, assuming Proposition~\ref{prop:full-fsw}]
Fix $q_0:=1/3$  and $c_0:=1/3$.  Lemmas~\ref{lem:paley-separation}, \ref{lem:paley-prime}, and \ref{lem:partial-affine-rigidity} show that, for every sufficiently large prime $p\equiv1\pmod4$, the graph $P_p$ is $(q_0,(3\sqrt p)^{-1};\Phi_p)$-reasonable.  Equations \eqref{eq:paley-group-size} and \eqref{eq:full-aut-group} verify the remaining hypotheses of Proposition~\ref{prop:full-fsw}.  Hence $P_p$ is a fractalizer.

Let $g:=|\Aut(P_p)|=p(p-1)/2$.  On $p^m$ vertices, the balanced
iterated blow-up is unique up to isomorphism.  If $F_m$ denotes its
number of labeled embeddings of $P_p$ in such an iterated blow-up, then by Lemma \ref{lem:embedding-classification}
\[
F_m=g(p^{m-1})^p+pF_{m-1}.
\]
Let $a_m:=F_m/p^{mp}$.  Dividing the recurrence by $p^{mp}$ gives
\[
a_m=\frac{g}{p^p}+p^{1-p}a_{m-1}.
\]
Since $p^{1-p}<1$, it follows that
\[
a_m\to
\frac{g/p^p}{1-p^{1-p}}
=
\frac{g}{p^p-p}.
\]
Hence $\emb(P_p)=g/(p^p-p)$.  Since every unlabeled induced copy
accounts for exactly $g$ labeled embeddings,
$\ind(P_p)=p!\emb(P_p)/g$, proving
\eqref{eq:inducibility-formulas}.

\end{proof}

It remains to prove Proposition~\ref{prop:full-fsw}.  The principal
difficulty is that the Paley rigidity theorem is available only at
the scale $\delta=\Theta(k^{-1/2})$, whereas the criterion of
\cite{FSW} is designed for
$\delta=\Theta((\log k)^{1/5}k^{-1/5})$.  Equivalently, the latter
permits roughly $(\log k)^{1/5}k^{4/5}$ exceptional vertices, while
our rigidity theorem permits only $O(\sqrt{k})$.

\section{Signatures and small-set embedding estimates}
\label{sec:weighted-hitting}

\subsection{Signatures and disjoint candidate sets}

A signature is a small set of vertices whose adjacency patterns
distinguish any two vertices outside it.  Once the image of a
signature is fixed in the host graph, different remaining vertices
have disjoint sets of possible images.  This is the basic reason that
embedding counts can be bounded by products of candidate-set sizes;
see \cite{FHL} and \cite[Section~3]{FSW}.

Throughout this subsection, $H$ is a fixed deterministic graph on $k$
vertices satisfying Definition~\ref{def:reasonable}\textnormal{(b)}
with parameter $q$.  For distinct vertices $u,v\in V(H)$, let
\[
\Delta_H(u,v):=
\left\{
z\in V(H)\setminus\{u,v\}:
a_H(z,u)\neq a_H(z,v)
\right\}.
\]
Thus $\Delta_H(u,v)$ is the set of vertices that distinguish $u$ and
$v$.  Definition~\ref{def:reasonable}\textnormal{(b)} states that
\[
|\Delta_H(u,v)|\ge qk
\]
for every pair of distinct vertices $u,v$.

\begin{definition}\label{def:signature}
A set $S\subseteq V(H)$ is a \emph{signature} of $H$ if, for every
two distinct vertices $u,v\in V(H)\setminus S$, there exists
$s\in S$ such that
\[
a_H(s,u)\neq a_H(s,v).
\]
Equivalently, $S\cap\Delta_H(u,v)\neq\varnothing$ for all distinct $u,v\in V(H)\setminus S$.
\end{definition}

\begin{lemma}\cite[Lemma 3.2]{FSW}
\label{lem:candidate-disjointness}
Let $S$ be a signature of $H$, let $H'$ be an induced subgraph of $H$
with $S\subseteq V(H')$, and let $G$ be any graph.  Fix a map
\[\psi:S\to V(G).\]  For each $i\in V(H')\setminus S$, let $C_i$ be
the set of vertices $w\in V(G)$ for which there exists an embedding
$\vartheta:H'\hookrightarrow G$ satisfying $\vartheta|_S=\psi$
and  $\vartheta(i)=w.$
Then the sets $C_i$, $i\in V(H')\setminus S$, are pairwise disjoint.
\end{lemma}

\begin{proof}
Suppose, to the contrary, that $w\in C_i\cap C_j$ for two distinct
vertices $i,j\in V(H')\setminus S$.  There may be two different
embeddings witnessing the two memberships, but both embeddings agree
with $\psi$ on $S$.

For every $s\in S$, the embedding witnessing $w\in C_i$ gives
\[
a_G(w,\psi(s))=a_H(i,s),
\]
whereas the embedding witnessing $w\in C_j$ gives
\[
a_G(w,\psi(s))=a_H(j,s).
\]
Hence $a_H(i,s)=a_H(j,s)$ for every $s\in S$.  Thus no vertex of
$S$ distinguishes $i$ and $j$, contradicting the definition of a
signature.
\end{proof}

The next observation shows that every sufficiently large subset of
$V(H)$ is a signature. 
\begin{lemma}\cite[Lemma 3.3]{FSW}
\label{lem:large-signature}
Suppose  $|V(H)| = k$ and 
$H$ satisfies Definition~\ref{def:reasonable}\textnormal{(b)}. Then every set $S\subseteq V(H)$ with $|S|\ge(1-q)k$ is a signature of $H$.
\end{lemma}

\begin{proof}
Suppose that $S$ is not a signature.  Then there are distinct
$u,v\in V(H)\setminus S$ such that no vertex of $S$ distinguishes
$u$ and $v$.  Hence $\Delta_H(u,v)
\subseteq
V(H)\setminus(S\cup\{u,v\}).$
It follows that $|\Delta_H(u,v)| \le k-|S|-2 \le qk-2 <qk,$
contradicting Definition~\ref{def:reasonable}\textnormal{(b)}.
\end{proof}

The next lemma shows that this signature can be logarithmically
small.  The graph $H$ is fixed throughout; randomness is used only as
a probabilistic method for selecting a deterministic signature.  This
is the same observation as \cite[Lemma~3.4]{FSW}, and we include the
proof to emphasize that no random-graph hypothesis is involved.

\begin{lemma}\cite[Lemma~3.4]{FSW}
\label{lem:small-signature}
Assume $H$ satisfies Definition~\ref{def:reasonable}\textnormal{(b)} with $q\in(0,1/2)$.  For all sufficiently large $k=k(q)$, the following
holds.  If $X\subseteq V(H)$ satisfies
$|X|\ge\left(1-\frac q2\right)k,$
then $X$ contains a signature $S$ of $H$ with
$|S|\le\frac5q\log k.$
\end{lemma}

\begin{proof}
Let $t:=\left\lceil\frac4q\log k\right\rceil.$
For all sufficiently large $k$,
$t\le(5/q)\log k$. 
Choose $s_1,\dots,s_t$ independently and uniformly from $X$, allowing
repetitions, and let $S:=\{s_1,\dots,s_t\}.$

Fix distinct vertices $u,v\in V(H)$.  Since
$|\Delta_H(u,v)|\ge qk$ and
$|V(H)\setminus X|\le qk/2$, we have
\[
|\Delta_H(u,v)\cap X|
\ge
|\Delta_H(u,v)|-|V(H)\setminus X|
\ge\frac{qk}{2}.
\]
As $|X|\le k$, a uniformly chosen element of $X$ belongs to
$\Delta_H(u,v)$ with probability at least $q/2$.  Therefore
\[
\begin{aligned}
\mathbb P\bigl(S\cap\Delta_H(u,v)=\varnothing\bigr)
&=
\mathbb P\bigl(s_i\notin\Delta_H(u,v)
               \text{ for every }i\bigr) 
 \le
\left(1-\frac q2\right)^t
\le
\exp\left(-\frac{qt}{2}\right)
\le k^{-2}.
\end{aligned}
\]

There are fewer than $k^2/2$ unordered pairs $\{u,v\}$.  By the union
bound, the probability that some pair is not distinguished by the
sampled set is less than $1/2$.  Hence there exists a choice of
$s_1,\dots,s_t$ for which $S\cap\Delta_H(u,v)\neq\varnothing$ for every pair of distinct vertices $u,v$.  In particular, $S$ is a
signature, and
$|S|\le t\le\frac5q\log k.$
\end{proof}

\begin{definition}\label{def:super-signature}
Let $\rho>0$.  A nonempty set $S\subseteq V(H)$ is a
\emph{$\rho$-super-signature} if
\[
|S\cap\Delta_H(u,v)|\ge\rho|S|
\]
for every two distinct vertices $u,v\in V(H)\setminus S$.
\end{definition}

Later we need a robust fingerprint: every pair should be separated
not merely once, but by a positive proportion of the sampled
vertices.  The same sampling argument, now combined with a Chernoff
bound, produces such a super-signature.  This is the deterministic
statement of \cite[Lemma~3.6]{FSW}. We include the short proof for completeness. 

\begin{lemma}\cite[Lemma~3.6]{FSW}
\label{lem:small-super-signature}
Assume $H$ satisfies Definition~\ref{def:reasonable}\textnormal{(b)} with $q\in(0,1/2)$.  For all sufficiently large $k=k(q)$, the following
holds.  If $X\subseteq V(H)$ satisfies
$|X|\ge\left(1-\frac {q}{2}\right)k$,
then $X$ contains a $(q/4)$-super-signature $S$ satisfying
$|S|\le\frac{33}{q}\log k.$
\end{lemma}

\begin{proof}
Let 
$t:=\left\lceil\frac{32}{q}\log k\right\rceil.$
For all sufficiently large $k$,
$t\le(33/q)\log k$.
Pick $s_1,\dots,s_t$ independently and uniformly from $X$.
Fix distinct vertices $u,v\in V(H)$, and define 
$Z_i:=
\mathbf 1_{\{s_i\in\Delta_H(u,v)\}}.$

All probabilities and expectations below are taken over the
independent uniform choices $s_1,\dots,s_t\in X$; the set $X$ itself
is fixed.
As in the proof of Lemma~\ref{lem:small-signature},
\[
\mathbb{E}[Z_i]
=
\frac{|\Delta_H(u,v)\cap X|}{|X|}
\ge\frac q2.
\]
Thus $Z_1+\cdots+Z_t$ is binomial with mean
$\mathbb{E}[Z_i]t\ge qt/2$.  By the Chernoff bound,
\[
\begin{aligned}
\mathbb P\left(
\sum_{i=1}^t Z_i<\frac{qt}{4}
\right)
&\le
\mathbb P\left(
\sum_{i=1}^t Z_i<\frac{\mathbb{E}[Z_i]t}{2}
\right)
\le
\exp\left(-\frac{\mathbb{E}[Z_i]t}{8}\right)
\le
\exp\left(-\frac{qt}{16}\right)
\le k^{-2}.
\end{aligned}
\]
A union bound over all unordered pairs shows that, with probability
greater than $1/2$, we have, simultaneously for all distinct $u,v\in V(H)$, the bound
$\left|
\left\{i\in[t]:s_i\in\Delta_H(u,v)\right\}
\right|
\ge\frac{qt}{4}$
.

We also claim that the sampled vertices are all distinct with
probability greater than $3/4$.  Indeed, by the union bound,
\[
\mathbb P\bigl(s_i=s_j\text{ for some }i<j\bigr)
\le
\frac{\binom{t}{2}}{|X|}.
\]
For fixed $q$, the right-hand side is
$O_q((\log k)^2/k)$ and is therefore smaller than $1/4$ for
sufficiently large $k$.

It follows that, with positive probability, the vertices
$s_1,\dots,s_t$ are distinct and every pair $u,v$ is distinguished by
at least $qt/4$ of them.  Fix such a choice and set $S:=\{s_1,\dots,s_t\}.$
Then $|S|=t$, and for every two distinct
$u,v\in V(H)\setminus S$, we have $|S\cap\Delta_H(u,v)|
\ge\frac{qt}{4} = \frac q4|S|.$
Thus $S$ is a $(q/4)$-super-signature of the required size.
\end{proof}

\subsection{Counting inequalities}

To keep the proof self-contained at the level of statements, we record
the deterministic counting results from \cite{FSW} that are used
below. Several of these estimates originate from \cite{FHL}.  The
first two are elementary consequences of the disjointness of
candidate sets.

\begin{lemma}\cite[Lemma~3.8]{FSW}
\label{lem:balanced-product-tools}
Let $\ell,\ell'\ge1$, and let $m,m'\ge0$ be integers.
\begin{enumerate}[label=\textnormal{(\roman*)},leftmargin=2.5em]
\item If $m_1,\dots,m_\ell$ are nonnegative integers satisfying
$\sum_{i=1}^{\ell}m_i\le m,$
then
$\prod_{i=1}^{\ell}m_i\le\E_\ell(m).$

\item 
$\E_\ell(m)\E_{\ell'}(m')
\le
\E_{\ell+\ell'}(m+m').$

\item If $m\ge\ell$, $\ell'\le\ell$, and
$m'\le(1-\mu)m$ for some $\mu\ge0$, then
\[
\E_{\ell'}(m')
\le
\exp\bigl(3(\ell-\ell')-\mu\ell/2\bigr)
\left(\frac{\ell}{m}\right)^{\ell-\ell'}
\E_\ell(m).
\]
\end{enumerate}
\end{lemma}

The following is \cite[Lemma~3.9]{FSW}.  It bounds the number of
embeddings extending a fixed image of a signature. It is a simple corollary of Lemma \ref{lem:candidate-disjointness}.

\begin{lemma}\cite[Lemma~3.9]{FSW}
\label{lem:fixed-signature-extension}
Let $G$ be a graph, let $U\subseteq V(G)$, and let
$S\subseteq V(H)$ be a signature.  Fix a map
$f:S\to V(G)$.  Then the number of embeddings
$\vartheta:H\hookrightarrow G$ satisfying
$\vartheta|_S=f $ and 
$\vartheta(v)\in U$
for every $v\in V(H)\setminus S$
is at most $\E_{k-|S|}(|U|).$
\end{lemma}

The next lemma bounds embeddings sending many pattern vertices
into a small host set. It improves
\cite[Corollary~3.11]{FSW} in the parameter range needed here;
the quantitative comparison is given in
Remark~\ref{rem:aggregate-small-set-comparison}. 
 For the parameters arising below, $\gamma/\beta^2$ tends to infinity, so \cite[Corollary~3.11]{FSW} gives no useful decay, whereas $\gamma/\beta$ tends to zero.  The estimate below therefore gives exponential decay in a range not covered by the corresponding bound in \cite{FSW}.

\begin{lemma}
\label{lem:weighted-hitting}
Let $H$ be a $k$-vertex graph satisfying
Definition~\ref{def:reasonable}\textnormal{(b)} with parameter
$0<q<1/2$, and let $H'$ be an induced subgraph of $H$ on
$k'\ge k-4$ vertices.  Let $G$ be an $n$-vertex graph with $n\ge k$,
and let $U\subseteq V(G)$ satisfy $|U|\le\gamma n$, where
$0<\gamma<1$.  Suppose that $0<\beta\le q/3$.

For all sufficiently large $k=k(q)$, there is a signature
$T\subseteq V(H')$ with $t:=|T|\le(5/q)\log k$.  Write
$r:=\lceil\beta k\rceil-t$ and $d_0:=k-k'+t$.  Then
$k'-t-r\ge1$ for all sufficiently large $k$.  

If $r\ge1$, the number
$N$ of embeddings $\vartheta:H'\hookrightarrow G$ satisfying
$|\vartheta^{-1}(U)|\ge\beta k$ is at most
\begin{equation}
\label{eq:weighted-hitting-exact}
N\le
 e^{3d_0}k^{d_0}
 \left(\frac{e^4\gamma k}{r}\right)^r
 \frac{\E_k(n)}{n^{k-k'}}.
\end{equation}

If, in addition, $\beta k\ge2t+2$ and
$2e^4\gamma/\beta<1$, then
\begin{equation}
\label{eq:weighted-hitting-simple}
N\le
 \exp\left(\frac{40}{q}(\log k)^2\right)
 \left(\frac{2e^4\gamma}{\beta}\right)^{\beta k/2}
 \frac{\E_k(n)}{n^{k-k'}}.
\end{equation}
\end{lemma}

\begin{proof}
Since $k'\ge k-4$, we have $k'\ge(1-q/2)k$ for all sufficiently
large $k=k(q)$.  Lemma~\ref{lem:small-signature}, applied with
$X=V(H')$, gives a signature $T\subseteq V(H')$ with
$t\le(5/q)\log k$.  Since $t+r=\lceil\beta k\rceil$ and
$\beta\le q/3$, we have
$k'-t-r=k'-\lceil\beta k\rceil\ge(1-q/3)k-5\ge1$ for all
sufficiently large $k$.

Assume first that $r\ge1$. Choose a map $\xi:T\to V(G)$ for
which the number of embeddings counted by $N$ extending $\xi$
is maximal.  For each $v\in V(H')\setminus T$, let
$C_v$ be the set of possible images of $v$ among embeddings of $H'$
extending $\xi$.  By Lemma~\ref{lem:candidate-disjointness}, the
sets $C_v$ are pairwise disjoint.  Write  $a_v:=|C_v\cap U|$.  Then
\begin{equation}
    \sum_{v\in V(H')\setminus T}a_v\le |U|\le\gamma n. \label{eq:sumav}
\end{equation}

Let $N_\xi$ denote the number of embeddings counted by $N$ that
extend the fixed map $\xi$.  Every such embedding maps at least
$\lceil\beta k\rceil$ vertices into $U$, at most $t$ of which lie in
$T$.  It therefore maps at least $r$ vertices of
$V(H')\setminus T$ into $U$.

For each set $J\subseteq V(H')\setminus T$ with $|J|=r$, consider
the embeddings extending $\xi$ for which
$\vartheta(v)\in U$ for every $v\in J$.  Every embedding counted by
$N_\xi$ belongs to at least one of these families.  Notice that the
vertices outside $J$ are unrestricted and may include additional
vertices mapped into $U$; thus embeddings with more than $r$ such
vertices may be counted more than once, which is harmless for an
upper bound.

For a fixed $J$, there are at most
$\prod_{v\in J}a_v$ choices for the images of the vertices in $J$.
After those images have been fixed, the admissible image sets of the
remaining vertices are pairwise disjoint and have total size at most
$n$.  Hence there are at most $\E_{k'-t-r}(n)$ choices for the
remaining images.  Therefore
\[
N_\xi
\le
\E_{k'-t-r}(n)
\sum_{\substack{J\subseteq V(H')\setminus T\\ |J|=r}}
\prod_{v\in J}a_v.
\]
By the choice of $\xi$, we have $N\le n^tN_\xi$, and hence \begin{equation}
\label{eq:fixed-xi-union}
N\le n^t\E_{k'-t-r}(n)
 \sum_{\substack{J\subseteq V(H')\setminus T\\ |J|=r}}
 \prod_{v\in J}a_v \leq n^t e^{3(d_0+r)}\left(\frac{k}{n}\right)^{d_0+r}\E_k(n)\sum_{\substack{J\subseteq V(H')\setminus T\\ |J|=r}}
 \prod_{v\in J}a_v,
\end{equation}
where the last inequality follows from
Lemma~\ref{lem:balanced-product-tools}\textnormal{(iii)},  applied
with $\ell=k$, $\ell'=k'-t-r$, $m=m'=n$, and $\mu=0$,

 Since $1\le r< |V(H')\setminus T|$,
Maclaurin's inequality, $\binom{|V(H')\setminus T|}{r} \le(e |V(H')\setminus T|/r)^r$, and (\ref{eq:sumav}) give
\begin{equation}
\label{eq:elementary-symmetric-bound}
 \sum_{\substack{J\subseteq V(H')\setminus T\\ |J|=r}}
 \prod_{v\in J}a_v
 \le \binom{|V(H')\setminus T|}{r}\left(\frac{1}{|V(H')\setminus T|}\sum_{v\in V(H')\setminus T}a_v\right)^r
 \le\left(\frac{e\gamma n}{r}\right)^r.
\end{equation}

Substitution into \eqref{eq:fixed-xi-union} yields
\[
 N\le e^{3d_0}k^{d_0}
 \left(\frac{e^4\gamma k}{r}\right)^r
 n^{t-d_0}\E_k(n).
\]
Since $t-d_0=-(k-k')$, this proves
\eqref{eq:weighted-hitting-exact}.

Assume now that $\beta k\ge2t+2$.  Then
$r=\lceil\beta k\rceil-t>\beta k/2$, so
$e^4\gamma k/r\le2e^4\gamma/\beta$.  
Since $0<2e^4\gamma/\beta<1$, we obtain
$(e^4\gamma k/r)^r\le (2e^4\gamma/\beta)^{\beta k/2}$.  Finally,
$d_0\le4+(5/q)\log k$, and therefore
$e^{3d_0}k^{d_0}\le\exp((40/q)(\log k)^2)$ for all sufficiently
large $k$.  Combining these estimates with
\eqref{eq:weighted-hitting-exact} proves
\eqref{eq:weighted-hitting-simple}.
\end{proof}

The exact bound \eqref{eq:weighted-hitting-exact} is valid without
assuming any ordering between $\gamma$ and $\beta$, although it may
be vacuous when $\gamma$ is comparable to or larger than $\beta$.
The simplified bound \eqref{eq:weighted-hitting-simple} is a useful
rare-event estimate: its hypothesis
$\frac{2e^4\gamma}{\beta}<1$
forces $\gamma\ll\beta$.  In both applications below one has
$\gamma/\beta\to0$.

\begin{remark}[Comparison with the estimate of \cite{FSW}]
\label{rem:aggregate-small-set-comparison}

For a specified set $J$ of $\lceil\beta k\rceil$ pattern vertices,
\cite[Lemma~3.10]{FSW} gives a bound whose main exponential factor
is
$\bigl(e^4\gamma/\beta\bigr)^{\beta k}$.
To count embeddings for which at least $\beta k$ vertices are mapped
into $U$, without specifying these vertices in advance,
\cite[Corollary~3.11]{FSW} applies this bound to every possible
choice of $J$.  The number of choices is bounded by
$\binom{k'}{\lceil\beta k\rceil}
 \le(e^2/\beta)^{\beta k}$.
This produces the factor
$\bigl(e^6\gamma/\beta^2\bigr)^{\beta k}$.

The proof above still sums over the possible sets $J$, but it does
so before replacing their contributions by a common worst-case
bound.  For a fixed image of the signature, the contribution of
$J$ is bounded by $\prod_{v\in J}a_v$, where
$a_v=|C_v\cap U|$.  Maclaurin's inequality estimates the aggregate
sum as $\sum_{|J|=r}\prod_{v\in J}a_v
\le \left(\frac{e\gamma n}{r}\right)^r.$
Because $r$ is of order $\beta k$, this leaves only one power of
$\beta^{-1}$ in the base.  The factor $1/2$ lost in the exponent of
\eqref{eq:weighted-hitting-simple} comes from the possibility that
some vertices mapped into $U$ belong to the signature.

This improvement is essential at the rigidity scale available for
the Paley graph.  In both applications below, the relevant quantities
satisfy
$\gamma\asymp_q(\log k)^2/k$ and
$\beta\asymp_{q,c}1/(\sqrt{k}\log k)$.  Hence
$\gamma/\beta^2\asymp_{q,c}(\log k)^4$, which tends to infinity,
whereas
$\gamma/\beta\asymp_{q,c}(\log k)^3/\sqrt{k}$, which tends to zero.
Thus the estimate of \cite[Corollary~3.11]{FSW} gives no decay at the
parameter scale forced by our partial-rigidity theorem, while
Lemma~\ref{lem:weighted-hitting} gives the required exponential
saving.  The precise parameter choices are made in the next section.
\end{remark}

\section{Choosing the error parameters at the square-root scale}

The stability argument in the next section uses five error
parameters.  Roughly speaking, $\epsilon_1$ is the tolerance used
when assigning a host vertex to a prospective blow-up part;
$\epsilon_2$ is the proportion of vertices on which an embedding is
allowed to fail to follow an automorphism; $\epsilon_3$ and
$\epsilon_4$ are the two thresholds used in the decomposition of
embeddings; and $\epsilon_5$ is the threshold in the cleaning step.
Their precise roles will be given when they first appear.

Partial rigidity requires $\epsilon_2<\delta$, so we choose
$\epsilon_2$ of order $k^{-1/2}$. The parameters $\epsilon_3,\epsilon_4,\epsilon_5$ must be smaller still, while
remaining large enough for the relevant counting and cleaning
estimates. In contrast, $\epsilon_1=q/3$ is fixed. Lemma~\ref{lem:weighted-hitting} makes these requirements
compatible.

We use the same five levels of error as in
\cite[Sections~3--7]{FSW}, but at different scales.  The two
conditions from \cite{FSW} that are incompatible with the
square-root rigidity scale are precisely the conditions replaced by
Lemma~\ref{lem:weighted-hitting}, as explained in the preceding
remark.  The next lemma records one choice that satisfies all the
conditions needed later.

\begin{lemma}[Parameter choice]
\label{lem:parameters}
Fix $0<q\le10^{-20}$ and $c>0$.  Define
\begin{equation}
\label{eq:parameter-constants}
 a_2:=\min\left\{\frac c4,\frac q{10^4}\right\},
 \qquad
 a_3:=\frac{a_2}{10^4},
 \qquad
 a_4:=\frac{qa_3}{100}.
\end{equation}
Let
\begin{equation}
\label{eq:parameter-definitions}
 \epsilon_1:=\frac q3,
 \epsilon_2:=\frac{a_2}{\sqrt k},
 \epsilon_3:=\frac{a_3}{\sqrt k\,\log k},
 \epsilon_4:=\frac{a_4}{\sqrt k\,\log k},
 \epsilon_5:=\frac{100(\log k)^2}{qk}.
\end{equation}
For all sufficiently large $k=k(q,c)$, the following assertions
hold.

\begin{enumerate}[label=\textnormal{(P\arabic*)},leftmargin=3em]
\item
$0<\epsilon_5<\epsilon_4<\epsilon_3\le\epsilon_2/2$ and
$\epsilon_2<\epsilon_1<q/2$.  Moreover,
$\epsilon_i<1/100$ for $2\le i\le5$ and  $\epsilon_3\le10^{-20}$. 

\item
$\epsilon_2\log(1/\epsilon_2)
 <(\log2)\epsilon_1/8$ and
$\epsilon_3\log(1/\epsilon_3)<\epsilon_2/100$.

\item
If $\delta\ge c/\sqrt k$, then $\epsilon_2<\delta$.

\item
$\epsilon_4<(q/20)\epsilon_3$.

\item
$\epsilon_5<q/10^4$ and
$
\label{eq:growth-hierarchy}
 \frac{\epsilon_i k}{(\log k)^2}\to\infty$ for $i=1, 2,3,4$ as  $k\to\infty$, with $q,c$ fixed.
In particular,
$ \epsilon_5>
 \frac{40}{q}\frac{(\log k)^2}{k}
 >
 \frac{10^3}{q}\frac{\log k}{k}
 >
 \frac{10^3}{k}.$

\item
Let
\begin{equation}
\label{eq:two-tail-parameters}
 \gamma_1:=\frac{24}{q}\frac{(\log k)^2}{k},
 \qquad
 \beta_1:=\frac{\epsilon_3}{3},
 \qquad
 \gamma_2:=\epsilon_5,
 \qquad
 \beta_2:=\frac{\epsilon_4}{2}.
\end{equation}
For $i=1,2$, one has
$\beta_i k\ge(10/q)\log k+2$ and
$2e^4\gamma_i/\beta_i<1$.  Moreover,
\begin{align}
 \exp\left(\frac{40}{q}(\log k)^2\right)
 \left(\frac{2e^4\gamma_1}{\beta_1}\right)^{\beta_1k/2}
 &\le k^{-7},
 \label{eq:first-tail}\\
 \exp\left(\frac{40}{q}(\log k)^2\right)
 \left(\frac{2e^4\gamma_2}{\beta_2}\right)^{\beta_2k/2}
 &\le k^{-8}.
 \label{eq:second-tail}
\end{align}
For $i=1,2$, one has
$0<\gamma_i<1$ and $0<\beta_i\le q/3$.
\item
The parameters used in the collision estimate satisfy
$qk\ge3$, $3\epsilon_5<q/3$, and
$qk/3\ge(20/q)(\log k)^2$.  They also satisfy $\left({9e^4\epsilon_5}/{q}\right)^{qk/3}
 \le k^{-9}.$

\end{enumerate}
\end{lemma}

The verification is given in
Appendix~\ref{app:parameter-choice}.

\section{Local stability at the top level}

The purpose of this section is to turn the abundance of induced copies
in an extremal host into an approximate top-level blow-up.  We first
fix the image of a small signature and obtain disjoint candidate
classes for the remaining labels.  We then clean these classes and
show that almost every embedding approximately follows an
automorphism.  Finally, extremality makes the resulting global
classes be a partition of the host.

This local stability argument follows the architecture of
\cite[Sections~4--7]{FSW}.  We record in Appendix~\ref{app:inherited-local-stability} the exact
lemmas that carry over without a new
argument, and verify their hypotheses at the points where they are
used.  Two counting steps require a different proof in the present
parameter regime and we include full proofs of these.

Before beginning the stability argument, we record the one global
extremal principle from \cite{FSW}.  It says that no vertex of an
extremal host can participate in far fewer copies than average:
deleting such a vertex and cloning a more productive one would improve
the graph.
The following is \cite[Lemma~3.14]{FSW}, originating in \cite[Lemma~4.4]{FHL}. The proof is very short and we omit the proof here. 

\begin{lemma}\cite[Lemma~4.4]{FHL}
\label{lem:extremal-vertex-coverage}
Let $H$ have $k$ vertices, let $n\ge2$, and let $\Gamma$ be an
$n$-vertex graph satisfying
$\emb(H,\Gamma)=\emb(H,n).$
Every vertex of $\Gamma$ belongs to the image of at least
$\frac{k}{n+k}\emb(H,n)$
embeddings $H\hookrightarrow \Gamma$.
\end{lemma}

We now begin the proof of Proposition~\ref{prop:full-fsw}.  Fix $q_0,c_0$, and set 
$q:=\min\{q_0,10^{-20}\}$. Since $q\le q_0$, condition~\textnormal{(b)} in
Definition~\ref{def:reasonable} with parameter $q_0$ implies the same
condition with parameter $q$, while conditions~\textnormal{(a)} and
\textnormal{(c)} are unchanged.  Hence $H$ is also
$(q,\delta;\Phi)$-reasonable.

From now on, all auxiliary parameters are chosen using this smaller
separation parameter $q$.  Let $k$ be sufficiently large in terms of
$q$ and $c_0$, and choose
$\epsilon_1,\dots,\epsilon_5$ from
Lemma~\ref{lem:parameters} with $c=c_0$.  Write $g:=|\Phi|$.

By the hypotheses of Proposition~\ref{prop:full-fsw},
\begin{equation}\label{eq:abstract-group-size}
2\le g\le k^2.
\end{equation}
Moreover, taking $X=V(H)$ in
Definition~\ref{def:reasonable}\textnormal{(c)} shows that
$\Phi=\Aut(H)$.  In the proof below, however, we use only that every
$\phi\in\Phi$ is an automorphism of $H$ and the two cardinality bounds
in \eqref{eq:abstract-group-size}; we do not assume that $g=k^2$. 

Let $n\ge k$, and let $\Gamma$ be an $n$-vertex graph with
$\emb(H,\Gamma)=\emb(H,n)$. We first lower bound $\emb(H,\Gamma)$. A balanced one-level blow-up of $H$ gives, for each $\phi\in\Phi$, exactly $\E_k(n)$ embeddings prescribed by $\phi$.  These classes are disjoint, and hence
\begin{equation}\label{eq:abstract-lower-bound}
 \emb(H,n)\ge g\E_k(n).
\end{equation}
Together with Lemma~\ref{lem:extremal-vertex-coverage} \cite[Lemma~3.14]{FSW}, every vertex of $\Gamma$ belongs to at least 
\begin{align}
    k\emb(H,n)/(n+k) \geq kg\E_k(n)/(n+k)\ge\E_k(n)/n \label{eq:everyvertex}
\end{align} where the first inequality is by  \eqref{eq:abstract-lower-bound}, and the last inequality is by $n\ge k$, and $g\ge2$. 

We use a different notation for actual adjacency and for majority adjacency.
For distinct vertices $u,v\in V(\Gamma)$, let
\[
a_\Gamma(u,v):=
\begin{cases}
1,&uv\in E(\Gamma),\\
0,&uv\notin E(\Gamma).
\end{cases}
\]
For a graph $G$ and nonempty disjoint sets $A,B\subseteq V(G)$, denote
\[
d_G(A,B)
:=
\frac{|\{(a,b)\in A\times B:ab\in E(G)\}|}{|A||B|}
\]
and define
\[
m_G(A,B)
:=
\begin{cases}
1,&d_G(A,B)\ge1/2,\\
0,&d_G(A,B)<1/2.
\end{cases}
\]
We use this notation below with $G=\Gamma$.

\subsection{Cleaning and candidate sets $V_i$}
For a vertex $w\notin B$, write
$m_\Gamma(w,B):=m_\Gamma(\{w\},B)$. And clearly for two vertices $u, v$, $m_\Gamma(u, v) = a_\Gamma(u,v)$. 

We now anchor the local structure by fixing the image of a small
signature.  Once the image of the signature has been fixed, each remaining label
$i$ has a set $V_i'$ of possible images.  Because the signature
distinguishes every two labels, a host vertex can belong to at most one
such candidate set.  Thus the sets $V_i'$ are the first, deliberately
rough, approximation to the top-level parts of a blow-up.  The rest of
the section cleans this approximation and then makes it global.

By Lemma \ref{lem:small-signature}, choose a signature $S\subseteq V(H)$ with
\[
s:=|S|\le\frac5q\log k.
\]
By averaging, there is a map
\[
\psi:S\to V(\Gamma)
\]
extended by at least $\emb(H,\Gamma)/n^s$ embeddings.  For
$i\in V(H)\setminus S$, let $V_i'$ be the set of possible images of
$i$ among embeddings extending $\psi$.  By
Lemma~\ref{lem:candidate-disjointness}, the sets $V_i'$ are pairwise
disjoint.

The choice of $\psi$ and
\eqref{eq:abstract-lower-bound} give together with $g\ge 2$
\begin{equation}\label{eq:local-extension-lower-bound}
\#\{\vartheta:H\hookrightarrow\Gamma:\vartheta|_S=\psi\}
\ge
2\frac{\E_k(n)}{n^s}.
\end{equation}

The candidate classes $V_i'$ are disjoint, but their mutual
adjacencies may still be noisy.  We use the following cleaning
operation, stated in a general form that will also be used in the appendix.

\begin{definition}[Cleaning candidate sets $V_i$]
\label{def:bad-pairs}
Let $F$ and $G$ be graphs, let $T\subseteq V(F)$, let
$(C_i')_{i\in V(F)\setminus T}$ be a family of pairwise disjoint
subsets of $V(G)$, and let $\eta>0$.  For distinct
$i,j\in V(F)\setminus T$, a pair
$(u,v)\in C_i'\times C_j'$ is \emph{bad} if
$a_G(u,v)\ne a_F(i,j)$.  A vertex $u\in C_i'$ is \emph{bad} if it
has at least $\eta|V(G)|$ bad partners.  Let $C_i$ be obtained from
$C_i'$ by deleting all bad vertices, and let
\[
 C_{\mathrm{rest}}
 :=V(G)\setminus\bigcup_{i\in V(F)\setminus T}C_i.
\]
\end{definition}

We apply Definition~\ref{def:bad-pairs} with
$F=H$, $G=\Gamma$, $T=S$, $C_i'=V_i'$, and
$\eta=\epsilon_5$.  We write $V_i$ and $V_{\mathrm{rest}}$ for the
resulting cleaned classes and exceptional set.  From this point on,
we fix the pair $(S,\psi)$ used to construct these sets.  The
embeddings considered later are not required to extend $\psi$; the
pair $(S,\psi)$ serves only to produce the disjoint test classes
$V_i'$ and their cleaned subsets $V_i$.

The following four consequences of the cleaning step are the only
ones needed later.  They control, respectively, the cost of deleting
a bad vertex, the number of embeddings that survive, the size of the
exceptional region, and the coherence of the majority-adjacency
pattern.

\begin{lemma}[Properties of the cleaned candidate sets]
\label{lem:cleaning-package}
For all sufficiently large $k=k(q,c_0)$, the following statements
hold.
\begin{enumerate}[label=\textnormal{(\roman*)},leftmargin=2.5em]
\item
If $i\in V(H)\setminus S$ and $v\in V_i'\setminus V_i$, then at most
$\E_k(n)/n^{s+1}$ embeddings $\vartheta:H\hookrightarrow\Gamma$
extending $\psi$ satisfy $\vartheta(i)=v$.

\item
At least $\E_k(n)/n^s$ embeddings extending $\psi$ map every
$i\in V(H)\setminus S$ into $V_i$.  Consequently, every $V_i$ is
nonempty and
\begin{equation}
\label{eq:cleaned-product-lower}
 \prod_{i\in V(H)\setminus S}|V_i|
 \ge\frac{\E_k(n)}{n^s}.
\end{equation}

\item
For every $i\in V(H)\setminus S$,
\begin{equation}
\label{eq:Vi-Vrest-bound}
 |V_i|+|V_{\mathrm{rest}}|
 \le\frac{24}{q}\frac{(\log k)^2}{k}\,n.
\end{equation}

\item
If $i_1,\ldots,i_r,j_1,\ldots,j_r$ are distinct elements of
$V(H)\setminus S$ and
$m_\Gamma(V_{i_\ell},V_{j_\ell})\ne a_H(i_\ell,j_\ell)$ for every
$1\le\ell\le r$, then
\begin{equation}
\label{eq:few-wrong-part-pairs}
 r\le\frac{20}{q}(\log k)^2.
\end{equation}
\end{enumerate}
\end{lemma}

\begin{proof}
We apply Lemma~\ref{lem:deterministic-cleaning} with
$F=H$, $G=\Gamma$, $m=k$, $N=n$, $\rho=q$, $T=S$, $t=s$,
$\xi=\psi$, $C_i'=V_i'$, and $\eta_5=\epsilon_5$.  The resulting
sets $C_i$ and $C_{\mathrm{rest}}$ are precisely $V_i$ and
$V_{\mathrm{rest}}$.  Condition~\textnormal{(C1)} is
Lemma~\ref{lem:candidate-disjointness}, and
condition~\textnormal{(C2)} is
\eqref{eq:local-extension-lower-bound}.

For condition~\textnormal{(C3)}, let $L:=\log k$.  Since
$s\le(5/q)L$ and $\epsilon_5k=100L^2/q$, for all sufficiently large
$k$ we have
$(3+L)(s+1)\le(12/q)L^2\le\epsilon_5k/2$ and
$(3+L)s/\log2\le(20/q)L^2$.  Finally,
$0<\epsilon_5<1$ by
Lemma~\ref{lem:parameters}\textnormal{(P1)}.  Thus every hypothesis
of Lemma~\ref{lem:deterministic-cleaning} holds, and its conclusions
are exactly statements~\textnormal{(i)}--\textnormal{(iv)} above.
\end{proof}

 The cleaned candidate classes $V_i$ describe a tentative blow-up structure.
We call an embedding loyal if, for some automorphism $\phi$, all but
a small proportion of its vertices are mapped to the classes
prescribed by $\phi$.  Later, this approximate alignment will be
upgraded to an exact partition of the host graph $\Gamma$.

\begin{definition}[Loyal embedding relative to candidate classes]
\label{def:relative-loyalty}
Let $F$ and $G$ be graphs, let $T\subseteq V(F)$, and let
$\mathcal C=(C_i)_{i\in V(F)\setminus T}$ be a family of subsets of
$V(G)$.  Let $0\le\eta<1$ and let $\sigma\in\Aut(F)$.  An embedding
$\theta:F\hookrightarrow G$ is \emph{$\sigma$-loyal to
$\mathcal C$ with error $\eta$} if its image contains at most one
vertex from each $C_i$ and $\theta(x)\in C_{\sigma(x)}$ for at least
$(1-\eta)|V(F)|$ vertices $x$ satisfying $\sigma(x)\notin T$.

For $\Psi\subseteq\Aut(F)$, the embedding is
\emph{$\Psi$-loyal to $\mathcal C$ with error $\eta$} if it is
$\sigma$-loyal to $\mathcal C$ with error $\eta$ for some $\sigma\in\Psi$; otherwise it is
\emph{$\Psi$-disloyal}.
\end{definition}

For the remainder of the proof, loyalty is understood relative to
$F=H$, $G=\Gamma$, $T=S$, $\mathcal C=(V_i)_{i\in V(H)\setminus S}$,
$\eta=\epsilon_2$, and $\Psi=\Phi$.  We therefore write simply
\emph{$\phi$-loyal}, \emph{loyal}, and \emph{disloyal}.

The candidate classes $V_i$ arise from a single fixed map $\psi$, so
they need not cover $V(\Gamma)$.  We use them instead as local
reference classes to define global sets $W_j$: a vertex of $\Gamma$
is assigned to $W_j$ when its majority adjacency pattern to the
classes $V_i$ agrees with the adjacency pattern of $j$ in $H$.  The sets $W_j$ should therefore be viewed as the prospective
top-level parts of the extremal blow-up.  We will prove in
Lemma~\ref{lem:W-cover} that they cover $V(\Gamma)$; together with
Lemma~\ref{lem:W-disjoint}, this will show that they form a
partition.

\begin{definition}
  For $j\in V(H)$, let $W_j$ be the set of vertices $w\in V(\Gamma)$ for which at least $(1-\epsilon_1)k$ vertices
$i\in V(H)\setminus(S\cup\{j\})$ satisfy
$w\notin V_i$ and $m_\Gamma(w,V_i)=a_H(j,i)$. 
\end{definition}
\begin{lemma}\label{lem:W-disjoint}
The sets $W_j$, $j\in V(H)$, are pairwise disjoint.
\end{lemma}

\begin{proof}
Suppose that $w\in W_j\cap W_{j'}$ for distinct $j,j'$.  By the
definition of the two sets, at least
$(1-2\epsilon_1)k$
vertices $i\in V(H)\setminus(S\cup\{j,j'\})$ satisfy
$a_H(j,i)
=
m_\Gamma(w,V_i)
=
a_H(j',i).$
Thus at most $2\epsilon_1k$ vertices distinguish $j$ and $j'$.
Since $2\epsilon_1<q$, this contradicts
Definition~\ref{def:reasonable}\textnormal{(b)}.
\end{proof}

Loyalty records approximate agreement with a symmetry using the local
candidate classes $V_i$.  We now introduce regulation, which records
exact agreement with the global classes $W_j$: every vertex of the
embedding must lie in the class prescribed by the symmetry.
Corollary~\ref{cor:loyal-unregulated} will show that, for each fixed
$\phi$, only a negligible number of $\phi$-loyal embeddings fail to
be $\phi$-regulated, i.e., the gap
between these two notions is negligible.
 
 \begin{definition}\label{def:regulated}
An embedding is \emph{$\phi$-regulated} if
$\vartheta(x)\in W_{\phi(x)}$ for every $x\in V(H)$, and is regulated if it is $\phi$-regulated for some $\phi\in\Phi$.  
\end{definition}

We now show that almost every embedding respects the global classes
$W_j$.  We first treat embeddings that already follow a fixed
automorphism on almost all vertices, and then show that embeddings
following no automorphism are rare.  Together, these estimates will
allow us to prove that the sets $W_j$ cover the host graph.

\subsection{Embeddings following a fixed automorphism}

Fix $\phi\in\Phi$.  A $\phi$-loyal embedding already follows the
classes $V_j$ predicted by $\phi$ on almost all vertices.  This subsection
shows that, with very few exceptions, this approximate agreement
forces every vertex to lie in the corresponding global class
$W_{\phi(x)}$.

\begin{lemma}
\label{lem:fixed-automorphism-defect}
Let $\phi\in\Phi$, let $x,x'\in V(H)$ be distinct, and let
$w,w'\in V(\Gamma)$.  If $w\notin W_{\phi(x)}$, then there are at
most $k^{-8}\E_k(n)/n^2$ $\phi$-loyal embeddings
$\vartheta:H\hookrightarrow\Gamma$ satisfying
$\vartheta(x)=w$ and $\vartheta(x')=w'$.
\end{lemma}
This lemma is  \cite[Lemma~5.2]{FSW} (see Lemma~\ref{lem:deterministic-fixed-defect}). The proof of this lemma in \cite{FSW} is simple but important. Roughly speaking, if
$w\notin W_{\phi(x)}$, then $w$ has the wrong majority relation to
 linearly many candidate classes $V_j$ expected from the label $\phi(x)$.
A $\phi$-loyal embedding uses almost all of
these classes, and each remaining disagreement forces a choice from
the minority side of the corresponding class.  These losses multiply
to give the required exponential bound.

\begin{proof}[Proof of Lemma \ref{lem:fixed-automorphism-defect}]
We apply Lemma~\ref{lem:deterministic-fixed-defect} with
$F=H$, $G=\Gamma$, $m=k$, $N=n$, $\rho=q$,
$\eta_1=\epsilon_1$, $\eta_2=\epsilon_2$, $T=S$,
$C_i=V_i$, $D_j=W_j$, and $\sigma=\phi$.  We take
$a=x$, $a'=x'$, $b=w$, and $b'=w'$. The standing assumption $n\ge k$ gives $N\ge m$.

The sets $V_i$ are nonempty by
Lemma~\ref{lem:cleaning-package}\textnormal{(ii)} and pairwise
disjoint because $V_i\subseteq V_i'$.  Pair separation is
Definition~\ref{def:reasonable}\textnormal{(b)}, and
$\phi\in\Aut(H)$.  Lemma~\ref{lem:parameters}\textnormal{(P1)}
gives $\epsilon_2<q/2$ and
$\epsilon_2<1/100<e^{-4}$; part~\textnormal{(P2)} gives
$2\epsilon_2\log(1/\epsilon_2)<(\log2)\epsilon_1/4$; and
part~\textnormal{(P5)} gives $\epsilon_2k\ge16$ for all
sufficiently large $k$.

It remains to check the intersection and decay assumptions.  From
the definitions of the parameters and $s\le(5/q)\log k$, we have
$\epsilon_2\le\epsilon_1/4$ and
$s+5\le\epsilon_1k/4$ for all sufficiently large $k$.  Hence
$(\epsilon_1-\epsilon_2)k-s-4\ge\epsilon_1k/2$.  Finally,
$(\log2)\epsilon_1k/4\ge10\log k+6$ for all sufficiently large
$k$, and therefore
$e^6k^2\exp(- (\log2)\epsilon_1k/4)\le k^{-8}$.  All hypotheses of
Lemma~\ref{lem:deterministic-fixed-defect} are satisfied.
\end{proof}

\begin{corollary}\label{cor:one-fixed-image-defect}
Let $\phi\in\Phi$, $x\in V(H)$, and
$w\notin W_{\phi(x)}$.  Then at most
$k^{-8}\frac{\E_k(n)}{n}$
$\phi$-loyal embeddings satisfy $\vartheta(x)=w$.
\end{corollary}

\begin{proof}
Choose any $x'\ne x$ and sum
Lemma~\ref{lem:fixed-automorphism-defect} over the at most $n$
possible values of $\vartheta(x')$.
\end{proof}

The next corollary says there are only few loyal but unregulated embeddings. 
\begin{corollary}
\label{cor:loyal-unregulated}
For each fixed $\phi\in\Phi$, at most
$k^{-7}\E_k(n)$
embeddings are $\phi$-loyal but not $\phi$-regulated.
\end{corollary}

\begin{proof}
Such an embedding satisfies
$\vartheta(x)\notin W_{\phi(x)}$ for some $x\in V(H)$.  For fixed
$x$, summing Corollary~\ref{cor:one-fixed-image-defect} over the at
most $n$ possible values of $\vartheta(x)$ gives at most
$k^{-8}\E_k(n)$ embeddings.  Summing over $x$ gives the result.
\end{proof}

The next corollary says that very few loyal embeddings using vertices not covered by $\bigcup W_j$, if exists. 
\begin{corollary}
\label{cor:uncovered-loyal} Assume $|\Phi|\le k^2$. 
If $w\notin W_j$ for every $j\in V(H)$, then $w$ belongs to the
image of at most
$k^{-5}\frac{\E_k(n)}{n}$
loyal embeddings.
\end{corollary}

\begin{proof}
Fix $\phi\in\Phi$.  For every $x\in V(H)$, we have
$w\notin W_{\phi(x)}$, so
Corollary~\ref{cor:one-fixed-image-defect} gives at most
$k^{-8}\E_k(n)/n$ $\phi$-loyal embeddings with
$\vartheta(x)=w$.  Summing over $x$ gives
$k^{-7}\E_k(n)/n$ for this $\phi$.  Finally sum over
$|\Phi|\le k^2$.
\end{proof}

The next corollary says few loyal embeddings contain a wrong pair. 
\begin{corollary}
\label{cor:wrong-pair-loyal} Assume $|\Phi|\le k^2$. 
Let $j\ne j'$, let $w\in W_j$ and $w'\in W_{j'}$, and suppose that
$a_\Gamma(w,w')\ne a_H(j,j').$
Then at most
$k^{-4}\frac{\E_k(n)}{n^2}$
loyal embeddings contain both $w$ and $w'$.
\end{corollary}

\begin{proof}
Fix $\phi\in\Phi$.  Every $\phi$-loyal embedding $\vartheta$ containing both
$w$ and $w'$ determines a unique ordered pair of vertices
$(x,x')\in V(H)^2$ such that
$\vartheta(x)=w$ and $\vartheta(x')=w'$.  Since
$j\ne j'$ and the sets $W_j$ are pairwise disjoint, we have
$w\ne w'$, and hence $x\ne x'$.

Fix such an ordered pair $(x,x')$.  We claim that $w\notin W_{\phi(x)}$ or  $w'\notin W_{\phi(x')}.$

We prove by contradiction.  Since
$w\in W_j\cap W_{\phi(x)}$ and the sets $W_\ell$ are pairwise
disjoint, we must have $j=\phi(x)$.  Similarly,
$w'\in W_{j'}\cap W_{\phi(x')}$ implies
$j'=\phi(x')$.

Because $\vartheta$ is an induced embedding and
$\vartheta(x)=w$, $\vartheta(x')=w'$, we have
$a_\Gamma(w,w')=a_H(x,x').$
Because $\phi\in\Aut(H)$, we have 
$a_H(x,x')=a_H(\phi(x),\phi(x')).$
Substituting $\phi(x)=j$ and $\phi(x')=j'$ gives
$a_\Gamma(w,w')=a_H(j,j')$, contradicting the assumption
$a_\Gamma(w,w')\ne a_H(j,j')$.  This proves the claim.

If $w\notin W_{\phi(x)}$, then
Lemma~\ref{lem:fixed-automorphism-defect} gives at most
$k^{-8}\E_k(n)/n^2$ such embeddings with
$\vartheta(x)=w$ and $\vartheta(x')=w'$.  If instead
$w'\notin W_{\phi(x')}$, we apply the same lemma with
$x,x'$ and $w,w'$ interchanged, obtaining the same bound.
There are fewer than $k^2$ ordered pairs $(x,x')$, so for the fixed
automorphism $\phi$ there are at most
$k^{-6}\E_k(n)/n^2$ embeddings under consideration.  Finally,
summing over $|\Phi|\le k^2$ gives at most
$k^{-4}\E_k(n)/n^2$ loyal embeddings containing both $w$ and $w'$.
\end{proof}

\subsection{Few disloyal embeddings}
\label{subsec:few-disloyal}

We now show that embeddings which do not approximately follow any
automorphism in $\Phi$ are rare.  

Fix distinct
$x,x'\in V(H)$ and vertices $z,z'\in V(\Gamma)$.  Throughout this
subsection, all embedding counts are subject to
$\vartheta(x)=z$ and $\vartheta(x')=z'$.  If $z=z'$, there are no
such embeddings, so we assume that $z\ne z'$.

The main tool is based on
\cite[Lemma~7.4]{FSW}, which says that every embedding is either
loyal or fails to follow the candidate classes in one of four
specific ways.  We first verify that this decomposition applies in
our setting.

We apply Lemma~\ref{lem:embedding-decomposition}\cite[Lemma 7.4]{FSW} with
$F=H$, $G=\Gamma$, $m=k$, $\rho=q$,
$\eta_2=\epsilon_2$, $\eta_3=\epsilon_3$,
$\eta_4=\epsilon_4$, $T=S$, $C_i=V_i$,
$C_{\mathrm{rest}}=V_{\mathrm{rest}}$, and $\Psi=\Phi$.
The inequalities
$\epsilon_3\le\epsilon_2/2$ and
$\epsilon_4\le(q/20)\epsilon_3$ follow from
Lemma~\ref{lem:parameters}\textnormal{(P1), (P4)}.
We next verify the two structural assumptions in Lemma~\ref{lem:embedding-decomposition}.  Since $\epsilon_3<q$, every set of size at least
$(1-\epsilon_3/2)k$ has size at least $(1-q/2)k$.
Lemma~\ref{lem:small-super-signature} therefore verifies
condition~\textnormal{(D1)}.  For condition~\textnormal{(D2)}, let
$Y\subseteq V(H)$ have $|Y|\ge(1-\epsilon_2)k$, and let
$f:Y\to V(H)$ be an induced embedding.  Since
$\epsilon_2<\delta$ by
Lemma~\ref{lem:parameters}\textnormal{(P3)}, we have
$|Y|\ge(1-\delta)k$, and
Definition~\ref{def:reasonable}\textnormal{(c)} shows that
$f=\phi|_Y$ for some $\phi\in\Phi$.

Thus every disloyal embedding falls into at least one of the following
four cases.  In the list below,
$\vartheta$ is \emph{$f$-aligned} if
$\vartheta(v)\in V_{f(v)}$ for every vertex in the domain of $f$ and
the image of $\vartheta$ meets each $V_i$ in at most one vertex.  A
pair $\{u,v\}$ is \emph{$f$-inconsistent} if
$a_H(f(u),f(v))\ne a_H(u,v)$.  These are the specializations of
Definitions~\ref{def:aligned} and~\ref{def:consistent}; the term
\emph{fancy} is understood in the sense of
Definition~\ref{def:fancy} with $\rho=q$ and
$\eta_3=\epsilon_3$.

\begin{enumerate}[label=\textnormal{(\Roman*)},leftmargin=3em]
\item At least $\epsilon_3k/2$ vertices of $H$ are mapped into
$V_{\mathrm{rest}}$.

\item The image of $\vartheta$ contains two vertices from the same
cleaned candidate class $V_i$, for some
$i\in V(H)\setminus S$.

\item There is a set $A\subseteq V(H)$ and an injective map
$f:A\to V(H)\setminus S$ such that $\vartheta$ is $f$-aligned and
some vertex of $A$ belongs to more than $\epsilon_4k$
$f$-inconsistent pairs.

\item There are a set $B\subseteq V(H)$ with
$|B|\ge(1-\epsilon_3)k$ and a fancy injective map
$f:B\to V(H)\setminus S$ such that $\vartheta$ is $f$-aligned and
$B$ contains at least $\epsilon_2k/4$ mutually vertex-disjoint
$f$-inconsistent pairs.
\end{enumerate}

Lemma 
\ref{lem:embedding-decomposition} \cite[Lemma 7.4]{FSW} has a simple proof. Even though we omit the proof we provide its main proof idea here.  If
alternatives~\textnormal{(I)} and~\textnormal{(II)} fail, then almost
every vertex of $H$ is mapped into a distinct candidate class, and
the class labels define an injective map on a large subset of
$V(H)$.  If alternative~\textnormal{(III)} also fails, every vertex
has few inconsistent partners.  A small super-signature then allows
us to delete the few vertices having too many inconsistencies with
the signature, producing the fancy map in
alternative~\textnormal{(IV)}.  Finally, if
alternative~\textnormal{(IV)} fails, deleting the endpoints of a
maximal matching of inconsistent pairs leaves a large set on which
the map preserves both adjacency and nonadjacency.  Partial rigidity
then identifies this map with an automorphism, so the original
embedding is loyal.

We bound these four cases in turn.  The arguments inherited from
\cite{FSW} are stated in
Appendix~\ref{app:disloyal-embeddings}.  The two cases for which the
estimate of \cite{FSW} is not strong enough are proved here using
Lemma~\ref{lem:weighted-hitting}.

\alternativeheading{I}{many vertices in the exceptional set}

\begin{lemma}
\label{lem:rest-heavy}
At most $k^{-7}\E_k(n)/n^2$ embeddings have
alternative~\textnormal{(I)} of
Lemma~\ref{lem:embedding-decomposition}.
\end{lemma}

\begin{proof}
Write $H_1:=H\setminus\{x,x'\}$.  An embedding with
alternative~\textnormal{(I)} maps at least
$\epsilon_3k/2-2$ vertices of $H_1$ into
$V_{\mathrm{rest}}$. By Lemma~\ref{lem:parameters}\textnormal{(P5)}, we have
$\epsilon_3k\to\infty$ as $k\to\infty$.  Thus, for all sufficiently
large $k$, $\epsilon_3k\ge12$, and consequently $\frac{\epsilon_3k}{2}-2
\ge \frac{\epsilon_3k}{3}.$ 

Set $\gamma_1:=(24/q)(\log k)^2/k$ and
$\beta_1:=\epsilon_3/3$.  Lemma~\ref{lem:cleaning-package}
\textnormal{(iii)} gives $|V_{\mathrm{rest}}|\le\gamma_1n$.
Moreover, $|V(H_1)|=k-2\ge k-4$ and
$0<\beta_1\le q/3$ by
Lemma~\ref{lem:parameters}\textnormal{(P1)}.

Apply Lemma~\ref{lem:weighted-hitting} to $H_1$ with
$G=\Gamma$, $U=V_{\mathrm{rest}}$, $\gamma=\gamma_1$, and
$\beta=\beta_1$.  In the notation of that lemma, its auxiliary
signature has size $t\le(5/q)\log k$.  By
Lemma~\ref{lem:parameters}\textnormal{(P6)},
$\beta_1k\ge(10/q)\log k+2\ge2t+2$ and
$2e^4\gamma_1/\beta_1<1$.  Thus the simplified bound
\eqref{eq:weighted-hitting-simple} applies, and
\eqref{eq:first-tail} gives at most
$k^{-7}\E_k(n)/n^2$ possible restrictions to $H_1$.  Since the
images of $x$ and $x'$ are fixed, restriction to $H_1$ is injective
on the family being counted.
\end{proof}

\alternativeheading{II}{two vertices use one cleaned candidate class}

This is the collision estimate \cite[Lemma~6.6]{FSW}, recorded in
our notation as Lemma~\ref{lem:candidate-collision}; we do not repeat
its proof.  The main idea is as follows. Suppose that $\vartheta(u)$ and $\vartheta(v)$ lie in the
same cleaned class $V_i$.  By
Definition~\ref{def:reasonable}\textnormal{(b)}, linearly many
vertices of $H$ are adjacent to exactly one of $u$ and $v$.
Inducedness forces their images to be adjacent to exactly one of
$\vartheta(u)$ and $\vartheta(v)$.
There are few such vertices in $\Gamma$.  Outside
$V_i\cup V_{\mathrm{rest}}$, any vertex adjacent to exactly one of
$\vartheta(u)$ and $\vartheta(v)$ is a bad partner of at least one of
them.  Since both $\vartheta(u)$ and $\vartheta(v)$ survived the
cleaning, each has few bad partners.  Thus a fixed large set of
vertices of $H$ is forced into a small subset of $\Gamma$.
Lemma~\ref{lem:candidate-collision} makes this observation
quantitative.

We apply Lemma~\ref{lem:candidate-collision} \cite[Lemma~6.6]{FSW} with
$F=H$, $G=\Gamma$, $m=k$, $N=n$, $\rho=q$,
$\eta_5=\epsilon_5$, $T=S$, $\xi=\psi$,
$C_i'=V_i'$, $C_i=V_i$, and
$C_{\mathrm{rest}}=V_{\mathrm{rest}}$.  Pair separation is
Definition~\ref{def:reasonable}\textnormal{(b)}, and the sets
$V_i'$ are pairwise disjoint by
Lemma~\ref{lem:candidate-disjointness}.  The cleaned classes are
obtained from Definition~\ref{def:bad-pairs} under this
identification.  By Lemma~\ref{lem:cleaning-package}
\textnormal{(iii)} and the definition of $\epsilon_5$,
\[
 |V_i|+|V_{\mathrm{rest}}|
 \le\frac{24}{q}\frac{(\log k)^2}{k}\,n
 <\epsilon_5n
\]
for every $i\in V(H)\setminus S$.  Finally,
$0<\epsilon_5<1$ follows from
Lemma~\ref{lem:parameters}\textnormal{(P1)}, and the remaining
numerical assumptions are exactly
Lemma~\ref{lem:parameters}\textnormal{(P7)}.  Hence at most
\[k^{-7}\E_k(n)/n^2\] embeddings have
alternative~\textnormal{(II)}.

\alternativeheading{III}{one vertex has many local inconsistencies}

\begin{lemma}
\label{lem:local-inconsistency}
Let $x,x'\in V(H)$ be distinct, and let
$z,z'\in V(\Gamma)$. 
At most $k^{-7}\E_k(n)/n^2$ embeddings $\vartheta:H\hookrightarrow\Gamma$ satisfying $\vartheta(x)=z$ and  $\vartheta(x')=z'$ have alternative~\textnormal{(III)} of
Lemma~\ref{lem:embedding-decomposition}.
\end{lemma}

This case leads to the same type of rare event as
alternative~\textnormal{(II)}, but the relevant domain vertices in $H$ are not fixed in advance.  Once $y$ and $\vartheta(y)$ are fixed, every
$f$-inconsistent partner of $y$ must be mapped to a bad partner of
$\vartheta(y)$.  Since $\vartheta(y)$ lies in a cleaned class, it has
fewer than $\epsilon_5n$ bad partners.  We therefore apply
Lemma~\ref{lem:weighted-hitting}, which sums over all possible choices
of the inconsistent partners.

\begin{proof}
Let $A\subseteq V(H)$ and let
$f:A\to V(H)\setminus S$ be an injective map witnessing
alternative~\textnormal{(III)}.  There is a vertex $y\in A$ that
belongs to more than $\epsilon_4k$ $f$-inconsistent pairs
$\{y,v\}$.  For every such $v$, alignment gives
$\vartheta(y)\in V_{f(y)}$ and
$\vartheta(v)\in V_{f(v)}$.  Since $f$ is injective,
$f(y)\ne f(v)$.  Moreover, the fact that it is an induced embedding  and $f$-inconsistent gives
\[
a_\Gamma(\vartheta(y),\vartheta(v))
=
a_H(y,v)
\ne
a_H(f(y),f(v)).
\]
Thus $(\vartheta(y),\vartheta(v))$ is a bad pair in the sense of
Definition~\ref{def:bad-pairs}.

Write $D_y:=\{x,x',y\}$ and $H_y:=H\setminus D_y$. Among the $f$-inconsistent partners of $y$, the set $D_y$ contains at most $x$ and $x'$, since $y$ is not an inconsistent partner of itself. Thus passing from $H$ to $H_y$ removes at most two inconsistent partners of $y$.  
Therefore more than $\epsilon_4k-2$ inconsistent partners remain.
By Lemma~\ref{lem:parameters}\textnormal{(P5)}, we have
$\epsilon_4k\ge4$ for all sufficiently large $k$, and hence $\epsilon_4k-2
\ge \frac{\epsilon_4k}{2}.$
Thus at least $\epsilon_4k/2$ inconsistent partners remain.  We use
this weaker threshold in order to apply 
Lemma \ref{lem:weighted-hitting} with $\beta=\epsilon_4/2$.

Fix $y$ and, when
$y\notin\{x,x'\}$, also fix $\vartheta(y)$.  Let $U_y$ be the set of
bad partners of $\vartheta(y)$.  Since $\vartheta(y)\in V_{f(y)}$ and $V_{f(y)}$ is a cleaned class,
the vertex $\vartheta(y)$ has fewer than $\epsilon_5n$ bad partners.
Thus $|U_y|<\epsilon_5n$.

Set $\gamma_2:=\epsilon_5$ and $\beta_2:=\epsilon_4/2$.  Then
$|V(H_y)|\ge k-3$ and
$0<\beta_2\le q/3$ by
Lemma~\ref{lem:parameters}\textnormal{(P1)}.  Apply
Lemma~\ref{lem:weighted-hitting} to $H_y$ with
$G=\Gamma$, $U=U_y$, $\gamma=\gamma_2$, and $\beta=\beta_2$.
Its auxiliary signature has size $t\le(5/q)\log k$. 
Lemma~\ref{lem:parameters}\textnormal{(P6)} gives
$\beta_2k\ge(10/q)\log k+2\ge2t+2$ and
$2e^4\gamma_2/\beta_2<1$.  Hence the simplified bound applies, and
\eqref{eq:second-tail} gives at most
$k^{-8}\E_k(n)/n^{|D_y|}$ restrictions for each fixed value of
$\vartheta(y)$.

If $y\notin\{x,x'\}$, then $|D_y|=3$, summing over its at most $n$ possible images of $y$ 
gives $k^{-8}\E_k(n)/n^2$ embeddings.  If
$y\in\{x,x'\}$, then $|D_y|=2$ and the image of $y$ is already prescribed, so the denominator $n^2$ is already present.  Summing
over the at most $k$ choices of $y$ proves the lemma.
\end{proof}

\alternativeheading{IV}{a large matching of inconsistencies}

We use the standard entropy--energy argument of
\cite[Lemmas~7.8--7.9]{FSW}, recorded in
Appendix~\ref{app:disloyal-embeddings} as
Lemmas~\ref{lem:number-fancy-maps} and
\ref{lem:many-inconsistent-aligned}, and combined in
Lemma~\ref{lem:fancy-inconsistency}.  We omit the inherited proofs
and recall only the mechanism.  Fix a fancy
map $f$.  Alternative~\textnormal{(IV)} gives many vertex-disjoint
pairs $\{u,v\}$ for which the adjacency of $u$ and $v$ in $H$ is
opposite to the adjacency prescribed by the labels $f(u)$ and
$f(v)$.  For all but a small number of exceptional class pairs, the
majority adjacency between $V_{f(u)}$ and $V_{f(v)}$ agrees with the
latter prescription.  An $f$-aligned embedding must therefore choose,
for each of many disjoint pairs, two images from the at most half of
$V_{f(u)}\times V_{f(v)}$ having the required opposite adjacency.
This gives an exponential saving.  On the other hand, the values of a
fancy map on a logarithmic-size super-signature strongly restrict its
values elsewhere, so there are comparatively few fancy maps.  The
exponential saving for each fixed map dominates the number of
possible maps.

We apply Lemma~\ref{lem:fancy-inconsistency} with
$F=H$, $G=\Gamma$, $m=k$, $N=n$, $\rho=q$,
$\eta_2=\epsilon_2$, $\eta_3=\epsilon_3$, $T=S$,
$C_i=V_i$, $a=x$, $a'=x'$, $b=z$, and $b'=z'$.
The classes $V_i$ are nonempty by
Lemma~\ref{lem:cleaning-package}\textnormal{(ii)} and pairwise
disjoint because $V_i\subseteq V_i'$.  The assumptions give $n\ge k$, and
$\epsilon_3\le q$ follows from
Lemma~\ref{lem:parameters}\textnormal{(P1)}.
Moreover, by
Lemma~\ref{lem:parameters}\textnormal{(P5)}, for all sufficiently
large $k$, $6+2\log k
\le
\left(\frac{\log2}{8}-\frac1{20}\right)\epsilon_2k.$ Consequently,
$2^{-\epsilon_2k/8}e^6k^2
\le\exp(-\epsilon_2k/20)$.  The inequalities
$\epsilon_3\le10^{-20}$,
$(67/q)(\log k)^2\le\epsilon_3k$,
$2\le\epsilon_2k/16$, and
$(20/q)(\log k)^2\le\epsilon_2k/16$ follow from
Lemma~\ref{lem:parameters}\textnormal{(P1), (P5)}.  Every set of
size at least $(1-q)k$ is a signature by
Lemma~\ref{lem:large-signature}, and
Lemma~\ref{lem:cleaning-package}\textnormal{(iv)} gives the required
bound on vertex-disjoint majority-adjacency disagreements.  Finally,
$4\epsilon_3\log(1/\epsilon_3)<\epsilon_2/25$ follows from
Lemma~\ref{lem:parameters}\textnormal{(P2)}, while
$\exp(-\epsilon_2k/100)\le k^{-7}$ follows from
part~\textnormal{(P5)}.  Thus at most
\[k^{-7}\E_k(n)/n^2\] embeddings have
alternative~\textnormal{(IV)}.

\subsection{Summary of the local stability argument}
The first summary is that there are only few disloyal embeddings. 
\begin{proposition}
\label{prop:few-disloyal}
For distinct $x,x'\in V(H)$ and arbitrary $z,z'\in V(\Gamma)$,
there are at most $k^{-6}\E_k(n)/n^2$ disloyal embeddings
$\vartheta:H\hookrightarrow\Gamma$ satisfying
$\vartheta(x)=z$ and $\vartheta(x')=z'$.
\end{proposition}

\begin{proof}
Every disloyal embedding has at least one of alternatives
\textnormal{(I)}--\textnormal{(IV)}.  Lemmas
\ref{lem:rest-heavy}, \ref{lem:candidate-collision},
\ref{lem:local-inconsistency}, and
\ref{lem:fancy-inconsistency} bound the four alternatives,
respectively, by $k^{-7}\E_k(n)/n^2$.  Their union therefore has
size at most $4k^{-7}\E_k(n)/n^2\le k^{-6}\E_k(n)/n^2$.
\end{proof}

We now combine Proposition~\ref{prop:few-disloyal} with the
fixed-automorphism defect estimates to show that the sets $W_j$ cover
the entire host graph.

\begin{lemma}\label{lem:W-cover}
The sets $W_j$, $j\in V(H)$, cover $V(\Gamma)$.
\end{lemma}

\begin{proof}
Suppose that $w\notin W_j$ for every $j\in V(H)$.
By Corollary~\ref{cor:uncovered-loyal}, the vertex $w$ belongs to at most $k^{-5}\frac{\E_k(n)}{n}$ loyal embeddings.
Fix $x\in V(H)$ and choose $x'\ne x$.  Summing
Proposition~\ref{prop:few-disloyal} over the at most $n$ possible
values of $\vartheta(x')$ shows that at most $k^{-6}\frac{\E_k(n)}{n}$ disloyal embeddings satisfy $\vartheta(x)=w$.  Summing over the at most $k$ possible preimages $x$ shows that $w$ belongs to at most $k^{-5}\frac{\E_k(n)}{n}$ disloyal embeddings.
Thus $w$ belongs to fewer than
$2k^{-5}\frac{\E_k(n)}{n} < \frac{\E_k(n)}{n}$ embeddings, contradicting
Lemma~\ref{lem:extremal-vertex-coverage} and
\eqref{eq:abstract-lower-bound}.
\end{proof}

The local stability step is now complete.  The next proposition
records the four conclusions used in the global correction argument.
Properties~\textnormal{(S1)} and~\textnormal{(S4)} control the copies
that may be lost when interpart adjacencies are corrected;
property~\textnormal{(S2)} gives the prospective top-level partition;
and property~\textnormal{(S3)} shows that almost every loyal copy is
genuinely top-level.

\begin{proposition}
\label{prop:structural-package}
For all sufficiently large $k=k(q,c_0)$, the preceding construction
has the following properties.
\begin{enumerate}[label=\textnormal{(S\arabic*)},leftmargin=3em]
\item For distinct $x,x'\in V(H)$ and fixed
$z,z'\in V(\Gamma)$, at most $k^{-6}\frac{\E_k(n)}{n^2}$ disloyal embeddings satisfy $\vartheta(x)=z$ and $\vartheta(x')=z'$.

\item The sets $(W_j)_{j\in V(H)}$ form a partition of
$V(\Gamma)$.

\item For each fixed $\phi\in\Phi$, at most 
$k^{-7}\E_k(n)$ embeddings are $\phi$-loyal but not $\phi$-regulated.

\item If $j\ne j'$, $w\in W_j$, $w'\in W_{j'}$, and $a_\Gamma(w,w')\ne a_H(j,j'),$
then at most $k^{-4}\frac{\E_k(n)}{n^2}$ loyal embeddings contain both $w$ and $w'$.
\end{enumerate}
\end{proposition}

\begin{proof}
Property \textnormal{(S1)} is
Proposition~\ref{prop:few-disloyal}.
Property \textnormal{(S2)} follows from
Lemmas~\ref{lem:W-disjoint} and \ref{lem:W-cover}.
Property \textnormal{(S3)} is
Corollary~\ref{cor:loyal-unregulated}.
Property \textnormal{(S4)} is
Corollary~\ref{cor:wrong-pair-loyal}.
\end{proof}

\section{From the approximate partition to an exact balanced blow-up}
The preceding section gives a partition
$(W_j)_{j\in V(H)}$ of the extremal graph $\Gamma$ and shows that
almost every embedding approximately respects this partition.  Two
tasks remain.  We must first prove that the adjacency between every
two different parts agrees with $H$, and then prove that the part
sizes differ by at most one.

The argument has three steps.  We first find one automorphism
supporting a positive proportion of all top-level embeddings.  We
then correct every incorrect interpart adjacency and show that the
copies created by this correction outnumber those destroyed.  Once
the interpart structure is exact, the top-level contribution factors
as a single product, and a direct transfer argument forces the parts
to be balanced.

We assume $n\ge k$, since the conclusion is part of the definition of
a balanced iterated blow-up when $n<k$.  Write
$g:=|\Phi|$.

\subsection{A dominant automorphism}
In this subsection we show there is some automorphism which regulates many embeddings. (Lemma \ref{lem:significant-automorphism}). 
The calculation in this subsection is essentially the same as in
\cite[Lemma~4.10 and Corollary~4.11, Claim~4.12, 4.13, Lemma 4.14]{FSW} with some constants and parameters changed to cater to Paley graphs, and some arguments simplified because here we do not remove vertices from $H$.  For completeness, we include the short proofs.

 Fix distinct $x,x'\in V(H)$.  Summing (S1) in Proposition~\ref{prop:structural-package} over the $n^2$ choices of $z,z'\in V(\Gamma)$ shows that the total number of disloyal embeddings is at most $k^{-6}\E_k(n)$. 
For a fixed $\phi\in\Phi$, (S3) bounds the number of $\phi$-loyal but unregulated embeddings by $k^{-7}\E_k(n)$.  Since $g\le k^2$, the total number of unregulated embeddings is at most
\begin{equation}\label{eq:unregulated-bound}
 \bigl(k^{-6}+gk^{-7}\bigr)\E_k(n)
 \le(k^{-6}+k^{-5})\E_k(n)<\E_k(n).
\end{equation}
Together with \eqref{eq:abstract-lower-bound}, the total number of regulated embeddings is at least 
\begin{equation}\label{eq:many-regulated}
(g-1)\E_k(n).
\end{equation}

For $\phi\in\Phi$, let $N_\phi$ be the number of $\phi$-regulated embeddings.  Since $\phi$ is a permutation of $V(H)$ and the $W_j$ form a partition,
\begin{equation}\label{eq:Nphi-product-bound}
 N_\phi\le\prod_{x\in V(H)}|W_{\phi(x)}|
 =\prod_{j\in V(H)}|W_j|
 \le\E_k(n).
\end{equation}
Moreover, an embedding cannot be regulated with respect to two
different automorphisms: if $\phi(x)\ne\phi'(x)$ for some $x$, then
its image of $x$ would have to lie in two disjoint sets
$W_{\phi(x)}$ and $W_{\phi'(x)}$.  Thus the regulated embeddings are
partitioned by the sets counted by the numbers $N_\phi$.

\begin{lemma}
\label{lem:significant-automorphism}
There is an automorphism $\phi_*\in\Phi$ such that
\begin{equation}
\label{eq:dominant-automorphism}
N_{\phi_*}
\ge
\left(1-\frac1g\right)\E_k(n)
\ge
\frac12\E_k(n).
\end{equation}
\end{lemma}

\begin{proof}
By \eqref{eq:many-regulated},
$\sum_{\phi\in\Phi}N_\phi\ge(g-1)\E_k(n)$.
The result follows by averaging over the $g$ automorphisms in
$\Phi$ and the fact $|\Phi|\geq 2$.
\end{proof}
This lemma replaces
\cite[Claim~4.16]{FSW}. For our purpose we only need one such $\phi_*$, whereas \cite{FSW} needs many such $\phi_*$ because \cite{FSW} works with certain subgraphs of $H$.

Write $n_j:=|W_j|$.  Combining
\eqref{eq:dominant-automorphism} and
\eqref{eq:Nphi-product-bound}, we obtain
\begin{equation}
\label{eq:product-lower-bound}
\prod_{j\in V(H)}n_j
\ge
N_{\phi_*}
\ge
\frac12\E_k(n).
\end{equation}
In particular, every part is nonempty.

The next lemma is the full-graph analogue of
\cite[Lemma~4.14]{FSW}.  Its proof simplifies because every
automorphism uses every label.
\begin{lemma}\label{lem:part-size-bounds}
For every $j\in V(H)$,
\begin{equation}\label{eq:part-size-bounds}
 1\le |W_j|\le k^{-4/5}n.
\end{equation}
\end{lemma}

\begin{proof}
If $W_j=\varnothing$, then $N_\phi=0$ for every $\phi\in\Phi$, since
each $\phi$ is a permutation of all $k$ labels.  This contradicts
\eqref{eq:many-regulated}.

Suppose now that $|W_j|>k^{-4/5}n$.  Fix $\phi\in\Phi$ and let
$x=\phi^{-1}(j)$.  Write  $m':=n-|W_j|$ and $\mu:=\frac{|W_j|}{n}>k^{-4/5}.$
The image of $x$ has at most $|W_j|\le n$ choices, while the other
$k-1$ images lie in pairwise disjoint subsets of the remaining
$m'$ vertices.  Hence
\[
N_\phi\le |W_j|\E_{k-1}(m')\le n\E_{k-1}(m').
\]
Applying
Lemma~\ref{lem:balanced-product-tools}\textnormal{(iii)} with $\ell=k, 
\ell'=k-1, m=n, m'=n-|W_j|$
gives $\E_{k-1}(m') \le e^{3-\mu k/2}\frac{k}{n}\E_k(n).$
Therefore
\begin{equation}\label{eq:large-part-product-bound}
N_\phi
\le
e^3k e^{-\mu k/2}\E_k(n)
<
e^3k\exp(-k^{1/5}/2)\E_k(n).
\end{equation}
Summing over $\phi\in\Phi$, the total number of regulated embeddings is less than  $g e^3k\exp(-k^{1/5}/2)\E_k(n).$ For sufficiently large $k$, the factor
$e^3k\exp(-k^{1/5}/2)$ is smaller than
$(g-1)/g$, which is at least $1/2$.  This contradicts
\eqref{eq:many-regulated}.
\end{proof}

Finally, \eqref{eq:unregulated-bound} and
\eqref{eq:Nphi-product-bound} give the preliminary estimate
\begin{equation}
\label{eq:F-upper-preliminary}
\emb(H,n)
\le
(g+1)\E_k(n)
\le
(k^2+1)\left(\frac nk\right)^k
\le
2k^{-k+2}n^k.
\end{equation}

\subsection{Correction of the interpart adjacencies}\label{subsec:correction}

Let $\Gamma^*$ be obtained from $\Gamma$ by leaving every graph
$\Gamma[W_j]$ unchanged and correcting the adjacency between
different parts: for $w\in W_j$ and $w'\in W_{j'}$, with
$j\ne j'$, set
\[
a_{\Gamma^*}(w,w'):=a_H(j,j').
\]
Let $M$ be the number of unordered pairs whose adjacency is changed.

We compare the embeddings destroyed by the correction with the new
embeddings it creates.  Extremality will then imply that no adjacency
can have been changed.

The correction argument follows the strategy of
\cite[Lemmas~4.17--4.18 and Corollary~4.19]{FSW}.  The loss estimate is unchanged.  On the gain side, our full-graph
setting gives a simpler and quantitatively stronger statement: one
dominant automorphism works simultaneously for every pair of parts.

We first bound the number of embeddings destroyed by this
correction.
\begin{lemma}\cite[Lemma~4.17]{FSW}
\label{lem:lost-under-correction}
At most
$Mk^{-3}\E_k(n)/n^2$
embeddings of $H$ in $\Gamma$ stop to be embeddings in $\Gamma^*$.
\end{lemma}

\begin{proof}
Fix a changed pair $\{w,w'\}$, where
$w\in W_j$, $w'\in W_{j'}$, and $j\ne j'$.  Then
$a_\Gamma(w,w')\ne a_H(j,j')$.

For each ordered pair of distinct vertices $u,v\in V(H)$,
Proposition~\ref{prop:structural-package}\textnormal{(S1)} bounds the
disloyal embeddings satisfying
$\vartheta(u)=w$ and $\vartheta(v)=w'$ by
$k^{-6}\E_k(n)/n^2$.  Summing over fewer than $k^2$ ordered pairs
gives at most $k^{-4}\E_k(n)/n^2$ disloyal embeddings containing
$\{w,w'\}$.  Part~\textnormal{(S4)} gives the same bound for loyal
embeddings.  Thus at most
$2k^{-4}\E_k(n)/n^2\le k^{-3}\E_k(n)/n^2$
embeddings are charged to one changed pair.  Summing over the $M$
changed pairs proves the lemma.
\end{proof}

The gain side is the next lemma. 
The next lemma plays the role of \cite[Lemma~4.18]{FSW}, but the
full-graph setting gives both a simpler proof and a stronger bound.
In \cite{FSW}, one must find a significant symmetry whose image
contains two prescribed labels.  Here every automorphism uses every
label, and the dominant automorphism found above works for every pair
of parts.

\begin{lemma}
\label{lem:every-pair-many-embeddings}
Let $j\ne j'$, $w\in W_j$, and $w'\in W_{j'}$.  Then at least $\frac12\frac{\E_k(n)}{n^2}$ embeddings $H\hookrightarrow\Gamma^*$ contain both $w$ and $w'$.
\end{lemma}

\begin{proof} Let $\phi_*$ be the dominant automorphism guaranteed by
Lemma~\ref{lem:significant-automorphism}.
Let $u:=\phi_*^{-1}(j)$ and
$v:=\phi_*^{-1}(j')$.  Every map satisfying
$\vartheta(y)\in W_{\phi_*(y)}$ for all $y\in V(H)$ is an induced
embedding in $\Gamma^*$, because the corrected adjacency between
every two different parts agrees with $H$.

After fixing $\vartheta(u)=w$ and
$\vartheta(v)=w'$, the number of such maps is
\[
\prod_{y\in V(H)\setminus\{u,v\}}n_{\phi_*(y)}
=
\frac{\prod_{\ell\in V(H)}n_\ell}{n_jn_{j'}}
\ge
\frac{N_{\phi_*}}{n_jn_{j'}}
\ge
\frac12\frac{\E_k(n)}{n^2},
\]
where we used \eqref{eq:dominant-automorphism} and
$n_j,n_{j'}\le n$.
\end{proof}

\begin{lemma}\label{lem:Gammastar}
    $\Gamma = \Gamma^*$. In particular,  $\Gamma$ has the vertex partition $W_1 \cup \dots W_k$ where for each $1 \leq i < j \leq k$, it is complete between $W_i, W_j$ if $(i,j) \in E(H)$, and it is empty between $W_i, W_j$ if $(i, j) \notin E(H)$. 
\end{lemma}
\begin{proof}
The proof essentially follows the idea in \cite[Corollary 4.19]{FSW}. 
Every changed pair belongs to at least
$\frac12\E_k(n)/n^2$ embeddings of $\Gamma^*$ supplied by
Lemma~\ref{lem:every-pair-many-embeddings}.  None of these maps is an
embedding in $\Gamma$, because it contains a pair whose adjacency
was changed.

Counting incidences between changed pairs and new embeddings, and
using that an embedding contains at most $\binom{k}{2}<k^2/2$
unordered pairs, shows that the correction creates at least $M k^{-2}\E_k(n)/n^2$ new embeddings.  Lemma~\ref{lem:lost-under-correction} shows that it
destroys at most
$Mk^{-3}\E_k(n)/n^2$ embeddings.  If $M>0$, then the number of new
embeddings is strictly larger than the number destroyed, contrary to
the extremality of $\Gamma$.  Hence $M=0$.
\end{proof}

\subsection{The objective for a full-graph blow-up}
By Lemma~\ref{lem:embedding-classification}, every embedding either lies entirely in one part, or it uses all top-level parts
according to an automorphism. Every automorphism contributes the same
product, since it merely permutes the $k$ labels.

Once all interpart adjacencies are correct, the remaining problem is in some sense 
no longer structural but numerical. 

Let
\[
F(m):=\emb(H,m).
\]
Then
\begin{equation}
\label{eq:full-blowup-objective}
\emb(H,\Gamma)
=
g\prod_{j\in V(H)}n_j
+
\sum_{j\in V(H)}\emb(H,\Gamma[W_j]).
\end{equation}
Each graph $\Gamma[W_j]$ must itself be extremal on \(n_j\) vertices;
otherwise replacing it by a better graph of the same order would
increase the second term without changing the first.  Consequently,
\begin{equation}
\label{eq:extremal-recurrence}
F(n)
=
g\prod_{j\in V(H)}n_j
+
\sum_{j\in V(H)}F(n_j).
\end{equation}
The identity \eqref{eq:extremal-recurrence} is the point at which the
full-pattern problem becomes simpler than the deleted-pattern
problem treated in \cite{FSW}: the entire top-level contribution is one product.

\subsection{Balancing the parts}
\label{subsec:balancing}

We now prove that the top-level parts are balanced. 

We need one standard estimate for the discrete derivative of the
extremal function.  Its proof uses the Pippenger--Golumbic
inequalities and is deferred to Appendix~\ref{app:global-closing-estimates}.

\begin{lemma}
\label{lem:extremal-difference-estimate}
For integers $A,B$ with $A\ge B+2$ and $B\ge1$,
\begin{equation}
\label{eq:telescoped-difference}
\begin{aligned}
\bigl(F(A)-F(A-1)\bigr) - \bigl(F(B+1)-F(B)\bigr) 
\le 4(A-B)k^{-k+4}A^{k-2}.
\end{aligned}
\end{equation}
\end{lemma}

\begin{lemma}
\label{lem:balanced-top-level}
The part sizes $n_j$, $j\in V(H)$, differ by at most one.
\end{lemma}

\begin{proof}
Suppose otherwise.  Relabel the parts so that
$n_1$ is largest, $n_2$ is smallest, and
$d:=n_1-n_2\ge2$.  We make a local adjustment. Construct an exact blow-up $\Gamma'$ of $H$ by moving one vertex from
the largest part to the smallest part.  Thus its part sizes are
$n_1-1, n_2+1, n_3,\ldots,n_k,$
and place an extremal graph of the appropriate order inside each
part.

The top-level term in \eqref{eq:extremal-recurrence} increases by $g\bigl((n_1-1)(n_2+1)-n_1n_2\bigr)
 \prod_{j\notin\{1,2\}}n_j
= g(d-1)\prod_{j\notin\{1,2\}}n_j.$
Since $d\ge2$, this is at least
$\frac{gd}{2}\prod_{j\notin\{1,2\}}n_j.$

Now $n_2\le n/k$, while
Lemma~\ref{lem:part-size-bounds} gives
$n_1\le k^{-4/5}n$.  Together with
\eqref{eq:product-lower-bound}, this yields
\[
\prod_{j\notin\{1,2\}}n_j
=
\frac{\prod_jn_j}{n_1n_2}
\ge
\frac12 k^{9/5}\frac{\E_k(n)}{n^2}.
\]
Thus the gain in the top-level term is at least
\begin{equation}
\label{eq:gain-lower-bound}
\frac14gd\,k^{9/5}\frac{\E_k(n)}{n^2}.
\end{equation}

The possible loss in the two internal terms is
\[
L_{\mathrm{int}}
:=
F(n_1)-F(n_1-1)-F(n_2+1)+F(n_2).
\]
By Lemma~\ref{lem:extremal-difference-estimate} and
Lemma~\ref{lem:part-size-bounds},
\begin{equation}
\label{eq:internal-loss}
L_{\mathrm{int}}
\le
4d\,k^{-k+4}n_1^{k-2}
\le
4d\,k^{-k+4}(k^{-4/5}n)^{k-2}.
\end{equation}

Since $n\ge k$, every factor in a balanced $k$-fold product is at
least $n/(2k)$, and hence
$\E_k(n)\ge(n/(2k))^k$.  Comparing
\eqref{eq:internal-loss} with
\eqref{eq:gain-lower-bound}, the ratio of the possible loss to the
guaranteed gain is at most $\frac{16}{g}\,2^k k^{19/5-4k/5},$ which tends to zero as $k\to\infty$.  For all sufficiently large
$k$, the gain is therefore larger than the loss.  By
\eqref{eq:extremal-recurrence}, the graph $\Gamma'$ then has more
embeddings than $\Gamma$, a contradiction.
\end{proof}

We have proved that $\Gamma$ is a balanced blow-up of $H$.  Each
graph $\Gamma[W_j]$ is extremal on $n_j<n$ vertices, since all
$k\ge2$ parts are nonempty.  Induction on $n$ shows that every
$\Gamma[W_j]$ is a balanced iterated blow-up of $H$.  This proves
Proposition~\ref{prop:full-fsw}.

\addtocontents{toc}{\protect\addvspace{0.5em}}
\appendix

\section{Verification of the parameter choice}
\label{app:parameter-choice}

\begin{proof}[Proof of Lemma~\ref{lem:parameters}]
We verify the assertions in order.

First, clearly
$ \frac{\epsilon_3}{\epsilon_2}
 =\frac{1}{10^4\log k},
 \frac{\epsilon_4}{\epsilon_3}
 =\frac q{100},
 \frac{\epsilon_5}{\epsilon_4}
 =\frac{100}{qa_4}\frac{(\log k)^3}{\sqrt k}.$
The first and third ratios tend to zero, while the second is less
than one.  Thus
$\epsilon_5<\epsilon_4<\epsilon_3\le\epsilon_2/2$
for all sufficiently large $k$.  Since $\epsilon_2\to0$ and
$\epsilon_1=q/3$, we also have
$\epsilon_2<\epsilon_1<q/2$.  This proves
\textnormal{(P1)}, including the stated upper bounds by $1/100$.

For \textnormal{(P2)}, we have
\[
 \epsilon_2\log(1/\epsilon_2)
 =
 \frac{a_2}{\sqrt k}
 \left(\frac12\log k+\log(1/a_2)\right)
 \to 0.
\]
Since $(\log2)\epsilon_1/8$ is a fixed positive constant, this gives
the first inequality.  Moreover,
\[
 \frac{\epsilon_3\log(1/\epsilon_3)}{\epsilon_2}
 =
 \frac{a_3}{a_2\log k}
 \left(
   \frac12\log k+\log\log k+\log(1/a_3)
 \right)
 \to
 \frac{1}{2\cdot10^4},
\]
which is smaller than $1/100$.  This proves the second inequality.

Since $a_2\le c/4$, we have
$\epsilon_2\le c/(4\sqrt k)<\delta$ whenever
$\delta\ge c/\sqrt k$, proving \textnormal{(P3)}.
Also,
$\epsilon_4/\epsilon_3=q/100<q/20$, proving
\textnormal{(P4)}.

For \textnormal{(P5)}, the definitions give
$ \epsilon_2k=a_2\sqrt k,
 \epsilon_3k=\frac{a_3\sqrt k}{\log k}, 
 \epsilon_4k=\frac{a_4\sqrt k}{\log k}.$
Each of these quantities dominates $(\log k)^2$ as $k\to\infty$.
Also,
$\epsilon_5k=(100/q)(\log k)^2$, and
$\epsilon_5<q/10^4$ for all sufficiently large $k$. Also, $\frac{\epsilon_1k}{(\log k)^2} = \frac q3\frac{k}{(\log k)^2}
\to\infty.$
The remaining inequalities in \textnormal{(P5)} now follow directly.

For \textnormal{(P6)}, direct calculation gives
$ \frac{2e^4\gamma_1}{\beta_1}=
 \frac{144e^4}{qa_3}
 \frac{(\log k)^3}{\sqrt k}, 
 \frac{\beta_1k}{2} =
 \frac{a_3\sqrt k}{6\log k},
 \frac{2e^4\gamma_2}{\beta_2} =
 \frac{400e^4}{qa_4}
 \frac{(\log k)^3}{\sqrt k},
 \frac{\beta_2k}{2} =
 \frac{a_4\sqrt k}{4\log k}.$
Thus both bases tend to zero, while both exponents have order
$\sqrt k/\log k$.

We use the following elementary estimate.  For fixed positive
constants $B,b,C,M$,
\[
 \exp\bigl(C(\log k)^2\bigr)
 \left(
   B\frac{(\log k)^3}{\sqrt k}
 \right)^{b\sqrt k/\log k}
 \le k^{-M}
\]
for all sufficiently large $k$.  Indeed,
$ \log\left(
 B\frac{(\log k)^3}{\sqrt k}
 \right)
 \le-\frac14\log k$ for all sufficiently large $k$.  The logarithm of the left-hand side
is therefore at most
$C(\log k)^2-b\sqrt k/4$, which is at most
$-M\log k$ for sufficiently large $k$.
Applying this estimate to the two pairs above proves
\eqref{eq:first-tail} and \eqref{eq:second-tail}.

The same formulas show that each $\beta_i k$ is a positive constant,
depending only on $q,c$, times $\sqrt k/\log k$.  Since
$\sqrt k/(\log k)^2\to\infty$, we have
$\beta_i k\ge(10/q)\log k+2$ for $i=1,2$ and all sufficiently
large $k$.  This completes the proof of \textnormal{(P6)}.

Finally, the first and third inequalities in
\textnormal{(P7)} hold for all sufficiently large $k$, and the
second follows from \textnormal{(P5)}.  Moreover, $\frac{9e^4\epsilon_5}{q} = \frac{900e^4(\log k)^2}{q^2k}
 \le k^{-1/2}$ for all sufficiently large $k$.  Hence $ \left(\frac{9e^4\epsilon_5}{q}\right)^{qk/3}
 \le
 \exp\left(-\frac{qk\log k}{6}\right)
 \le k^{-9}.$
This proves \textnormal{(P7)}.
\end{proof}

\section{Technical details for the local stability argument}
\label{app:inherited-local-stability}
This appendix records, in the notation of the present paper, the
deterministic lemmas from \cite{FSW} used in the local stability
argument.  We omit the proofs of results quoted directly from
\cite{FSW}; brief explanations of their main ideas are given at the
corresponding points in the main text.  When a statement is
reformulated as a full-graph specialization, we indicate why the
proof in \cite{FSW} continues to apply.  All hypotheses needed in
the present setting are verified where the result is used.

All notation in this appendix is local to the statement in which it
appears.  We use $F$ and $G$ for the pattern and host graphs, $m$ and
$N$ for their orders, $\rho$ for the separation parameter,
$\eta_i$ for the error parameters, $T$ for a distinguished set,
$C_i'$ and $C_i$ for candidate and cleaned classes, and $\Psi$ for a
symmetry family.  Signatures, super-signatures, and distinguishing
sets are understood with $F$ in place of the standing graph $H$; in
particular, $\Delta_F(u,v)$ denotes the vertices of $F$ adjacent to
exactly one of $u,v$.  Each application in the main text explicitly
identifies these local variables with the standing objects.

\subsection{Cleaning lemma}
\label{app:cleaning-package}

The next statement is essentially
\cite[Lemmas~6.1--6.4]{FSW}.

\begin{lemma}\cite[Lemmas~6.1--6.4]{FSW}
\label{lem:deterministic-cleaning}
Let $F$ be a graph on $m$ vertices, let $0<\rho<1$, and let
$T\subseteq V(F)$ be a signature of size
$t\le(5/\rho)\log m$.  Let $G$ be an $N$-vertex graph with
$N\ge m$, and fix a map $\xi:T\to V(G)$.  For each
$i\in V(F)\setminus T$, let $C_i'$ be the set of possible images of
$i$ among embeddings $F\hookrightarrow G$ extending $\xi$.
Assume the following.
\begin{enumerate}[label=\textnormal{(C\arabic*)},leftmargin=3em]
\item The sets $C_i'$, $i\in V(F)\setminus T$, are pairwise disjoint.
\item At least $2\E_m(N)/N^t$ embeddings extend $\xi$.
\item There is $0<\eta_5<1$ such that
$(3+\log m)(t+1)\le(12/\rho)(\log m)^2\le\eta_5m/2$ and
$(3+\log m)t/\log2\le(20/\rho)(\log m)^2$.
\end{enumerate}
Apply Definition~\ref{def:bad-pairs} to $F,G,T,(C_i')$ with
threshold $\eta_5$, and denote the resulting cleaned classes and
exceptional set by $C_i$ and $C_{\mathrm{rest}}$.  Then:
\begin{enumerate}[label=\textnormal{(\roman*)},leftmargin=2.5em]
\item If $i\in V(F)\setminus T$ and $u\in C_i'\setminus C_i$, then
at most $\E_m(N)/N^{t+1}$ embeddings extend $\xi$ and map $i$ to
$u$.
\item At least $\E_m(N)/N^t$ embeddings extending $\xi$ map every
$i\in V(F)\setminus T$ into $C_i$.  In particular, every $C_i$ is
nonempty and
\[
 \prod_{i\in V(F)\setminus T}|C_i|\ge\frac{\E_m(N)}{N^t}.
\]
\item For every $i\in V(F)\setminus T$,
$|C_i|+|C_{\mathrm{rest}}|\le(24/\rho)(\log m)^2N/m$.
\item If $i_1,\ldots,i_r,j_1,\ldots,j_r$ are distinct elements of
$V(F)\setminus T$ and
$m_G(C_{i_\ell},C_{j_\ell})\ne a_F(i_\ell,j_\ell)$ for every
$1\le\ell\le r$, then $r\le(20/\rho)(\log m)^2$.
\end{enumerate}
\end{lemma}

\begin{proof}
This is by
\cite[Lemmas~6.1--6.4]{FSW}.  Part~\textnormal{(i)} uses only that $T\cup\{i\}$ is a signature,
Lemma~\ref{lem:fixed-signature-extension},
Lemma~\ref{lem:balanced-product-tools}\textnormal{(iii)}, and
condition~\textnormal{(C3)}.
Part~\textnormal{(ii)} uses condition~\textnormal{(C2)},
part~\textnormal{(i)}, and the disjointness in
condition~\textnormal{(C1)}.  The proofs of
parts~\textnormal{(iii)} and~\textnormal{(iv)} then use only
part~\textnormal{(ii)}, the disjointness of the cleaned classes,
majority adjacency, Lemma \ref{lem:balanced-product-tools}.  Thus the
proofs in \cite{FSW} apply under
\textnormal{(C1)}--\textnormal{(C3)}.
\end{proof}

\subsection{A deterministic fixed-symmetry estimate}
\label{app:fixed-automorphism-defect}

The following is the full-graph form of \cite[Lemma~5.2]{FSW}.  In
the cited proof, the fixed rotation or reflection is used only as an
adjacency-preserving bijection and may therefore be replaced by an
arbitrary fixed automorphism.

\begin{lemma}\cite[Lemma~5.2]{FSW}
\label{lem:deterministic-fixed-defect}
Fix $0<\rho<1$. 
Let $F$ be a graph on $m$ vertices satisfying
$|\Delta_F(u,v)|\ge\rho m$ for all distinct $u,v\in V(F)$.
Let $T\subseteq V(F)$ have size $t$, let $G$ be an $N$-vertex graph
with $N\ge m$, and let $(C_i)_{i\in V(F)\setminus T}$ be nonempty
pairwise disjoint subsets of $V(G)$.  Fix
$0<\eta_1,\eta_2<1$.  For $j\in V(F)$, let $D_j$ be the set of
vertices $w\in V(G)$ for which at least $(1-\eta_1)m$ indices
$i\in V(F)\setminus(T\cup\{j\})$ satisfy
$w\notin C_i$ and $m_G(w,C_i)=a_F(j,i)$.

Let $\sigma\in\Aut(F)$.  Loyalty is understood relative to
$T,(C_i)$, and error $\eta_2$ in the sense of
Definition~\ref{def:relative-loyalty}.  Suppose that
$\eta_2\le\min\{\rho/2,e^{-4}\}$, $\eta_2m\ge16$, and
$2\eta_2\log(1/\eta_2)<(\log2)\eta_1/4$.  Suppose also that
\begin{equation}
\label{eq:fixed-defect-intersection-assumption}
 (1-\eta_2)m-2+\eta_1m-t-2-m\ge\frac{\eta_1m}{2}
\end{equation}
and $e^6m^2\exp(- (\log2)\eta_1m/4)\le m^{-8}$.
If $a,a'\in V(F)$ are distinct and $b,b'\in V(G)$ satisfy
$b\notin D_{\sigma(a)}$, then at most
$m^{-8}\E_m(N)/N^2$ $\sigma$-loyal embeddings
$\theta:F\hookrightarrow G$ satisfy
$\theta(a)=b$ and $\theta(a')=b'$.
\end{lemma}

\begin{proof}
The proof is the argument of \cite[Lemma~5.2]{FSW}, with the fixed
rotation or reflection replaced by $\sigma$.  The proof uses the
symmetry only as an adjacency-preserving bijection.

The failure of $b\in D_{\sigma(a)}$ gives at least
$\eta_1m-t-2$ candidate classes on which $b$ has the wrong majority
adjacency.  Loyalty supplies at least $(1-\eta_2)m-2$ coordinates
following $\sigma$, and
\eqref{eq:fixed-defect-intersection-assumption} leaves at least
$\eta_1m/2$ coordinates in the intersection.  Each such coordinate
contributes a factor at most $1/2$.  The assumptions on $\eta_2$
control the number of possible loyal coordinate sets and ensure that
the fixed coordinates form a signature.  Lemma  \ref{lem:balanced-product-tools} then give the stated bound.
\end{proof}

\subsection{Disloyal embeddings}
\label{app:disloyal-embeddings}
Throughout this subsection, let $F$ be a graph on $m$ vertices, let
$G$ be a host graph, let $T\subseteq V(F)$, and let
$\mathcal C=(C_i)_{i\in V(F)\setminus T}$ be a family of candidate
classes in $G$.  Additional assumptions on these objects will be
stated in the individual lemmas.

This subsection records the definitions and inherited estimates used
in Proposition~\ref{prop:few-disloyal}.  The robustification and
decomposition statements are the full-graph forms of
\cite[Lemmas~7.4--7.5]{FSW}; the collision estimate is the full-graph
form of \cite[Lemma~6.6]{FSW}; and the final entropy--energy estimate
comes from \cite[Lemmas~7.8--7.9]{FSW}.

\begin{definition}[Aligned map]
\label{def:aligned}
Let $A\subseteq V(F)$ and let $h:A\to V(F)\setminus T$ be injective.
An embedding $\theta:F\hookrightarrow G$ is \emph{$h$-aligned with
respect to $\mathcal C=(C_i)_{i\in V(F)\setminus T}$} if
$\theta(v)\in C_{h(v)}$ for every $v\in A$ and its image contains at
most one vertex from each $C_i$.
\end{definition}

\begin{definition}[Consistent map]
\label{def:consistent}
Let $A\subseteq V(F)$ and let $h:A\to V(F)\setminus T$ be injective.
A pair $\{u,v\}\subseteq A$ is \emph{$h$-consistent} if
$a_F(h(u),h(v))=a_F(u,v)$ and is \emph{$h$-inconsistent} otherwise.
Given $\eta_4>0$, the map $h$ is \emph{locally mostly consistent} if
every $u\in A$ belongs to at most $\eta_4m$ $h$-inconsistent pairs.
\end{definition}

\begin{definition}[Fancy maps]
\label{def:fancy}
Fix $\rho,\eta_3>0$.  Let $B\subseteq V(F)$ satisfy
$|B|\ge(1-\eta_3)m$, and let
$h:B\to V(F)\setminus T$ be injective.  We call $h$ \emph{fancy} if
there is a $(\rho/4)$-super-signature $R\subseteq B$ of $F$, with
$|R|\le(33/\rho)\log m$, such that every $v\in B\setminus R$
belongs to at most $(\rho/10)|R|$ $h$-inconsistent pairs
$\{v,r\}$ with $r\in R$.
\end{definition}

The following lemma will be used in the proof of
Lemma~\ref{lem:embedding-decomposition}.  After
alternatives~\textnormal{(I)}--\textnormal{(III)} have failed, it
converts the resulting locally mostly consistent candidate-class map
into a fancy map on a large subset.  It isolates the deterministic
pruning argument of \cite[Lemma~7.5]{FSW}.

\begin{lemma}\cite[Lemma~7.5]{FSW}
\label{lem:robustification}
Let $F$ be a graph on $m$ vertices, let $T\subseteq V(F)$, and let
$0<\rho,\eta_3,\eta_4<1$ satisfy
\[
\eta_4\le\frac{\rho\eta_3}{20}.
\]
Let $A\subseteq V(F)$ satisfy
$|A|\ge(1-\eta_3/2)m$, and let
$h:A\to V(F)\setminus T$ be injective and locally mostly consistent
with parameter $\eta_4$.  Suppose that $A$ contains a
$(\rho/4)$-super-signature $R$ of $F$ with
$|R|\le(33/\rho)\log m$.  Then there is a set $B\subseteq A$ with
$|B|\ge(1-\eta_3)m$ such that $h|_B$ is fancy.
\end{lemma}

We omit the proof, which is the pruning argument of
\cite[Lemma~7.5]{FSW}. The main idea is: local consistency bounds the
number of $h$-inconsistent pairs joining $R$ to $A\setminus R$ by
$\eta_4m|R|$.  Consequently, at most
$\frac{\eta_4m|R|}{(\rho/10)|R|}
= \frac{10\eta_4}{\rho}m
\le
\frac{\eta_3m}{2}$ vertices of $A\setminus R$ have more than
$(\rho/10)|R|$ inconsistent partners in $R$.  Deleting these vertices
leaves a set $B$ of the required size, still containing $R$, and
$h|_B$ is fancy.

The following is essentially \cite[Lemma~7.4]{FSW}.  The
symmetry family enters only in the final extension step, isolated as
condition~\textnormal{(D2)}.

\begin{lemma}\cite[Lemma~7.4]{FSW}
\label{lem:embedding-decomposition}
Let $F$ be a graph on $m$ vertices, let $G$ be a graph, let
$T\subseteq V(F)$, and let
$\mathcal C=(C_i)_{i\in V(F)\setminus T}$ be a family of pairwise
disjoint subsets of $V(G)$.  Let
$C_{\mathrm{rest}}:=V(G)\setminus\bigcup_{i\in V(F)\setminus T}C_i$.
Let $\Psi\subseteq\Aut(F)$, and let
$0<\rho,\eta_2,\eta_3,\eta_4<1$ satisfy
$\eta_3\le\eta_2/2$ and $\eta_4\le(\rho/20)\eta_3$.
Assume the following.
\begin{enumerate}[label=\textnormal{(D\arabic*)},leftmargin=3em]
\item Every $X\subseteq V(F)$ with
$|X|\ge(1-\eta_3/2)m$ contains a $(\rho/4)$-super-signature of
size at most $(33/\rho)\log m$.
\item Whenever $Y\subseteq V(F)$ has
$|Y|\ge(1-\eta_2)m$ and $h:Y\to V(F)$ is injective with
$a_F(h(u),h(v))=a_F(u,v)$ for all distinct $u,v\in Y$, there is
$\sigma\in\Psi$ such that $h=\sigma|_Y$.
\end{enumerate}
Loyalty is understood relative to $T,\mathcal C,\eta_2$, and $\Psi$
in the sense of Definition~\ref{def:relative-loyalty}.  Then every
embedding $\theta:F\hookrightarrow G$ has at least one of the
following five properties.
\begin{enumerate}[label=\textnormal{(\Roman*)},leftmargin=3em]
\item At least $\eta_3m/2$ vertices are mapped into
$C_{\mathrm{rest}}$.
\item The image of $\theta$ contains two vertices from the same class
$C_i$.
\item There is an $A\subseteq V(F)$ and an injective map
$h:A\to V(F)\setminus T$ such that $\theta$ is $h$-aligned and $h$
is not locally mostly consistent with parameter $\eta_4$.
\item There is a $B\subseteq V(F)$ with
$|B|\ge(1-\eta_3)m$ and a fancy injective map
$h:B\to V(F)\setminus T$ such that $\theta$ is $h$-aligned and $B$
contains at least $\eta_2m/4$ mutually vertex-disjoint
$h$-inconsistent pairs.
\item The embedding $\theta$ is loyal.
\end{enumerate}
\end{lemma}
The proof is the argument of \cite[Lemma~7.4]{FSW}.  If
\textnormal{(I)} and~\textnormal{(II)} fail, the candidate-class
labels of the vertices outside $C_{\mathrm{rest}}$ define an
injective map $h_1$ on a set $A$ of size at least
$(1-\eta_3/2)m$, and $\theta$ is $h_1$-aligned.  If
\textnormal{(III)} also fails, condition~\textnormal{(D1)} and
Lemma~\ref{lem:robustification} give a set
$B\subseteq A$ with $|B|\ge(1-\eta_3)m$ on which
$h_1$ restricts to a fancy map $h_2$.
If \textnormal{(IV)} fails, a maximal collection of mutually
vertex-disjoint $h_2$-inconsistent pairs has fewer than
$\eta_2m/4$ elements.  Deleting their endpoints leaves a set $Y$
with
\[
|Y|
\ge
(1-\eta_3)m-\frac{\eta_2m}{2}
\ge
(1-\eta_2)m,
\]
where the last inequality uses $\eta_3\le\eta_2/2$.  By maximality,
$h_2|_Y$ preserves both adjacency and nonadjacency.  Condition
\textnormal{(D2)} therefore gives $\sigma\in\Psi$ with
$h_2|_Y=\sigma|_Y$.  Since \textnormal{(II)} fails, the image of
$\theta$ meets each $C_i$ in at most one vertex, and for every
$y\in Y$ we have
$\sigma(y)\notin T$ and $\theta(y)\in C_{\sigma(y)}$.  Thus
$\theta$ is $\sigma$-loyal, giving \textnormal{(V)}.

\

The following is 
\cite[Lemma~6.6]{FSW}, with all standing hypotheses used in its
proof stated explicitly. It will be used to bound Alternative (II). 

\begin{lemma}[Alternative (II) bound]\cite[Lemma~6.6]{FSW}
\label{lem:candidate-collision}
Fix $0<\rho<1/2$.  For all sufficiently large $m=m(\rho)$, let
$F$ be a graph on $m$ vertices satisfying
$|\Delta_F(u,v)|\ge\rho m$ for all distinct $u,v\in V(F)$, and let
$G$ be an $N$-vertex graph with $N\ge m$.  Let
$T\subseteq V(F)$ be a signature and fix $\xi:T\to V(G)$.  For each
$i\in V(F)\setminus T$, let $C_i'$ be the set of possible images of
$i$ among embeddings $F\hookrightarrow G$ extending $\xi$, and
assume that the classes $C_i'$ are pairwise disjoint.  Fix
$0<\eta_5<1$, and apply Definition~\ref{def:bad-pairs} to
$F,G,T,(C_i')$ with threshold $\eta_5$, obtaining cleaned classes
$C_i$ and an exceptional set $C_{\mathrm{rest}}$.

Suppose that $|C_i|+|C_{\mathrm{rest}}|\le\eta_5N$ for every
$i\in V(F)\setminus T$.  Suppose also that
$ \rho m\ge3, 3\eta_5<\frac\rho3,
 \frac{\rho m}{3}\ge\frac{20}{\rho}(\log m)^2, 
 \left(\frac{9e^4\eta_5}{\rho}\right)^{\rho m/3}\le m^{-9}.$

Then, for any distinct $a,a'\in V(F)$ and any $b,b'\in V(G)$, at
most $m^{-7}\E_m(N)/N^2$ embeddings $\theta:F\hookrightarrow G$
satisfy $\theta(a)=b$, $\theta(a')=b'$, and
$|\theta(V(F))\cap C_i|\ge2$ for some $i\in V(F)\setminus T$.
\end{lemma}
In the statement above, the map $\xi$ is used only to
construct the candidate classes $C_i'$ and their cleaned subsets
$C_i$.  The embeddings counted by the estimate are not required to
extend $\xi$, nor are their remaining vertices required to map into
correspondingly indexed classes.

\subsubsection{The entropy--energy estimate for alternative
\textnormal{(IV)}}

The first estimate bounds the number of structured part maps.

\begin{lemma}\cite[Lemma~7.8]{FSW}
\label{lem:number-fancy-maps}
Let $m\ge2$, let $F$ be an $m$-vertex graph, let
$T\subseteq V(F)$, and let
$0<\rho<1$ and $0<\eta_3\le10^{-20}$.  Assume $\frac{67}{\rho}(\log m)^2\le\eta_3m.$
  Then the number of pairs $(B,h)$
such that $B\subseteq V(F)$, $|B|\ge(1-\eta_3)m$, and
$h:B\to V(F)\setminus T$ is a fancy injective map is at most
$\exp(4\eta_3m\log(1/\eta_3))$.
\end{lemma}

This is the  counting argument of
\cite[Lemma~7.8]{FSW}, with the standing assumptions used there
stated explicitly.    Its main idea is as follows. 
There are few
choices for the large domain $B$ and for the small super-signature
$R\subseteq B$.  Once $B$, $R$, and $h|_R$ are fixed, the
super-signature property and the consistency of $h$ with $R$ imply
that different vertices of $B\setminus R$ have disjoint sets of
possible images.  A balanced-product estimate therefore bounds the
number of extensions of $h|_R$.  Summing over the choices of $B$,
$R$, and $h|_R$ gives the stated estimate.

The second estimate bounds the
aligned embeddings with many inconsistencies. 

\begin{lemma}\cite[Lemma~7.9]{FSW}
\label{lem:many-inconsistent-aligned}
Let $F$ be an $m$-vertex graph, let $G$ be an $N$-vertex graph, let
$T\subseteq V(F)$, and let
$\mathcal C=(C_i)_{i\in V(F)\setminus T}$ be a family of nonempty
pairwise disjoint subsets of $V(G)$.  Fix
$0<\rho,\eta_2,\eta_3<1$.  
Assume
$\eta_2m\ge32$
and 
$\frac{20}{\rho}(\log m)^2\le\frac{\eta_2m}{16}.$
Assume also that $N\ge m$, that $\eta_3\le\rho$, and that
$2^{-\eta_2m/8}e^6m^2 \le \exp(-\eta_2m/20).$
 Assume also that every subset
of $V(F)$ of size at least $(1-\rho)m$ is a signature, and that
for every $r\ge1$ and every choice of $2r$ distinct
indices
\[
i_1,\ldots,i_r,j_1,\ldots,j_r\in V(F)\setminus T
\]
such that 
$m_G(C_{i_\ell},C_{j_\ell})\ne a_F(i_\ell,j_\ell)$ for every
$\ell$, one has $r\le(20/\rho)(\log m)^2$.

Let $B\subseteq V(F)$ have $|B|\ge(1-\eta_3)m$, and let
$h:B\to V(F)\setminus T$ be injective.  Suppose that $B$ contains
at least $\eta_2m/4$ mutually vertex-disjoint $h$-inconsistent
pairs.  Then, for distinct $a,a'\in V(F)$ and arbitrary
$b,b'\in V(G)$, the number of $h$-aligned embeddings
$\theta:F\hookrightarrow G$ satisfying
$\theta(a)=b$ and $\theta(a')=b'$ is at most
\[
 \exp(-\eta_2m/20)\frac{\E_m(N)}{N^2}.
\]
\end{lemma}

This is a special case of \cite[Lemma~7.9]{FSW}. We omit the proof and recall its main idea.  
Alternative~\textnormal{(IV)} supplies at
least $\eta_2m/4$ mutually vertex-disjoint $h$-inconsistent pairs. Removing the two prescribed vertices $a,a'$ destroys at most two of these pairs.  Hence at least $\frac{\eta_2m}{4}-2 \ge \frac{3\eta_2m}{16}$ pairs remain, where the inequality uses $\eta_2m\ge32$.
By the assumed bound on
vertex-disjoint majority-adjacency disagreements, at most
$\eta_2m/16$ of these pairs have the wrong majority relation between
their candidate classes.  Thus at least $\eta_2m/8$ pairs remain for
which the majority adjacency between the two candidate classes agrees
with the labels under $h$, and hence is opposite to the adjacency
required by the original pair.
For an $h$-aligned embedding, each such pair must therefore be chosen
from at most half of the corresponding product of candidate classes.
The pairs use mutually distinct classes, so these restrictions are
independent and give a factor $2^{-\eta_2m/8}$.  The vertices already
assigned to candidate classes form a signature, while all remaining
vertices must avoid those classes.  The fixed-signature extension
estimate and the balanced-product inequalities then give the stated
bound.

\begin{lemma}[Bound for alternative \textnormal{(IV)}]
\label{lem:fancy-inconsistency}
Assume that the data
$F,G,T,\mathcal C,\rho,\eta_2,\eta_3$ satisfy the hypotheses of
Lemmas~\ref{lem:number-fancy-maps} and
\ref{lem:many-inconsistent-aligned}.  Fix distinct
$a,a'\in V(F)$ and $b,b'\in V(G)$.  Suppose in addition that
$4\eta_3\log(1/\eta_3)<\eta_2/25$ and
$\exp(-\eta_2m/100)\le m^{-7}$.  Then at most
$m^{-7}\E_m(N)/N^2$ embeddings $\theta:F\hookrightarrow G$
satisfying $\theta(a)=b$ and $\theta(a')=b'$ have
alternative~\textnormal{(IV)} of
Lemma~\ref{lem:embedding-decomposition}.
\end{lemma}

\begin{proof}
Lemma~\ref{lem:number-fancy-maps} gives at most
$\exp(4\eta_3m\log(1/\eta_3))$ possible pairs $(B,h)$, and
Lemma~\ref{lem:many-inconsistent-aligned} gives at most
$\exp(-\eta_2m/20)\E_m(N)/N^2$ aligned embeddings for each one.
Hence the total number is at most
\[
\exp\left(
4\eta_3m\log(1/\eta_3)-\frac{\eta_2m}{20}
\right)\frac{\E_m(N)}{N^2}.
\]
Since
$4\eta_3\log(1/\eta_3)<\eta_2/25$, the exponent is less than
$-\eta_2m/100$.  The assumption
$\exp(-\eta_2m/100)\le m^{-7}$ now gives the result.
\end{proof}

\section{Auxiliary estimates for the global argument}
\label{app:global-closing-estimates}

\subsection{Difference estimates}

\begin{proof}[Proof of
Lemma~\ref{lem:extremal-difference-estimate}]
The Pippenger--Golumbic inequality states that,
for every integer $m\ge2$,
\begin{equation}
\label{eq:PG-differences}
\frac{k}{m-1}F(m-1)
\le
F(m)-F(m-1)
\le
\frac{k}{m}F(m).
\end{equation}
See \cite{PG} and also \cite[Lemma~8.2]{FSW}.

Write $\Delta_m:=F(m)-F(m-1)$.  For $m\ge3$, the upper bound in
\eqref{eq:PG-differences} at $m-1$ and the lower bound at $m$ give
$\Delta_{m-1}\le\Delta_m$.  Hence we have 
$F(m)-F(m-2) = \Delta_m+\Delta_{m-1}
\le 2\Delta_m \le \frac{2k}{m}F(m).$
Also, using the upper bound at $m$ and the lower bound at $m-1$,
$\Delta_m-\Delta_{m-1} \le
\frac{k}{m}F(m) - \frac{k}{m-2}F(m-2) \le \frac{k}{m}\bigl(F(m)-F(m-2)\bigr).$ For $m\ge k$, \eqref{eq:F-upper-preliminary} gives
$F(m)\le2k^{-k+2}m^k$, while for $m<k$ we have $F(m)=0$.
Thus the same upper bound holds for every positive integer $m$.
Finally, if $A-B\ge2$, then
$\Delta_A-\Delta_{B+1} =
\sum_{\ell=B+2}^{A}
(\Delta_\ell-\Delta_{\ell-1}) \le
4(A-B)k^{-k+4}A^{k-2},$ 
which is \eqref{eq:telescoped-difference}.
\end{proof}

\paragraph{AI acknowledgment.}
GPT~5.5 was used to locate the completion theorem of Sz\H{o}nyi~\cite{Szonyi} and Sziklai~\cite{Szi}, stated as
Theorem~\ref{thm:szonyi-sziklai-completion}. The authors also used AI
to polish the wording of the paper.


\begin{thebibliography}{99}
\bibitem{AIM25}
J.~Balogh, D.~Dong, B.~Lidick\'y, and A.~Raymond (organizers),
\emph{Flag algebras and extremal combinatorics: Workshop summary},
American Institute of Mathematics, October 13--17, 2025.
\url{https://aimath.org/pastworkshops/flagextremalrep.pdf}.


\bibitem{blumenthal2021inducibility}
A.~Blumenthal and M.~Phillips,
``Inducibility of the net graph,''
\emph{arXiv preprint arXiv:2103.06350}, 2021.

\bibitem{bodnar2025some}
L.~Bodn{\'a}r and O.~Pikhurko,
``Some exact inducibility-type results for graphs via flag algebras,''
\emph{arXiv preprint arXiv:2507.01596}, 2025.

\bibitem{brown1994inducibility}
J.~I.~Brown and A.~Sidorenko,
``The inducibility of complete bipartite graphs,''
\emph{Journal of Graph Theory},
vol.~18, no.~6, pp.~629--645, 1994.

\bibitem{Carlitz}
L.~Carlitz,
``A theorem on permutations in a finite field,''
\emph{Proc. Amer. Math. Soc.} \textbf{11} (1960), 456--459;
errata, ibid., 999--1000.

\bibitem{CL26}
W.~Chen and X.~Liu,
``Exact extremal constructions for the inducibility of blowup graphs,''
\emph{arXiv preprint arXiv:2606.06202}, 2026.

\bibitem{FST}
Fancsali, S.L., Sziklai, P. \& Tak\'ats, M. The number of directions determined by less than $q$ points. J Algebr Comb 37, 27–37 (2013)

\bibitem{FHL}
J.~Fox, H.~Huang, and C.~Lee,
\emph{A solution to the inducibility problem for almost all graphs},
unpublished manuscript, 2017.

\bibitem{FSW}
J.~Fox, L.~Sauermann, and F.~Wei,
\emph{On the inducibility problem for random Cayley graphs of abelian groups with a few deleted vertices},
Random Structures \& Algorithms \textbf{59} (2021), 554--615.

\bibitem{hatami2014inducibility}
H.~Hatami, J.~Hirst, and S.~Norine,
``The inducibility of blow-up graphs,''
\emph{Journal of Combinatorial Theory, Series B},
vol.~109, pp.~196--212, 2014.

\bibitem{hirst2014inducibility}
J.~Hirst,
``The inducibility of graphs on four vertices,''
\emph{Journal of Graph Theory},
vol.~75, no.~3, pp.~231--243, 2014.

\bibitem{jain2026binomial}
V.~Jain, M.~Michelen, and F.~Wei,
``The binomial random graph is a bad inducer,''
\emph{Random Structures \& Algorithms},
vol.~68, no.~3, Art.~no.~e70067, 2026.

\bibitem{lidicky2023c}
B.~Lidick{\'y}, C.~Mattes, and F.~Pfender,
``$C_5$ is almost a fractalizer,''
\emph{Journal of Graph Theory},
vol.~104, no.~1, pp.~220--244, 2023.

\bibitem{liu2023feasible}
X.~Liu, D.~Mubayi, and C.~Reiher,
``The feasible region of induced graphs,''
\emph{Journal of Combinatorial Theory, Series B},
vol.~158, pp.~105--135, 2023.

\bibitem{PG}
N. Pippenger and M. C. Golumbic, The inducibility of graphs, J. Combin. Theory Ser. B 19 (1975), 189–203.

\bibitem{Redei}
R\'edei, L.: L\"uckenhafte Polynome \"uber endlichen K\"orpern. Lehrb\"ucher und
Monographien aus dem Gebiete der Exakten Wissenschaften. Mathematische
Reihe, Band 42. Birkh\"auser, Basel-Stuttgart (1970) (English Translation: Lacunary
Polynomials Over Finite Fields. North Holland, Amsterdam (1973))

\bibitem{schelp1998remark}
R.~H.~Schelp and A.~Thomason,
``A remark on the number of complete and empty subgraphs,''
\emph{Combinatorics, Probability and Computing},
vol.~7, no.~2, pp.~217--219, 1998.


\bibitem{Somlai}
Somlai, G. A new proof of R\'edei’s theorem on the number of directions. Arch. Math. 122, 575–580 (2024).

\bibitem{Szi}
 Sziklai, P.: On subsets of $GF(q^2)$ with $d$th power differences. Discrete Math. 208/209, 547–555 (1999)

 \bibitem{Szonyi}
T.~Sz\H{o}nyi,
``On the number of directions determined by a set of points in an
affine Galois plane,''
\emph{J. Combin. Theory Ser. A} \textbf{74} (1996), 141--146.




\bibitem{yuster2019exact}
R.~Yuster,
``On the exact maximum induced density of almost all graphs and their inducibility,''
\emph{Journal of Combinatorial Theory, Series B},
vol.~136, pp.~81--109, 2019.

\end{thebibliography}
\end{document}